\RequirePackage{fix-cm}
\documentclass[smallextended]{svjour3}       
\smartqed  
\usepackage{graphicx}
\graphicspath{{fig/}}
\usepackage{a4wide,enumerate,xcolor}
\usepackage{amsmath,graphicx}
\allowdisplaybreaks
\usepackage{bm} 
\usepackage{amssymb}
\usepackage[export]{adjustbox}
\let\pa\partial
\let\na\nabla

\newcommand{\N}{{\mathbb N}}
\newcommand{\R}{{\mathbb R}}

\newcommand{\divM}{{\operatorname{div}}}
\newcommand{\dd}{\mathrm{d}}

\newcommand{\phm}{{\phantom{$-$}}}
\newcommand{\pvec}[1]{\vec{#1}\mkern2mu\vphantom{#1}}

\def\softd{{\leavevmode\setbox1=\hbox{d}%
		\hbox to 1.05\wd1{d\kern-0.4ex{\char039}\hss}}}
\journalname{Journal of Scientific Computing}

\begin{document}

\title{A structure-preserving numerical method for quasi-incompressible \\
Navier--Stokes--Maxwell--Stefan systems\thanks{A.~J.\ and M.~L.-M.\ gratefully acknowledge the Research in Teams Fellowship ``Entropy methods for evolutionary systems: analysis and numerics" of the Erwin Schr\"odinger International Institute for Mathematics and Physics, Vienna. A.~J.\ acknowledges partial support from the Austrian Science Fund (FWF), grants 10.55776/F65 and 10.55776/P33010.
This work has received funding from the European Research Council (ERC) under the European Union's Horizon 2020 research and innovation programme, ERC Advanced Grant NEUROMORPH, no.~101018153. The work of M.~L.-M.\ was supported by the Gutenberg Research College and by the Deutsche Forschungsgemeinschaft (DFG, German Research Foundation) -- project number 233630050 -- TRR 146 and	project number 525800857 -- SPP 2410 ``Hyperbolic Balance Laws: Complexity, Scales and Randomness".
She is also grateful to the Mainz Institute of Multiscale Modelling  for supporting her research.
A. B. was supported by the Gutenberg Research College and by the DFG -- project number 233630050 -- TRR 146 and	project number 441153493 -- SPP 2256 ``Variational Methods for Predicting Complex Phenomena in Engineering Structures and Materials". For open-access purposes, the authors have applied a CC BY public copyright license to any author-accepted manuscript version arising from this submission.} 
}

\titlerunning{A numerical method forquasi-incompressible Maxwell--Stefan--Navier--Stokes systems}        

\author{Aaron Brunk        \and
        Ansgar Jüngel      \and
        M\'aria Luk\'a\v{c}ov\'a-Medvi\softd ov\'a
}


\institute{A. Brunk \at
              Institute of Mathematics, Johannes Gutenberg-University Mainz, Staudinger Weg 9, 55288 Mainz, Germany \\
              \email{abrunk@uni-mainz.de}           
           \and
           A. J\"ungel \at
              Institute of Analysis and Scientific Computing, TU Wien, Wiedner Hauptstra\ss e 8--10, 1040 Wien, Austria \\ 
              Erwin Schr\"odinger International Institute for Mathematics and Physics, Boltzmanngasse 9, 1090 Wien, Austria \\
              \email{juengel@tuwien.ac.at}
          \and
           M. Luk\'a\v{c}ov\'a-Medvi\softd ov\'a \at
              Institute of Mathematics, Johannes Gutenberg-University Mainz, Staudinger Weg 9, 55288 Mainz, Germany \\ 
              Erwin Schr\"odinger International Institute for Mathematics and Physics, Boltzmanngasse 9, 1090 Wien, Austria \\
              \email{lukacova@uni-mainz.de}
}

\date{Received: date / Accepted: date}

\maketitle

\begin{abstract}
A conforming finite element scheme with mixed explicit--implicit time discretization for quasi-incompressible Navier--Stokes--Maxwell--Stefan systems in a bounded domain with periodic boundary conditions is presented. The system consists of the Navier--Stokes equations, together with a quasi-incompressibility constraint, coupled with the cross-diffusion Maxwell--Stefan equations. The numerical scheme preserves the partial masses and the quasi-incompressibility constraint and dissipates the discrete energy. Numerical experiments in two space dimensions illustrate the convergence of the scheme and the structure-preserving properties. 

\end{abstract}
\section{Introduction}

Many applications for multicomponent mixtures, like ion transport in biological channels, the dynamics of electrolytes in lithium-ion batteries, or helium--oxygen--carbon dioxide mixtures in the lung, can be described by Maxwell--Stefan equations. In the Fick--Onsager form, they consist of nonlinear parabolic cross-diffusion equations whose diffusion matrix (called mobility matrix) is positive semidefinite only \cite{BoDr23}. In the hydrostatic and isothermal case, the variables are the partial mass densities that satisfy a volume-filling constraint implying a constant total mass density. In the general case, the equations are coupled to the compressible or incompressible Navier--Stokes equations. In this paper, we consider the quasi-incompressible model of \cite{Dru21} and suggest, for the first time, a provable structure-preserving finite element scheme. 

\subsection{Model equations}

The model describes the mass transport of the liquid or gas components by convection and diffusion (for the partial mass densities $\rho_1,\ldots,\rho_N$) as well as the momentum balance (for the barycentric velocity $\bm{u}$ and the pressure $p$):
\begin{align}
  & \pa_t\rho_i + \divM(\rho_i\bm{u}) + \divM\bm{J}_i = 0,
  \label{1.mass} \\
  & \pa_t(\rho\bm{u}) + \divM\big(\rho\bm{u}\otimes\bm{u}
  - \mathbb{S}(\vec\rho,\na\bm{u}) + p\mathbb{I}\big) = 0\quad\mbox{in }\Omega,
  \ t>0, \label{1.mom} \\
  & \rho_i(\cdot,0) = \rho_i^0, \quad 
  (\rho\bm{u})(\cdot,0) = (\rho\bm{u})^0\quad\mbox{in }\Omega,\ 
  i=1,\ldots,N. \label{1.ic}
\end{align}
Here, $\Omega\subset\R^d$ is a bounded domain (we choose the torus), $\mathbb{I}\in\R^{N\times N}$ is the unit matrix, the total mass density is defined by $\rho:=\sum_{i=1}^N\rho_i$, and $\bm{J}_i$ are the diffusion fluxes satisfying the side condition $\sum_{i=1}^N\bm{J}_i=0$, which implies after summation of \eqref{1.mass} over $i=1,\ldots,N$ the total mass conservation $\pa_t\rho+\divM(\rho\bm{u})=0$. The stress tensor $\mathbb{S}(\vec\rho,\na\bm{u})$ is assumed to satisfy the coercivity condition $\mathbb{S}(\vec\rho,\na\bm{u}):\na\bm{u}\ge 0$, where ``:'' is the Frobenius matrix product (summation over both indices). A typical example is $\mathbb{S}(\vec\rho,\na\bm{u})=\nu(\na\bm{u}^T+\na\bm{u}) + \lambda(\divM\bm{u})\mathbb{I}$, where $\nu>0$ and $\lambda\geq -\tfrac{2}{d}\nu$ may depend on the density vector $\vec\rho:=(\rho_1,\ldots,\rho_N)$, and the coercivity condition follows from Korn's inequality \cite[Sec.~10.9]{FeNo09}.

The partial diffusion fluxes are given as linear combinations of the gradients of the chemical potentials $\mu_i$,
\begin{align}\label{1.flux}
  \bm{J}_i = -\sum_{j=1}^N M_{ij}(\vec\rho)\na\mu_j, \quad
  i=1,\ldots,N.
\end{align}
The mobility (or Onsager) matrix $M=(M_{ij})_{i,j=1}^N \in\R^{N\times N}$ is assumed to be symmetric and positive semi-definite satisfying $\sum_{j=1}^N M_{ij}(\vec\rho)=0$ for all $\vec\rho\in[0,\infty)^N$. This condition is consistent with the side condition of vanishing net diffusion flux. If all Maxwell--Stefan diffusion coefficients are equal, the mobility matrix has the entries
\begin{align}\label{1.M}
  M_{ij}(\vec\rho) = \rho_i\delta_{ij} - \frac{\rho_i\rho_j}{\rho},
  \quad i.j=1,\ldots,N.
\end{align}
We use this expression in our numerical experiments, but the numerical analysis does not rely on this choice. The chemical potentials $\mu_i$ are defined by 
\begin{align}\label{1.chem}
  \mu_i = \frac{\pa f}{\pa\rho_i}(\vec\rho) + V_ip, \quad i=1,\ldots,N,
\end{align}
where $f$ is the internal energy density and $V_i>0$ are the constant specific volumes of the molecules. A typical example for ideal gas mixtures is given by $f(\vec\rho) = \sum_{i=1}^N\rho_i \log(\rho_i/\rho)$, used for the numerical tests. Note that our numerical analysis holds for general energy densities that are positively homogeneous of degree one. The incompressibility constraint is imposed by
\begin{align}\label{1.incom}
   \sum_{i=1}^N V_i\rho_i = 1.
\end{align}
Observe that definition \eqref{1.chem} is consistent with the Gibbs--Duhem relation
\begin{align*}
  -f(\vec\rho) + \sum_{i=1}^N\rho_i\mu_i
  = -\sum_{i=1}^N\rho_i\frac{\pa f}{\pa\rho_i}
  + \sum_{i=1}^N\rho_i\mu_i = \sum_{i=1}^N \rho_i V_i p = p,
\end{align*} 
where the first identity follows from the positive homogeneity of $f$ and the last identity is a consequence of \eqref{1.incom}. 

In the case of equal specific volumes $V_i=V>0$ for all $i=1,\ldots,N$, we recover the volume-filling condition $\sum_{i=1}^N\rho_i=1/V$, which implies from mass conservation \eqref{1.mass} the standard divergence-free condition $\divM\bm{u}=0$. This yields the incompressible Navier--Stokes equations
\begin{align}\label{1.momim}
  V^{-1}\pa_t\bm{u} + \divM\big(V^{-1}\bm{u}\otimes\bm{u} 
  - \mathbb{S}(\vec\rho,\na\bm{u}) + p\mathbb{I}\big) = 0, \quad
  \divM\bm{u}=0,
\end{align}
 that are decoupled from the Maxwell--Stefan equations; see, e.g., \cite{ChJu15,MaTe15}. 

In the general case, we multiply \eqref{1.mass} by $V_i$ and sum over $i=1,\ldots,N$:
\begin{align}\label{1.VJ}
  0 = \pa_t\bigg(\sum_{i=1}^N V_i\rho_i\bigg) 
  + \divM\bigg(\sum_{i=1}^N V_i\rho_i\bm{u}\bigg) 
  + \divM\bigg(\sum_{i=1}^N V_i\bm{J}_i\bigg)
  = \divM\bm{u} + \divM\bigg(\sum_{i=1}^N V_i\bm{J}_i\bigg),
\end{align}
which is interpreted as a quasi-incompressibility constraint. In particular, a net local change of volume is possible because of mixing. Summarizing, we can formulate our system \eqref{1.mass}--\eqref{1.incom} as 
\begin{align}
  & \pa_t\rho_i + \divM(\rho_i\bm{u})
  - \divM\bigg(\sum_{i,j=1}^N M_{ij}(\vec\rho)\na\mu_j\bigg) = 0, 
  \label{1.mass2} \\
  & \pa_t(\rho\bm{u}) + \divM\big(\rho\bm{u}\otimes\bm{u}
  - \mathbb{S}(\vec\rho,\na\bm{u}) + p\mathbb{I}\big) = 0, \label{1.mom2} \\
  & \mu_i = \frac{\pa f}{\pa\rho_i} + V_ip, \quad i=1,\ldots,N,
  \label{1.mu} \\
  & \divM\bm{u} = \divM\bigg(\sum_{i,j=1}^N V_iM_{ij}(\vec\rho)
  \na\mu_j\bigg). \label{1.divMu}
\end{align}

An analogous model can be derived when the quasi-incompressibility constraint \eqref{1.incom} is neglected. In this case, no equation for the divergence \eqref{1.divMu} and no contribution of the pressure $p$ in the chemical potentials appear, cf.\ \cite{Dru21b,Eike25}. Therefore, the system needs to be considered with an equation of state for the pressure.

\subsection{Main results}

The aim of this paper is the design of a structure-preserving finite element method for system \eqref{1.mass2}--\eqref{1.divMu}. In particular, the numerical scheme preserves the partial masses and the quasi-incompressibility constraint \eqref{1.incom}, conserves the total mass, and satisfies an energy inequality. The energy density of the system is given by the sum of the kinetic and internal energy densities,
\begin{align}\label{1.E}
  E(\vec\rho,\bm{u}) = \frac{\rho}{2}|\bm{u}|^2 + f(\vec\rho).
\end{align}
We show in Lemma \ref{lem.energy} that the following energy identity holds for smooth solutions:
\begin{align*}
  \frac{\dd}{\dd t}\int_\Omega E(\vec\rho,\bm{u})\dd x
  + \int_\Omega\bigg(\mathbb{S}(\vec\rho,\na\bm{u}):\na\bm{u}
  + \sum_{i,j=1}^N M_{ij}(\vec\rho)\na\mu_i\cdot\na\mu_j\bigg)\dd x = 0.
\end{align*}

We prove two main results. First, with $(\rho_{h,i}^k,\bm{u}_h^k)$ being the finite element approximation of $(\rho_i,\bm{u})$ at time $t_k=k\tau$ (with time step $\tau>0$), we show in Theorem \ref{thm.id} that the discrete partial masses are preserved and that the discrete energy density $E_h^k=\frac12\rho_h^k|\bm{u}_h^k|^2+f(\vec\rho_h^k)$ (with $\rho_h^k:=\sum_{i=1}^N\rho_{h,i}^k$) satisfies
\begin{align*}
  \frac{1}{\tau}&\int_\Omega(E_h^{k+1}-E_h^k)\dd x \\
  &= -\int_\Omega\bigg(\mathbb{S}(\pvec{\rho}_{h}^{*},\na\bm{u}_h^{k+1}):\na\bm{u}_h^{k+1}
  + \sum_{i,j=1}^N M_{ij}(\pvec{\rho}_{h}^{*})\na\mu_{h,i}^{k+1}\cdot
  \na\mu_{h,j}^{k+1}\bigg)\dd x - \mathcal{D}_{\rm num} \le 0,
\end{align*}
where $\rho_{h,i}^*$ depends on $(\rho_{h,j}^{k},\rho_{h,i}^{k+1})$ and $\mu_{h,i}^{k+1}$ is the discrete chemical potential. The numerical dissipation $\mathcal{D}_{\rm num}$ is nonnegative and depends on the discrete kinetic defect and the Hessian of the internal energy density. 

Second, we prove that the constraint $\sum_{i=1}^N V_i\rho_{h,i}^k=1$ holds pointwise in $\Omega$ and that the discrete total mass density is strictly positive; see Theorem \ref{thm.bd}. 

The discrete energy inequality yields some a priori bounds detailed in Theorem \ref{thm.bd}. However, the existence of a discrete solution to scheme \eqref{1.mass2}--\eqref{1.divMu} is still an open problem. The existence of a global weak solution to the continuous problem in a bounded domain with Neumann boundary conditions was shown in \cite{Dru21}. The pressure turns out to be a measure with singular values at the minimal and maximal values of the total mass density. Thus, it is not clear to what extent the analysis of \cite{Dru21} can be applied to the finite element case. Another approach is the local-in-time existence result of \cite{BoDr21}, which yields an integrable pressure function. The idea of the proof is a reformulation of the problem in terms of so-called relative chemical potentials and the construction of a fixed point mapping. As the arguments are quite involved, it is not known whether the arguments hold in the finite-dimensional case.

\subsection{State of the art} 

The existence of global weak solutions to the incompressible Navier--Stokes--Maxwell--Stefan equations (i.e.\ $V_i=V$ for $i=1,\ldots,N$) was shown in \cite{BoPr17,ChJu15,DoDo21,MaTe15}. Existence results for the quasi-incompressible model (i.e.\ for general $V_i>0$) were presented in \cite{BoDr21,Dru21}. Analytical results for quasi-incompressible fluids, related but different to \eqref{1.incom}, can be found in, e.g., \cite{FLM16,GTP15}. In \cite{GTP15}, a mass conservative and energy dissipative discontinuous Galerkin finite element discretization for quasi-incompressible Navier--Stokes--Cahn--Hilliard systems was presented. In \cite{AFMV25,Van-Brunt22} numerical approximations of coupled Maxwell-–Stefan fluid models using finite elements methods are considered, where steady-state Maxwell-–Stefan equations are coupled with a compressible Stokes flow. Similar steady-state models and discretizations in the quasi-incompressible case are studied in \cite{Baier25}, using an explicit formulation of the forces instead of the fluxes. To include additional effects, like in \cite{VBru23}, the explicit force formulation might be more appropriate. However, structure-preserving numerical methods for non-stationary quasi-incompressible Navier--Stokes--Maxwell--Stefan systems are missing in the literature.

The literature is much richer for the Maxwell--Stefan equations with vanishing barycentric velocity. The existence of weak solutions was proved in \cite{Bot11,GeJu24,JuSt13}, and the analysis for heat-conducting flows was investigated in \cite{HeJu21,HuSa18}. Numerical works include explicit Euler finite difference schemes in one space dimension \cite{BGS12}, iterative solvers for such discretizations \cite{Gei15}, and mixed finite element methods \cite{McBo14}. More recently, numerical approximations focused on structure-preserving schemes. In \cite{HLTW21}, an energy-stable and positivity-preserving IMEX finite difference discretization was proposed, based on an equivalent optimization problem. The work \cite{CEM24} suggested a convergent finite volume method that preserves the non-negativity of the densities, the mass conservation, and the volume-filling constraint. For general cross-diffusion equations, including the Maxwell--Stefan systems, structure-preserving schemes were investigated in \cite{BPS22} (Galerkin approximation) and \cite{JuZu23} (finite volume discretization).  

\medskip
The paper is organized as follows. We derive in Section \ref{sec.2} various formulations of our model and show the energy identity. The numerical scheme and its structure-preserving properties are proved in Section \ref{sec.num}. This section contains our main analytical results. Finally, in Section \ref{sec.exp}, we present some numerical experiments. 


\section{Various formulations}\label{sec.2}

We introduce the following notation. Vectors in $\R^N$ are written as $\vec{v}$, while vectors in $\R^d$ are denoted by $\bm{v}$. Matrices are formulated as $\mathbb{M}$. The $L^2(\Omega)$ inner product is denoted by $\langle\cdot,\cdot\rangle$, and we set $\mathrm{1}=(1,\ldots,1)^T \in\R^N$. We impose the following assumptions:
\begin{itemize}
\item[(A0)] Domain: let $\Omega=(0,1)^d$ be topologically equivalent to the $d$-dimensional torus, $d\ge 1$. This means we enforce periodic boundary conditions.
\item[(A1)] Mobility matrix: $\mathbb{M}(\vec\rho)=(M_{ij}(\vec\rho))_{i,j=1}^N \in(0,\infty)^{N\times N}$ is symmetric positive semi-definite with $\mathrm{1}\in\mathrm{ker}(\mathbb{M})$.
\item[(A2)] Internal energy density: $f:[0,\infty)^N\to\R^N$ is smooth and positively homogeneous of degree one.
\item[(A3)] Viscous stress tensor: $\mathbb{S}:[0,\infty)^N\times\R^{d\times d}\to\R^{d\times d}$ satisfies $\mathbb{S}(\vec\rho,\na\bm{u}):\na\bm{u}\ge c_S|\na\bm{u}|^2$ for all $\na\bm{u}\in\R^{d\times d}$ and for some $c_S>0$.
\item[(A4)] Initial data and specific volumes: $\rho_i^0\ge 0$ and $(\rho\bm{u})^0$ are integrable functions. It holds that $\sum_{i=1}^N V_i\rho_i^0=1$ in $\Omega$, where $V_i>0$ for $i=1,\ldots,N$.
\end{itemize}

We call a function $f$ positively homogeneous of degree one if $f(\lambda\vec\rho)=\lambda f(\vec\rho)$ for all $\lambda>0$ and $\vec\rho\in[0,\infty)^N$. By Euler's homogeneous function theorem, any positively homogeneous function $f$ of degree one satisfies $f(\vec\rho)=\sum_{i=1}^N\rho_i\pa f/\pa\rho_i$. This implies, for smooth functions $\vec\rho=\vec\rho(x)$, that
\begin{align}\label{2.one}
  \sum_{i=1}^N\rho_i\na\frac{\pa f}{\pa\rho_i}
  = \na\sum_{i=1}^N\rho_i\frac{\pa f}{\pa\rho_i}
  - \sum_{i=1}^N\na\rho_i\frac{\pa f}{\pa\rho_i}
  = \na\sum_{i=1}^N\rho_i\frac{\pa f}{\pa\rho_i} 
  - \na f(\vec{\rho}) = 0.
\end{align}
Observe that any positively homogeneous function is convex (but not strictly convex). 

\begin{lemma}[Reformulation]
Equations \eqref{1.mass}--\eqref{1.flux}, \eqref{1.chem}--\eqref{1.incom} can be written as \eqref{1.mass2}--\eqref{1.divMu} for smooth solutions. If $V_i=V>0$ for $i=1,\ldots,N$, equations \eqref{1.mom2} and \eqref{1.divMu} reduce to \eqref{1.momim}. 
\end{lemma}

\begin{proof}
Relation \eqref{1.divMu} follows from the computation in \eqref{1.VJ}, proving the first assertion. If $V_i=V$, it follows from $\sum_{i=1}^NM_{ij}(\vec\rho))=0$, i.e.\ (A1), that $\sum_{i=1}^N\bm{J}_i=0$ and hence $\divM\bm{u}=0$. Since $\rho=V^{-1}$, we obtain equations \eqref{1.momim}.
\end{proof}

\begin{lemma}[Positivity of the mass densities]\label{lem.pos}
Let $f(\vec\rho)=\sum_{i=1}^N\rho_i\log(\rho_i/\rho)$ and let $(\vec\rho,\vec\mu,\bm{u},$ $p)$ be a smooth solution to \eqref{1.mass}--\eqref{1.flux}, \eqref{1.chem}--\eqref{1.incom}. Then $\rho_i(t)>0$ in $\Omega$, $t>0$, $i=1,\ldots,N$.
\end{lemma}

\begin{proof}
We deduce from \eqref{1.mu} and the explicit expression of $f$ that $\rho_i/\rho=\exp(\mu_i-V_ip)$. Then, multiplying this expression by $V_i$, summing over $i=1,\ldots,N$, and using the constraint \eqref{1.incom}, we obtain $1/\rho = \sum_{i=1}^N V_i\rho_i/\rho = \sum_{i=1}^N V_i\exp(\mu_i-V_ip)>0$. This implies that $\rho>0$ and consequently, $\rho_i= \rho\exp(\mu_i-V_ip)>0$.
\end{proof}

\begin{lemma}[Energy identity]\label{lem.energy}
It holds for smooth solutions to \eqref{1.mass}--\eqref{1.flux} that
\begin{align*}
  \frac{\dd}{\dd t}\int_\Omega E(\vec\rho,\bm{u})\dd x
  + \int_\Omega\bigg(\mathbb{S}(\vec\rho,\na\bm{u}):\na\bm{u}
  + \sum_{i,j=1}^N M_{ij}(\vec\rho)\na\mu_i\cdot\na\mu_j\bigg)\dd x = 0,
\end{align*}
where the energy density $E$ is defined in \eqref{1.E}.
\end{lemma}

\begin{proof}
We compute:
\begin{align*}
  \frac{\dd}{\dd t}\int_\Omega E(\vec\rho,\bm{u})\dd x
  = \int_\Omega\bigg(\frac12\pa_t\rho|\bm{u}|^2
  + \rho\bm{u}\cdot\pa_t\bm{u}
  + \sum_{i=1}^N\frac{\pa f}{\pa\rho_i}\pa_t\rho_i\bigg)\dd x.
\end{align*}
It follows from $\sum_{i=1}^N \bm{J}_i=0$ that $\pa_t\rho=-\divM(\rho\bm{u})$ and
\begin{align*}
  \rho\pa_t\bm{u} = \pa_t(\rho\bm{u}) - \bm{u}\pa_t\rho
  = -\divM(\rho\bm{u}\otimes\bm{u} - \mathbb{S}(\vec\rho,\na\bm{u}) + p\mathbb{I})
  + \bm{u}\divM(\rho\bm{u}).
\end{align*}
Inserting the mass and momentum balance equations and integrating by parts, this shows that
\begin{align*}
  \frac{\dd}{\dd t}\int_\Omega E(\vec\rho,\bm{u})\dd x
  &= \int_\Omega\bigg(-\frac12|\bm{u}|^2\divM(\rho\bm{u})
  - \divM(\rho\bm{u}\otimes\bm{u} - \mathbb{S}(\vec\rho,\na\bm{u}) + p\mathbb{I})\bm{u}  
  + |\bm{u}|^2\divM(\rho\bm{u}) \\
  &\phantom{xx}- \sum_{i=1}^N\frac{\pa f}{\pa\rho_i}
  \divM(\rho_i\bm{u}+\bm{J}_i)\bigg)\dd x \nonumber \\
  &= \int_\Omega\bigg(\rho\bm{u}\cdot\na\bm{u}\cdot\bm{u}
  + (\rho\bm{u}\otimes\bm{u} - \mathbb{S}(\vec\rho,\na\bm{u}) + p\mathbb{I}):\na\bm{u}
  - 2\rho\bm{u}\cdot\na\bm{u}\cdot\bm{u} 
  \nonumber \\
  &\phantom{xx}+ \sum_{i=1}^N(\rho_i\bm{u}+\bm{J}_i)
  \cdot\na\frac{\pa f}{\pa\rho_i}\bigg)\dd x \nonumber \\
  &= \int_\Omega\bigg(-\mathbb{S}(\vec\rho,\na\bm{u}):\na\bm{u} + p\divM\bm{u}
  + \sum_{i=1}^N\rho_i\na\frac{\pa f}{\pa\rho_i}\cdot\bm{u}
  + \sum_{i=1}^n\bm{J}_i\cdot\na\frac{\pa f}{\pa\rho_i}\bigg)\dd x.
  \nonumber 
\end{align*}
It follows from \eqref{2.one} that the third term on the right-hand side vanishes. Inserting the definition of $\bm{J}_i$ and observing that $\pa f/\pa\rho_i=\mu_i-pV_i$, we find that
\begin{align*}
  \frac{\dd}{\dd t}\int_\Omega E(\vec\rho,\bm u)\dd x
  = \int_\Omega\bigg(-\mathbb{S}(\vec\rho,\na\bm{u}):\na\bm{u} + p\divM\bm{u}
  - \sum_{i,j=1}^N M_{ij}\na\mu_j\cdot\na\mu_i
  - \sum_{i=1}^N V_i\bm{J}_i\cdot\na p\bigg)\dd x,
\end{align*}
and by \eqref{1.divMu}, the second and fourth terms on the right-hand side cancel. \qed
\end{proof}

For the following result, we introduce the skew-symmetric bilinear form
\begin{align*}
  (\bm{v},\bm{w})\mapsto \bm{b}_{\rm skw}(\rho\bm{u};\bm{v},\bm{w}) 
  = \frac12\big(\langle(\rho\bm{u}\cdot\na)\bm{v},\bm{w}\rangle
  - \langle(\rho\bm{u}\cdot\na)\bm{w},\bm{v}\rangle\big).
\end{align*}

To motivate our structure-preserving numerical method, we first derive a suitable weak formulation of \eqref{1.mass}--\eqref{1.flux}, \eqref{1.chem}--\eqref{1.incom}, such that the relevant structures as conservation of partial masses and the energy dissipation can be achieved by using suitable test functions. For instance, to derive the contribution of the kinetic energy from the standard conservative variables, we need to test the momentum equation with $\bm{u}$ and the total mass equation with $\frac{1}{2}|\bm{u}|^2$. The approximate velocity is typically contained in an approximation space, while $\frac{1}{2}|\bm{u}|^2$ is usually not in the approximation space for the densities.

\begin{lemma}[Variational formulation]
Smooth solutions $(\vec\rho,\vec\mu,\bm{u},p)$ to \eqref{1.mass}--\eqref{1.flux}, \eqref{1.chem}--\eqref{1.incom} satisfy the following variational equations:
\begin{align}
  & \langle\pa_t\rho_i,\psi_i\rangle 
  - \langle\rho_i\bm{u},\psi_i\rangle
  + \sum_{j=1}^N\big\langle M_{ij}(\vec\rho)\na\mu_j,\na\psi_i
  \big\rangle = 0, \label{2.var1} \\
  & \langle\mu_i,\xi_i\rangle - \bigg\langle
  \frac{\pa f}{\pa\rho_i}(\vec\rho) + V_i p,\xi_i\bigg\rangle = 0, 
  \quad i=1,\ldots,N, \label{2.var2} \\
  & \frac12\langle \bm{u}\pa_t\rho,\bm{v}\rangle
  + \langle\rho\pa_t\bm{u},\bm{v}\rangle
  - \bm{b}_{\rm skw}(\rho\bm{u};\bm{u},\bm{v}) \label{2.var3} \\
  &\phantom{xx}- \langle p,\divM\bm{v}\rangle
  + \langle\mathbb{S}((\vec\rho,\na\bm{u})),\na\bm{v}\rangle 
  + \sum_{i=1}^N\langle\rho_i\na(\mu_i-V_ip),\bm{v}\big\rangle = 0, 
  \nonumber \\
  & \langle\divM\bm{u},q\rangle = -\sum_{i,j=1}^N\big\langle 
  V_iM_{ij}(\vec\rho)\na\mu_j,\na q\big\rangle \label{2.var4}
\end{align}
for all smooth test functions $(\psi_i,\xi_i,\bm{v},q)$, $i=1,\ldots,N$.
\end{lemma}

\begin{proof}
Equations \eqref{2.var1}, \eqref{2.var2}, and \eqref{2.var4} follow directly from \eqref{1.mass2}, \eqref{1.mom2}, and \eqref{1.divMu}, respectively, after multiplication with a test function and integration by parts. We only need to show that \eqref{2.var3} follows from the momentum equation \eqref{1.mom2}. To this end, we multiply \eqref{1.mom2} by a test function $\bm{v}$ and integrate by parts:
\begin{align*}
  0 = \langle\pa_t(\rho\bm{u}),\bm{v}\rangle
  + \big\langle\divM(\rho\bm{u}\otimes\bm{u}),\bm{v}\big\rangle
  + \langle\mathbb{S}((\vec\rho,\na\bm{u})),\na\bm{v}\rangle
  - \langle p,\divM\bm{v}\rangle.
\end{align*}
Using the mass balance equation $\pa_t\rho=-\divM(\rho\bm{u})$ and integrating by parts, the first two terms can be written as
\begin{align*}
  \langle\pa_t&(\rho\bm{u}),\bm{v}\rangle
  + \big\langle\divM(\rho\bm{u}\otimes\bm{u}),\bm{v}\big\rangle
  = \langle\bm{u}\pa_t\rho,\bm{v}\rangle
  + \langle\rho\pa_t\bm{u},\bm{v}\rangle
  + \big\langle\divM(\rho\bm{u}\otimes\bm{u}),\bm{v}\big\rangle \\
  &= \frac12\langle\bm{u}\pa_t\rho,\bm{v}\rangle
  - \frac12\langle\bm{u}\divM(\rho\bm{u}),\bm{v}\rangle
  + \langle\rho\pa_t\bm{u},\bm{v}\rangle
  + \big\langle\divM(\rho\bm{u}\otimes\bm{u}),\bm{v}\big\rangle \\
  &= \frac12\langle\bm{u}\pa_t\rho,\bm{v}\rangle
  + \frac12\big(\langle\rho(\bm{u}\cdot\na)\bm{u},\bm{v}\rangle
  + \langle\rho(\bm{u}\cdot\na)\bm{v},\bm{u}\rangle\big)
  + \langle\rho\pa_t\bm{u},\bm{v}\rangle
  - \langle\rho(\bm{u}\cdot\na)\bm{u},\bm{v}\rangle \\
  &= \frac12\langle\bm{u}\pa_t\rho,\bm{v}\rangle
  + \langle\rho\pa_t\bm{u},\bm{v}\rangle
  - \frac12\big(\langle\rho(\bm{u}\cdot\na)\bm{u},\bm{v}\rangle
  - \langle\rho(\bm{u}\cdot\na)\bm{v},\bm{u}\rangle\big) \\
  &= \frac12\langle\bm{u}\pa_t\rho,\bm{v}\rangle
  + \langle\rho\pa_t\bm{u},\bm{v}\rangle 
  - \bm{b}_{\rm skw}(\rho\bm{u};\bm{u},\bm{v}).
\end{align*}
Adding the equation
\begin{align*}
  0 &= \sum_{i=1}^N\big(\langle\rho_i\na\mu_i,\bm{v}\rangle
  + \langle\mu_i,\divM(\rho_i\bm{v})\rangle\big) \\
  &= \sum_{i=1}^N\bigg(\langle\rho_i\na\mu_i,\bm{v}\rangle
  + \bigg\langle\frac{\pa f}{\pa\rho_i}+V_ip,
  \divM(\rho_i\bm{v})\bigg\rangle\bigg) \\
  &= \sum_{i=1}^N\bigg(\langle\rho_i\na\mu_i,\bm{v}\rangle
  - \bigg\langle\rho_i\na\frac{\pa f}{\pa\rho_i},\bm{v}\bigg\rangle
  - \langle\na p,V_i\rho_i\bm{v}\rangle\bigg) \\
  &= \sum_{i=1}^N\langle\rho_i\na(\mu_i-V_ip),\bm{v}\rangle
\end{align*}
yields our claim \eqref{2.var3}. \qed
\end{proof}

In the following section, we suggest a discrete finite element version of the previous variational formulation.


\section{Numerical analysis}\label{sec.num}

\subsection{Numerical scheme}

We partition the time interval $[0,T]$ uniformly into $n$ subintervals $[t_{k-1},t_k]$, where $t_k=k\tau$, $k=1,\ldots,n$, and $\tau=T/n$ is the time step size. We introduce 
\begin{align*}
  \Pi_\tau^0(\mathcal{U}) 
  &= \mbox{space of discontinuous piecewise constant functions on }
  \{t_0,\ldots,t_n\} \\
  &\phantom{xx} \mbox{with values in }\mathcal{U}, \\
  \Pi_\tau^1(\mathcal{U})  
  &= \mbox{space of continuous piecewise linear functions on }
  \{t_0,\ldots,t_n\}\mbox{ with values in }\mathcal{U}.
\end{align*}
For the spatial discretization, we require that $\mathcal{T}_h$ is a geometrically conforming partition of $\Omega$ into simplices that can be extended periodically to periodic extensions of $\Omega$, with $h>0$ being the maximal diameter of the simplices. In particular, the boundary of $\Omega$ coincides with the boundary of the triangulation. Let $P_k(K)$ be the space of polynomials with maximal degree $k\in\N$ on $K\in\mathcal{T}_h$. We define the spaces of piecewise linear and quadratic functions by, respectively,
\begin{align*}
  \mathcal{V}_h &= \big\{v\in H^1(\Omega)\cap C^0(\overline\Omega):
  v|_K\in P_1(K)\mbox{ for all }K\in\mathcal{T}_h\big\}, \\
  \mathcal{X}_h &= \big\{v\in H^1(\Omega)\cap C^0(\overline\Omega):
  v|_K\in P_2(K)\mbox{ for all }K\in\mathcal{T}_h\big\},
\end{align*}
and the spaces of mean-free and positive functions, respectively,
\begin{align*}
  \mathcal{Q}_h = \{v\in\mathcal{V}_h:\langle v,1\rangle=0\}, \quad
  \mathcal{V}_{h,+} = \{v\in\mathcal{V}_h:v(x)>0
  \mbox{ for }x\in\Omega\},
\end{align*}
recalling that $\langle u,v\rangle = \sum_{K\in\mathcal{T}_h} u|_K v|_K$. 

Let the initial data $({\pvec{\rho}}_{h}^{0},\bm{u}_h^0) \in\mathcal{V}_{h}^N\times\mathcal{X}_h^d$ be given. The numerical scheme reads as follows: Find 
$(\pvec{\rho}_{h}^{k}, \bm{u}_h^{k})\in \Pi_\tau^1(\mathcal{V}_{h}^N\times\mathcal{X}_h^d)$ and 
$(\vec\mu_{h}^k, p_h^k)\in \Pi_\tau^0(\mathcal{V}_h^N\times\mathcal{Q}_h)$
for $k=0,\ldots,n-1$ such that for $i=1,\ldots,N$,
\begin{align}
  & \frac{1}{\tau}\langle \rho_{h,i}^{k+1}-\rho_{h,i}^{k},\psi_i\rangle
  - \langle\rho_{h,i}^k\bm{u}_h^{k+1},\na\psi_i\rangle
  + \sum_{j=1}^N\big\langle M_{ij}(\pvec{\rho}_{h}^{*})\na\mu_{h,j}^{k+1},
  \na\psi_i\big\rangle = 0, \label{num.mass} \\
  & \langle\mu_{h,i}^{k+1},\xi_i\rangle 
  - \bigg\langle\frac{\pa f}{\pa\rho_i}(\pvec{\rho}_h^{k+1})
  + V_ip_h^{k+1},\xi_i\bigg\rangle = 0, \label{num.mu} \\
  & \frac{1}{2\tau}\big\langle\bm{u}_h^{k+1}
  (\rho_h^{k+1}-\rho_h^k),\bm{v}\big\rangle 
  + \frac{1}{\tau}\big\langle\rho_h^k
  (\bm{u}_h^{k+1}-\bm{u}_h^k),\bm{v}\big\rangle
  + \bm{b}_{\rm skw}(\rho_h^*\bm{u}_h^*;\bm{u}_h^{k+1},\bm{v}) 
  \label{num.mom} \\
  &\phantom{xx}- \langle p_h^{k+1},\divM\bm{v}\rangle
  + \big\langle\mathbb{S}(\pvec{\rho}_{h}^{*},
  \na\bm{u}_h^{k+1}),\na\bm{v}\big\rangle 
  + \sum_{i=1}^N\big\langle\rho_{h,i}^k
  \na(\mu_{h,i}^{k+1}-V_ip_h^{k+1}), 
  \bm{v}\big\rangle = 0, \nonumber \\
  &\langle\divM\bm{u}_h^{k+1},q\rangle
  = -\sum_{i,j=1}^N\big\langle V_iM_{ij}(\pvec{\rho}_{h}^{*})
  \na\mu_{h,j}^{k+1},\na q\big\rangle \label{num.divMu}
\end{align}
for all $(\psi_i,\xi_i,\bm{v},q)\in\mathcal{V}_h^N\times\mathcal{V}_h^N
\times\mathcal{X}_h^d\times\mathcal{Q}_h$. Here, $\rho_h^*$ and $\bm{u}_h^*$ are functions of $(\rho_h^{k},\rho_h^{k+1})$ and $(\bm{u}_h^k,\bm{u}_h^{k+1})$, respectively. Examples are the explicit scheme $(\rho_h^*,\bm{u}_h^*)=(\rho_h^k,\bm{u}_h^k)$, the implicit scheme $(\rho_h^*,\bm{u}_h^*)=(\rho_h^{k+1},\bm{u}_h^{k+1})$, or the mid-point rule. In the numerical experiments, we have chosen the explicit approximation. The discrete pressure terms in \eqref{num.mom} encode the incompressibility condition \eqref{1.incom}. In fact, if \eqref{1.incom} holds, these terms cancel out. This means that if the discrete solution satisfies the constraint \eqref{1.incom}, the discrete pressure does not appear in \eqref{num.mom} directly, but only implicitly via the chemical potentials.

We mentioned already in the introduction that we assume the existence of a discrete solution to \eqref{num.mass}--\eqref{num.divMu}, since a proof is far from being trivial. In particular, we expect that $\rho_{h,i}^k$ is nonnegative for $i=1,\ldots,N$ and $k=1,\ldots,n$. This is consistent with the continuous model; see Lemma \ref{lem.pos}. Hence, we expect that the positivity of the discrete partial densities is a consequence of the singular behavior of the free energy density or its derivatives.

\subsection{Properties of the numerical scheme}

The following numerical properties hold, given that the discrete solutions are positive.

\begin{theorem}[Discrete identities]\label{thm.id}
Let Assumptions (A0) -- (A4) hold. Every discrete solution $(\pvec{\rho}_{h}^{k}, \bm{u}_h^{k})\in \Pi_\tau^1(\mathcal{V}_{h,+}^N \times\mathcal{X}_h^d)$, $(\vec\mu_{h}^k, p_h^k)\in \Pi_\tau^0(\mathcal{V}_h^N\times\mathcal{Q}_h)$ to \eqref{num.mass}--\eqref{num.divMu} satisfies the following identities for $k\in \{0,\ldots,n-1\}$.
\begin{itemize}
\item Partial mass conservation: For all $i=1,\ldots,N$,
\begin{align*}
  \langle\rho_{h,i}^{k+1},\mathrm{1}\rangle
  = \langle\rho_{h,i}^{k},\mathrm{1}\rangle.
\end{align*}
\item Total mass conservation equation: For all $q\in\mathcal{V}_h$,
\begin{align*}
  \frac{1}{\tau}\langle\rho_h^{k+1}-\rho_h^k,q\rangle
  - \langle\rho_h^k\bm{u}_h^{k+1},\na q\rangle = 0.
\end{align*}
\item Constraint: For all $q\in\mathcal{Q}_h$,
\begin{align}\label{num.varcon}
  \frac{1}{\tau}\bigg\langle\sum_{i=1}^N V_i\rho_{h,i}^{k+1} - 1,q
  \bigg\rangle - \frac{1}{\tau}\bigg\langle\sum_{i=1}^N V_i\rho_{h,i}^{k} - 1,q\bigg\rangle
  = \bigg\langle\bigg(\sum_{i=1}^N V_i\rho_{h,i}^{k} - 1\bigg)
  \bm{u}_h^{k+1},\na q\bigg\rangle.
\end{align}
\item Discrete energy equality:
\begin{align*}
  \frac{1}{\tau}&\langle E_h^{k+1}-E_h^k,\mathrm{1}\rangle \\
  &= -\big\langle\mathbb{S}(\pvec{\rho}_{h}^{*},\na\bm{u}_h^{k+1}),\na\bm{u}_h^{k+1}
  \big\rangle - \sum_{i,j=1}^N\big\langle M_{ij}(\pvec{\rho}_{h}^{*})
  \na\mu_{h,i}^{k+1},\na\mu_{h,j}^{k+1}\big\rangle 
  - \mathcal{D}_{\rm num} \le 0,
\end{align*}
where the discrete energy density and numerical dissipation are given by
\begin{align*}
  E_h^k &= \frac12\rho_h^k|\bm{u}_h^k|^2 + f(\pvec{\rho}_{h}^{k}), \\
  \mathcal{D}_{\rm num} &= \frac{1}{2\tau}\|(\rho_h^k)^{1/2}
  (\bm{u}_h^{k+1}-\bm{u}_h^k)\|_{L^2(\Omega)}^2
  + \frac{1}{\tau}\sum_{i,j=1}^N
  \bigg\langle\frac{\pa^2 f}{\pa\rho_i\pa\rho_j}
  (\vec\omega_h)(\rho_{h,i}^{k+1}-\rho_{h,i}^k),
  \rho_{h,j}^{k+1}-\rho_{h,j}^k\bigg\rangle,
\end{align*}
and $\vec\omega_h$ is a convex combination of $\pvec{\rho}_{h}^{k}$ and $\pvec{\rho}_{h}^{k+1}$.
\end{itemize} 
\end{theorem}

\begin{proof}
The conservation of the partial masses is obtained by choosing $\psi_i=1$ in \eqref{num.mass}. The sum of \eqref{num.mass} over $i=1,\ldots,N$ with $\psi_i=q$ and the fact that $\sum_{i=1}^N M_{ij}(\vec\rho_h)=0$ leads to the total mass conservation equation. The equation for the constraint follows from \eqref{num.mass} by choosing $\psi_i=V_iq\in\mathcal{V}_h$ for $q\in\mathcal{Q}_h$ (hence $\langle q,\mathrm{1}\rangle=0$), summing over $i=1,\ldots,N$, and using \eqref{num.divMu}:
\begin{align*}
  \frac{1}{\tau}\bigg\langle&\sum_{i=1}^N V_i\rho_{h,i}^{k+1}-1,q
  \bigg\rangle - \frac{1}{\tau}\bigg\langle\sum_{i=1}^N V_i\rho_{h,i}^{k}-1,q\bigg\rangle \\
  &= \bigg\langle\sum_{i=1}^N V_i\rho_{h,i}^k\bm{u}_h^{k+1},\na q
  \bigg\rangle + \bigg\langle\sum_{i=1}^N V_i\bm{J}_{h,i}^{k+1},
  \na q\bigg\rangle \\
  &= \bigg\langle\sum_{i=1}^N V_i\rho_{h,i}^k\bm{u}_h^{k+1},\na q
  \bigg\rangle + \langle\divM\bm{u}_h^{k+1},q\rangle 
  = \bigg\langle\bigg(\sum_{i=1}^N V_i\rho_{h,i}^k-1\bigg)
  \bm{u}_h^{k+1},\na q\bigg\rangle.
\end{align*}
For the energy identity, we choose the test function $\bm{v}=\bm{u}_h^{k+1}$ in \eqref{num.mom}:
\begin{align}
  \frac{1}{2\tau}&\big\langle \bm{u}_h^{k+1}(\rho_h^{k+1}-\rho_h^k),
  \bm{u}_h^{k+1}\big\rangle + \frac{1}{\tau}\big\langle
  \rho_h^k(\bm{u}_h^{k+1}-\bm{u}_h^k),\bm{u}_h^{k+1}\big\rangle
  + \bm{b}_{\rm skw}(\rho_h^*\bm{u}_h^*,\bm{u}_h^{k+1},\bm{u}_h^{k+1}) 
  \label{num.aux} \\
  &- \langle p_h^{n+1},\divM\bm{u}_h^{k+1} \rangle+\big\langle\mathbb{S}(\pvec{\rho}_{h}^{*},\na\bm{u}_h^{k+1}),
  \na\bm{u}_h^{k+1}\big\rangle 
  + \sum_{i=1}^N\langle\rho_{h,i}^k\na(\mu_{h,i}^{k+1} -V_ip_h^{k+1}) ,
  \bm{u}_h^{k+1}\rangle = 0. \nonumber 
\end{align}
Notice that $\bm{b}_{\rm skw}(\rho_h^*\bm{u}_h^*,\bm{u}_h^{k+1}, \bm{u}_h^{k+1}) = 0$. An elementary computation shows that
\begin{align*}
  \frac12(&\rho_h^{k+1}-\rho_h^k)|\bm{u}_h^{k+1}|^2
  + \rho_h^k(\bm{u}_h^{k+1}-\bm{u}_h^k)\cdot\bm{u}_h^{k+1} \\
  &= \frac12\rho_h^{k+1}|\bm{u}_h^{k+1}|^2 
  - \frac12\rho_h^k|\bm{u}_h^k|^2
  + \frac12\rho_h^k|\bm{u}_h^{k+1}-\bm{u}_h^k|^2.
\end{align*}
Thus, replacing the first two terms in \eqref{num.aux}, we have
\begin{align}\label{num.aux2}
  \frac{1}{2\tau}&\langle\rho_h^{k+1}|\bm{u}_h^{k+1}|^2,\mathrm{1}\rangle
  - \frac{1}{2\tau}\langle\rho_h^k|\bm{u}_h^k|^2,\mathrm{1}\rangle
  + \frac{1}{2\tau}\langle\rho_h^k|\bm{u}_h^{k+1}-\bm{u}_h^k|^2,\mathrm{1}\rangle
  - \langle p_h^{k+1},\divM\bm{u}_h^{k+1}\rangle \\
  &+ \big\langle\mathbb{S}(\pvec{\rho}_{h}^{*},\na\bm{u}_h^{k+1}),
  \na\bm{u}_h^{k+1}\big\rangle
  + \sum_{i=1}^N\big\langle\rho_{h,i}^k\na(\mu_{h,i}^{k+1}
  - V_ip_h^{k+1}),\bm{u}_h^{k+1}\big\rangle = 0. \nonumber 
\end{align}
To reformulate the fourth term, we take $q=p_h^{k+1}$ in \eqref{num.divMu} and then $\psi_i=V_ip_h^{k+1}$ in \eqref{num.mass}:
\begin{align*}
  \langle p_h^{k+1},\divM \bm{u}_h^{k+1} \rangle 
  &= - \sum_{i,j=1}^N \langle V_iM_{ij}(\pvec{\rho}_h^{*})
  \nabla\mu_{h,j}^{k+1},\nabla p_{h}^{k+1} \rangle \\
  &= \sum_{i=1}^N \frac{1}{\tau}\langle 
  \rho_{h,i}^{k+1}-\rho_{h,i}^k, V_ip_h^{k+1} \rangle 
  - \sum_{i=1}^N\langle \rho_{h,i}^k\bm{u}_h^{k+1},
  V_i\nabla p_h^{k+1}\rangle. 
\end{align*}
Combining \eqref{num.aux2} and the previous expression, some terms cancel:
\begin{align}\label{num.aux4}
  \frac{1}{2\tau}&\langle\rho_h^{k+1}|\bm{u}_h^{k+1}|^2,
  \mathrm{1}\rangle
  - \frac{1}{2\tau}\langle\rho_h^k|\bm{u}_h^k|^2,
  \mathrm{1}\rangle
  + \frac{1}{2\tau}\langle\rho_h^k|\bm{u}_h^{k+1}-\bm{u}_h^k|^2,
  \mathrm{1}\rangle \\
  &+ \big\langle\mathbb{S}(\pvec{\rho}_{h}^{*},\na\bm{u}_h^{k+1}),
  \na\bm{u}_h^{k+1}\big\rangle
  + \sum_{i=1}^N\big\langle\rho_{h,i}^k\na\mu_{h,i}^{k+1}, 
  \bm{u}_h^{k+1}\big\rangle 
  = \frac{1}{\tau}\sum_{i}^N\langle \rho_{h,i}^{k+1}-\rho_{h,i}^k, 
  V_ip_h^{k+1} \rangle. \nonumber 
\end{align}
For the internal energy contribution, we add and subtract some terms:
\begin{align}\label{num.aux5}
  \frac{1}{\tau}&\langle f(\pvec{\rho}_h^{k+1}),1\rangle - \frac{1}{\tau} f(\pvec{\rho}_h^{k}),1\rangle 
  = \frac{1}{\tau}\sum_{i=1}^N \bigg\langle \frac{\partial f(\pvec{\rho}_h^{k+1})}{\partial \rho_i},\rho_{h,i}^{k+1}-\rho_{h,i}^{k} \bigg\rangle \\
  &+ \frac{1}{\tau}\langle f(\pvec{\rho}_h^{k+1}),1\rangle - \frac{1}{\tau}\langle f(\pvec{\rho}_h^{k}),1\rangle - \frac{1}{\tau}\sum_{i=1}^N \bigg\langle \frac{\partial f(\pvec{\rho}_h^{k+1})}{\partial \rho_i},\rho_{h,i}^{k+1}-\rho_{h,i}^{k} \bigg\rangle. \nonumber 
\end{align}
The last line can be rewritten as part of the numerical dissipation $\mathcal{D}_{\rm num}$ using Taylor expansion.
Using the definition of the discrete energy $E_h$ and the numerical dissipation $\mathcal{D}_{\rm num}$, a combination of \eqref{num.aux4} and \eqref{num.aux5} gives
\begin{align}\label{num.aux7}
  \frac{1}{\tau}\langle E^{k+1}_h - E^{k}_h,1 \rangle 
  &= - \mathcal{D}_{\rm num} 
  - \big\langle\mathbb{S}(\pvec{\rho}_{h}^{*},\na\bm{u}_h^{k+1}),
  \na\bm{u}_h^{k+1}\big\rangle \\
  &\phantom{xx}+ \frac{1}{\tau}\sum_{i=1}^N \bigg\langle \rho_{h,i}^{k+1}-\rho_{h,i}^{k},\frac{\partial f(\pvec{\rho}_h^{k+1})}{\partial \rho_i} + V_ip_h^{k+1}
  \bigg\rangle
  - \sum_{i=1}^N\langle\rho_{h,i}^k\na \mu_{h,i}^{k+1},
  \bm{u}_h^{k+1}\rangle. \notag
\end{align}
We choose $\xi_i=(\rho_{h,i}^{k+1}-\rho_{h,i}^{k})/\tau$ in \eqref{num.mu} and then $\psi_i=\mu_{h,i}^{k+1}$ in \eqref{num.mass} to rewrite the third term on the right-hand side:
\begin{align*}
  \frac{1}{\tau}\sum_{i=1}^N & \bigg\langle \rho_{h,i}^{k+1}-\rho_{h,i}^{k},\frac{\partial f(\pvec{\rho}_h^{k+1})}{\partial \rho_i} + V_ip_h^{k+1}\bigg\rangle 
  = \frac{1}{\tau}\sum_{i=1}^N \langle \rho_{h,i}^{k+1}-\rho_{h,i}^{k},
  \mu_{h,i}^{k+1} \rangle \\
  &= \sum_{i=1}^N \langle \rho_{h,i}^k \bm{u}_h^{k+1},\nabla\mu_{h,i}^{k+1} \rangle 
  - \sum_{i,j=1}^N \langle M_{i,j}(\pvec{\rho}_{h}^{*})\nabla\mu_{h,i}^{k+1},
  \nabla\mu_{h,j}^{k+1} \rangle.
\end{align*}
The first term on the right-hand side cancels with the last term in \eqref{num.aux7}. This shows that
\begin{align*}
  \frac{1}{\tau}\langle E^{k+1}_h - E^{k}_h,1 \rangle 
  = -\mathcal{D}_{\rm num}
  - \big\langle\mathbb{S}(\pvec{\rho}_{h}^{*},\na\bm{u}_h^{k+1}),
  \na\bm{u}_h^{k+1}\big\rangle 
  - \sum_{i,j=1}^N \langle M_{i,j}(\pvec{\rho}_{h}^{*})
  \nabla\mu_{h,i}^{k+1},\nabla\mu_{h,j}^{k+1}\rangle,
\end{align*}
which finishes the proof.
\end{proof}

Assumptions (A0)--(A4) imply some pointwise bounds and bounds in $L^2(\Omega)$.

\begin{theorem}[A priori bounds]\label{thm.bd}
Let Assumptions (A0)--(A4) hold and  let $(\pvec{\rho}_{h}^{k}, \bm{u}_h^{k})\in\Pi_\tau^1(\mathcal{V}_{h,+}^N$ $\times\mathcal{X}_h^d)$, $(\vec\mu_{h}^k, p_h^k)\in \Pi_\tau^0(\mathcal{V}_h^N\times\mathcal{Q}_h)$ be a solution to \eqref{num.mass}--\eqref{num.divMu}. Then, for all $k=1,\ldots,n$:
\begin{itemize}
\item Pointwise constraint:
\begin{align*}
  \sum_{i=1}^NV_i\rho_{h,i}^k(x)=1\quad\mbox{for all }x\in\Omega.
\end{align*}
\item Pointwise bounds:  
\begin{align*}
  \rho_{h,i}^k\le V_i^{-1}\quad\mbox{for }i=1,\ldots,N, \quad 
  V_{\rm max}^{-1}\le\rho_h^k\le V_{\rm min}^{-1},
\end{align*}
where $V_{\rm min}=\min\{V_1,\ldots,V_N\}$ and $V_{\rm max}=\max\{V_1,\ldots,V_N\}$. 
\item A priori estimates:
\begin{align*}
  \max_{k=1,\ldots,n}\big(\|\bm{u}_h^k\|_{L^2(\Omega)}^2
  + \|(\rho_{h}^k)^{1/2}\bm{u}_h^k\|_{L^2(\Omega)}^2\big)
  + \sum_{k=1}^{n}\tau\|\na\bm{u}_h^k\|_{L^2(\Omega)}^2
  \le C(E_h^0), \\
\end{align*}
\end{itemize}
\end{theorem}

Bounds for the gradients of the partial mass densities strongly depend on the structure of the mobility matrix $(M_{ij}(\pvec{\rho}_{h}^{k}))$. Matrix analysis arguments show for the continuous Maxwell--Stefan system that the mass densities are bounded in $H^1(\Omega)$ \cite[Lemma 3.2]{JuSt13}.

\begin{proof}
Set $Z_h^k:=\sum_{i=1}^N V_i\rho_{h,i}^k-1$. Assumption (A4) implies that $Z_h^0=0$. 
We use an induction argument to show that $Z_h^k=0$ for all $k$. We wish to take $Z_h^{k+1}$ as a test function in \eqref{num.varcon}. To this end, we first show that $Z_h^{k+1}\in \mathcal{Q}_h$.
By partial mass conservation, $\langle V_i\rho_{h,i}^{k},\mathrm{1}\rangle = \langle V_i\rho_{h,i}^0,\mathrm{1}\rangle$ for all $i=1,\ldots,N$ and $k=1,\ldots,n$. Summation over $i=1,\ldots,N$ shows that
\begin{align*}
  \langle Z_h^k,\mathrm{1}\rangle
  = \bigg\langle\sum_{i=1}^N V_i\rho_{h,i}^k-1,\mathrm{1}\bigg\rangle
  = \bigg\langle\sum_{i=1}^N V_i\rho_{h,i}^0-1,\mathrm{1}\bigg\rangle 
  = \langle Z_h^0,\mathrm{1}\rangle = 0.
\end{align*}
Equation \eqref{num.varcon} with the test function $q=Z_h^{k+1}\in\mathcal{Q}_h$ reads as
\begin{align*}
  \|Z_h^{k+1}\|_{L^2(\Omega)}^2 + \|Z_h^{k+1}-Z_h^{k}\|_{L^2(\Omega)}^2
  =\|Z_h^{k}\|_{L^2(\Omega)}^2 + \tau\langle Z_h^k\bm{u}_h^{k+1},\na Z_h^{k+1}\rangle = 0.
\end{align*}
We conclude from the above equation that if $Z_h^{k}=0$ everywhere, then also $Z_h^{k+1}=0$ everywhere. This shows the claim.

Next, the upper bound for $\rho_{h,i}^k$ follows from $V_j\rho_{h,j}^k=1-\sum_{i\neq j} V_i\rho_{h,i}^k\le 1$, using the nonnegativity of $\rho_{h,i}^k$. The bounds for the total mass density is a consequence of
\begin{align*}
  \frac{1}{V_{\rm max}} 
  = \sum_{i=1}^N \frac{V_i}{V_{\rm max}}\rho_{h,i}^k
  \le \sum_{i=1}^N\rho_{h,i}^k 
  \le \sum_{i=1}^N \frac{V_i}{V_{\rm min}}\rho_{h,i}^k
  = \frac{1}{V_{\rm min}}.
\end{align*}
Finally, the discrete energy identity from Theorem \ref{thm.id} and Assumption (A3) imply directly the a priori estimates.
\end{proof}

\subsection{Remarks}

We discuss some variants of the numerical scheme. 

\begin{remark}[Implicit discretization in \eqref{num.mass}]\rm
Observe that the choice of the old time step $\rho_{h,i}^k$ in the transport term of \eqref{num.mass} allows us to derive the pointwise form of the constraint. A similar result can be proven for the implicit discretization, i.e.\ replacing $\rho_{h,i}^k\bm{u}_h^{k+1}$ by $\rho_{h,i}^{k+1}\bm{u}_h^{k+1}$ in \eqref{num.mass} under the assumption that all mass densities $\rho_{h,i}^k$ are uniformly positive independently of the time step $\tau$ and under a CFL-type condition. Indeed, similarly as in the previous proof, we find that
\begin{align*}
  \langle Z_h^1,q\rangle - \langle Z_h^0,q\rangle
  = \tau\langle Z_h^1\bm{u}_h^1,\na q\rangle,
\end{align*}
and the choice $q=Z_h^1$ leads to 
\begin{align*}
  \|q\|_{L^2(\Omega)}^2 = \frac{\tau}{2}\langle \bm{u}_h^{1},
  \na|Z_h^1|^2\rangle = -\frac{\tau}{2}\langle \divM\bm{u}_h^{1},
  |Z_h^1|^2\rangle \le \frac{\tau}{2}\|\divM\bm{u}_h^1\|_{L^2(\Omega)}
  \|Z_h^1\|_{L^2(\Omega)}^2.
\end{align*}
It follows from the inverse inequality for finite element functions between the spaces $W^{1,\infty}(\Omega)$ and $L^2(\Omega)$ that
\begin{align*}
  \|q\|_{L^2(\Omega)}^2 \le C\tau h^{-1-d/2}
  \|\bm{u}_h^1\|_{L^2(\Omega)}\|Z_h^1\|_{L^2(\Omega)}^2
  \le C_0\tau h^{-1-d/2}\|q\|_{L^2(\Omega)}^2.
\end{align*}
We infer from the CFL-type condition $\tau h^{-1-d/2} < 1/C_0$ that $\|q\|_{L^2(\Omega)}^2=0$ and hence $Z_h^1=0$.
\end{remark}

\begin{remark}[Equal specific volumes]\rm
In the case $V_i=V$ for $i=1,\ldots,N$, the numerical scheme simplifies. Indeed, we obtain $1=\sum_{i=1}^N V_i\rho_{h,i}^k=V\rho_h^k$ and hence $\rho_h^k=V^{-1}$. Next, we choose the test function $\psi_i=\psi\in\mathcal{V}_h$ in \eqref{num.mass}, sum over $i=1,\ldots,N$, and use $\sum_{i=1}^N M_{ij}(\pvec{\rho}_{h}^{*})=0$:
\begin{align*}
  \langle V^{-1}\bm{u}_h^{k+1},\na\psi\rangle 
  = -\frac{1}{\tau}\sum_{i=1}^N\langle\rho_{h,i}^{k+1}-\rho_{h,i}^k,
  \psi\rangle
  + \sum_{i,j=1}^n\big\langle M_{ij}(\pvec{\rho}_{h}^{*})
  \na\mu_{h,j}^{k+1},\na\psi\big\rangle = 0,
\end{align*}
which yields $\langle\bm{u}_h^{k+1},\na\psi\rangle=0$ for all $\psi\in\mathcal{V}_h$. The discrete momentum equation \eqref{num.mom} then becomes for $\bm{v}\in\mathcal{X}_h^d$: 
\begin{align*}
  \langle\rho&(\bm{u}_h^{k+1}-\bm{u}_h^k),\bm{v}\rangle
  + \bm{b}_{\rm skw}(\rho\bm{u}_h^*,\bm{u}_h^{k+1},\bm{v})
  - \langle p_h^{k+1},\divM\bm{v}\rangle \\
  &+ \big\langle\mathbb{S}(\pvec{\rho}_{h}^{*},\na\bm{u}_h^{k+1}),\na\bm{v}\big\rangle
  + \sum_{i=1}^N\big\langle\rho_{h,i}^k\na(\mu_{h,i}^{k+1}
  -V_ip_h^{k+1}),\bm{v}\big\rangle = 0.
\end{align*}

We claim that we can formulate a nodal representation in which the pressure does not appear. For this, let $(\phi_\ell)$ be a (finite-dimensional) basis of $\mathcal{V}_h$, let $L_{\ell m}=\langle\phi_\ell,\phi_m\rangle$ be the entries of the mass matrix $\mathbb{L}$, $K_{ij,\ell m}^*=\langle M_{ij}(\pvec{\rho}_{h}^{*})\na\phi_\ell,\na\phi_m\rangle$ be the entries of the weighted stiffness matrix $\mathbb{K}_{ij}^*$, and $F_{i,j}^k=\langle(\pa f/\pa\rho_i)(\pvec{\rho}_{h}^{k}),\phi_j\rangle$ be the entries of the nonlinear term $\bm{F}_i^k$. Slightly abusing our notation, we set $\bm{\mu}_{i}^k = \langle\mu_{h,i}^k, \phi_\ell\rangle_{\ell}$ and $\bm{p}^k=\langle p_h^k, \phi_\ell\rangle_{\ell}$. Then \eqref{num.mu} becomes
\begin{align*}
  \bm{\mu}_i^k = \mathbb{L}^{-1}\bm{F}_i^{k} + V\bm{p}^k
  \quad\mbox{for }i=1,\ldots,N
\end{align*}
and consequently,
\begin{align*}
  \sum_{j=1}^n\big\langle M_{ij}(\pvec{\rho}_{h}^{*})\na\mu_{h,j}^k,
  \na\phi_i\big\rangle_\ell
  = \sum_{j=1}^N \mathbb{K}_{ij}^*\bm{\mu}_j^k
  = \sum_{j=1}^N \mathbb{K}_{ij}^*\big(\mathbb{L}^{-1}\bm{F}_j^k
  + V\bm{p}^k\big) 
  = \sum_{j=1}^N \mathbb{K}_{ij}^*\mathbb{L}^{-1}\bm{F}_j^k,
\end{align*}
where the last step follows from the fact that the sum over the entries of the mobility matrix vanishes. Then, defining $\widehat{\bm{\mu}}_j^k=\mathbb{L}^{-1}\bm{F}_j^k$,
\begin{align*}
  \sum_{j=1}^n\big\langle M_{ij}(\pvec{\rho}_{h}^{*})\na\mu_{h,j}^k,
  \na\phi_\ell\big\rangle
  = \sum_{j=1}^n\big\langle M_{ij}(\pvec{\rho}_{h}^{*})\na\widehat{\mu}_{j}^k,
  \na\phi_\ell\big\rangle,
\end{align*}
which means that we can replace $\bm{\mu}_j^k$ by $\widehat{\bm{\mu}}_j^k$. Furthermore, introducing the matrix $\mathbb{G}(\rho_{h,i}^k)$ with entries $G_{\ell m}(\rho_{h,i}^k)=\langle\rho_{h,i}^k\na\phi_\ell,\phi_m\rangle$ and inserting the vector equation for $\bm{\mu}_i^k$,
\begin{align*}
  \sum_{i=1}^N\big\langle\rho_{h,i}^k(\na\mu_{h,i}^k - V\na p_h^k),
  \bm{v}\big\rangle 
  &= \sum_{i=1}^N\mathbb{G}(\rho_{h,i}^k)\big(\bm{\mu}_i^k
  - V\bm{p}^k\big)\bm{v} \\
  &= \sum_{i=1}^N\mathbb{G}(\rho_{h,i}^k)\mathbb{L}^{-1}\bm{F}_i^k\bm{v}
  = \sum_{i=1}^N\langle\rho_{h,i}^k\na\widehat{\mu}_{h,i}^k,
  \bm{v}\rangle.
\end{align*}
Thus, only $\widehat{\bm{\mu}}_i^k$ appears in the incompressible Navier--Stokes--Maxwell--Stefan equations instead of $\bm{\mu}_i^k$.
\end{remark}


\section{Numerical experiments}\label{sec.exp}

In this section, we present some numerical simulations in two space dimensions, illustrating the convergence of the scheme and the structure-preserving properties. The scheme is implemented in NGSolve \cite{Sch14}, using Netgen for the generation of an unstructured mesh of triangles. The resulting nonlinear system is solved by Newton's method with a tolerance of $10^{-9}$ (measured in the $L^2(\Omega)$ norm of the Newton updates), and the linear systems are computed via direct LU decomposition. In all tests, we fixed the computational domain $\Omega=(0,1)^2$ with periodic boundary conditions. This means that we identify $\Omega$ with the two-dimensional torus. 

\subsection{Convergence tests}

We consider $N=2$ components, the final time $T=0.1$, and the stress tensor $\mathbb{S}(\na\bm{u})=\nu(\na\bm{u}^T+\na\bm{u}) + \lambda(\divM\bm{u})\mathbb{I}$ with $\nu=10^{-3}$ and $\lambda=0$. The internal energy density is chosen as
$f(\vec\rho)=\sum_{i=1}^2\rho_i\log(\rho_i/\rho)$, and the mobility matrix is given by \eqref{1.M}, i.e. $M_{ij}(\pvec{\rho})=\rho_i\delta_{ij} - \rho_i\rho_j/\rho$. The initial data equals
\begin{align*}
  & \rho_1^0(x,y) = 1+0.8\sin(4\pi x)\sin(2\pi y), \quad
  \rho_2^0(x,y) = V_2^{-1}(1-V_1\rho_1^0(x,y)), \\
  & \bm{u}^0(x,y) = \big(-\sin(\pi x)^2\sin(2\pi y),
  \sin(2\pi x)\sin(\pi y)^2\big)^T.
\end{align*}
Since no exact solution is available, we compute the error by using the reference solution at the finest space resolution. Given the mesh sizes $h_k=2^{-k-1}$ for $k=1,\ldots,8$ and $h_{\rm ref}=2^{-8}$, we define the error quantities 
\begin{align*}
  & \mbox{err}(\vec\rho) = \max_{\ell=1,\ldots,n} \sqrt{
  \sum_{i=1}^2\|\rho_{h_k,i}^\ell-\rho_{{\rm ref},i}^\ell
  \|_{L^2(\Omega)}^2}, \quad
  \mbox{err}(\vec\mu) = \sqrt{\tau\sum_{\ell=1}^{n}\sum_{i=1}^2
  \|\mu_{h_k,i}^\ell-\mu_{{\rm ref},i}^\ell\|_{L^2(\Omega)}^2}, \\
  & \mbox{err}(\bm{u}) = \max_{\ell=1,\ldots,n}
  \sqrt{\|\bm{u}_{h_k}^\ell-\bm{u}_{\rm ref}^\ell\|_{L^2(\Omega)}^2}, \qquad\quad
  \mbox{err}(p) = \sqrt{\tau\sum_{\ell=1}^{n}\|p_{h_k}^\ell-p_{\rm ref}^\ell
  \|_{L^2(\Omega)}^2}.
\end{align*}
We have chosen the time step size $\tau=10^{-3}$ such that $n=T/\tau=100$. Table \ref{table1_sqrt} presents the errors and the experimental order of convergence for $V_1=0.3$, $V_2=0.7$, while Table \ref{table2_sqrt} shows these values for the choice $V_1=V_2=0.5$. As expected, we observe second-order convergence for $(\vec\rho,\vec\mu,p)$ and second-order convergence for $\bm{u}$ (recall that the velocity is approximated by piecewise quadratic polynomials, for which third-order might be possible). In the second test case, the order of convergence for the velocity is larger than in the first test, which comes from the choice $V_1=V_2$, i.e., we have solved the incompressible Navier--Stokes equations. In fact, for the incompressible Navier--Stokes system, one can show optimal third-order convergence, while in the quasi-incompressible case, it is unclear if that order can be proven. In both tests, the partial and total masses and the point wise quasi-incompressibility constraint are conserved up to errors of order $10^{-14}$. The time step is chosen sufficiently small for two reasons. The first is to ensure that the temporal error does not dominate the spatial error. The second is to guarantee convergence of Newton iterations for all meshes. The latter is related to the existence of (unique) discrete solutions.


\begin{table}[htbp!]
	\centering
	\small
	\caption{$L^2(\Omega)$ errors and experimental order of convergence for $V_1=0.3$, $V_2=0.7$.} 
	\begin{tabular}{|c||c|c|c|c|c|c|c|c|}
		\hline
		$ k $ & $\text{err}(\vec\rho)$ &  eoc & $\text{err}(\vec\mu)$ & eoc 
		  & $\text{err}(\bm{u})$ &  eoc & $\text{err}(p)$ & eoc   \\
		\hline
		1 & $2.30\cdot 10^{-1}$ & --       &  $5.52\cdot 10^{-2}$ & --    
		  & $3.77\cdot 10^{-1}$ & --       &  $7.53\cdot 10^{-2}$ & -- \\
		2 & $2.81\cdot 10^{-1}$ & $-$0.29    &  $6.89\cdot 10^{-2}$ & $-$0.32 
		  & $7.76\cdot 10^{-2}$ & \phm2.28 &  $9.39\cdot 10^{-2}$ & $-$0.32 \\
		3 & $8.27\cdot 10^{-2}$ & \phm1.77 &  $2.48\cdot 10^{-2}$ & \phm1.48 
		  & $3.85\cdot 10^{-2}$ & \phm1.01 &  $3.15\cdot 10^{-2}$ & \phm1.58\\
		4 & $1.94\cdot 10^{-2}$ & \phm2.09 &  $6.68\cdot 10^{-3}$ & \phm1.89
	      & $1.40\cdot 10^{-2}$ & \phm1.45 &  $8.12\cdot 10^{-3}$ & \phm1.96 \\
		5 & $4.86\cdot 10^{-3}$ & \phm2.00 &  $1.69\cdot 10^{-3}$ & \phm1.99 
          & $3.74\cdot 10^{-3}$ & \phm1.91 &  $2.04\cdot 10^{-3}$ & \phm1.99 \\
		6 & $1.19\cdot 10^{-3}$ & \phm2.02 &  $4.11\cdot 10^{-4}$ & \phm2.04 
		  & $6.82\cdot 10^{-4}$ & \phm2.46 &  $4.99\cdot 10^{-4}$ & \phm2.04 \\
		7 & $3.04\cdot 10^{-4}$ & \phm1.97 &  $1.06\cdot 10^{-4}$ & \phm1.96 
		  & $1.25\cdot 10^{-4}$ & \phm2.45 &  $1.28\cdot 10^{-4}$ & \phm1.96 \\
		\hline
	\end{tabular}
	\label{table1_sqrt}
\end{table}

\begin{table}[htbp!]
	\centering
	\small
	\caption{$L^2(\Omega)$ errors and experimental order of convergence for $V_1=V_2=0.5$.} 
	\begin{tabular}{|c||c|c|c|c|c|c|c|c|}
		\hline
		$ k $ & $\text{err}(\vec\rho)$ &  eoc & $\text{err}(\vec\mu)$ & eoc 
		  & $\text{err}(\bm{u})$ &  eoc & $\text{err}(p)$ & eoc   \\
		\hline
		1 & $2.99\cdot 10^{-1}$ & --       &  $9.02\cdot 10^{-2}$ & --    
		  & $3.70\cdot 10^{-1}$ & --       &  $1.21\cdot 10^{-1}$ & -- \\
		2 & $3.64\cdot 10^{-1}$ & $-$0.29  &  $8.26\cdot 10^{-2}$ & \phm0.13 
		  & $2.20\cdot 10^{-2}$ & \phm4.08 &  $1.05\cdot 10^{-1}$ & \phm0.21 \\
		3 & $1.07\cdot 10^{-1}$ & \phm1.77 &  $1.47\cdot 10^{-2}$ & \phm2.49 
		  & $5.96\cdot 10^{-3}$ & \phm1.89 &  $6.64\cdot 10^{-3}$ & \phm3.98\\
		4 & $2.52\cdot 10^{-2}$ & \phm2.09 &  $4.14\cdot 10^{-3}$ & \phm1.83
	      & $2.50\cdot 10^{-3}$ & \phm1.25 &  $6.91\cdot 10^{-4}$ & \phm3.27 \\
		5 & $6.22\cdot 10^{-3}$ & \phm2.02 &  $1.05\cdot 10^{-3}$ & \phm1.98 
          & $6.56\cdot 10^{-4}$ & \phm1.93 &  $1.19\cdot 10^{-4}$ & \phm2.55 \\
		6 & $1.53\cdot 10^{-3}$ & \phm2.02 &  $2.60\cdot 10^{-4}$ & \phm2.01 
		  & $1.20\cdot 10^{-4}$ & \phm2.45 &  $2.71\cdot 10^{-5}$ & \phm2.12 \\
		7 & $3.95\cdot 10^{-4}$ & \phm1.95 &  $6.72\cdot 10^{-5}$ & \phm1.95 
		  & $1.70\cdot 10^{-5}$ & \phm2.82 &  $6.93\cdot 10^{-6}$ & \phm1.97 \\
		\hline
	\end{tabular}
	\label{table2_sqrt}
\end{table}

\subsection{Three-component experiment I}\label{subsec:experiment}

We present a numerical test for $N=3$ components based on a similar test in \cite[Sec.~5.2]{CEM24}. We choose the stress tensor as in the previous section with $\nu=10^{-2}$ and $\lambda=0$, the internal energy density $f(\vec\rho)=\sum_{i=1}^3\rho_i\log(\rho_i/\rho)$, and the mobility matrix \eqref{1.M}, i.e. $M_{ij}(\pvec{\rho})=\rho_i\delta_{ij} - \rho_i\rho_j/\rho$. The specific volumes are $V_1=V_2=0.35$ and $V_3=0.8$. The final time is set to $T=0.5$, the time step size equals $\tau=10^{-4}$, and the maximal mesh size is $h=2^{-7}$. The initial data reads as
\begin{align*}
  \rho_1^0(x,y) &= 0.2 + 0.9\chi(x,y,y\in[0.1,0.2)\cup (0.8,0.9]) +0.9\chi(x,y,x\in[0.1,0.2)\cup (0.8,0.9])  \\
  &\phantom{xx}- 0.9\chi(x,y,(x,y)\in ((0.1,0.2]\times ([0.1,0.2)\cup[0.8,0.9))) \\
  &\phantom{xx}- 0.9\chi(x,y,(x,y)\in ([0.1,0.2)\cup[0.8,0.9))\times(0.1,0.2]), \\
  \rho_2^0(x,y) &= 0.2 
  + 0.9\chi(x,y,(x,y)\in\{(x-0.5)^2+(y-0.5)^2-0.25^2\leq 0\}), \\
  \rho_3^0(x,y) &= V_3^{-1}\big(1-V_1\rho_1(x,y) - V_2\rho_3(x,y)
  \big), \\
  \bm{u}(x,y) &= \big(-\sin(\pi x)^2\sin(2\pi y),
  \sin(\pi y)^2\sin(2\pi x)\big)^T,
\end{align*}  
where $\chi(x,y,y\in I)$ is the characteristic function on $[0,1]\times I$, and similarly for the other functions. Figures \ref{fig.rho}--\ref{fig.up} illustrate the time evolution of the partial mass densities $\rho_i$, total mass $\rho$, velocity magnitude $|\bm{u}|^2$, and pressure $p$. The coupling to the Navier--Stokes equations induces rotating shapes before the mass densities converge to the homogeneous steady state. The intermediate state is different from the numerical results of the Maxwell--Stefan equations presented in \cite[Sec.~5.2]{CEM24}, where the mass densities only diffuse and do not show any rotational effect. It should be emphasized that the model in \cite[Sec.~5.2]{CEM24} does not involve the Navier-Stokes equations. Hence, it is expected that the temporal evolution of the solutions differ.

\begin{figure}[ht]
\centering
\begin{tabular}{c@{}c@{}c@{}c@{}c@{}c}
	\multicolumn{6}{c}{\includegraphics[trim={24.0cm 67.5cm 28.0cm 0.2cm},clip,scale=0.14]{ 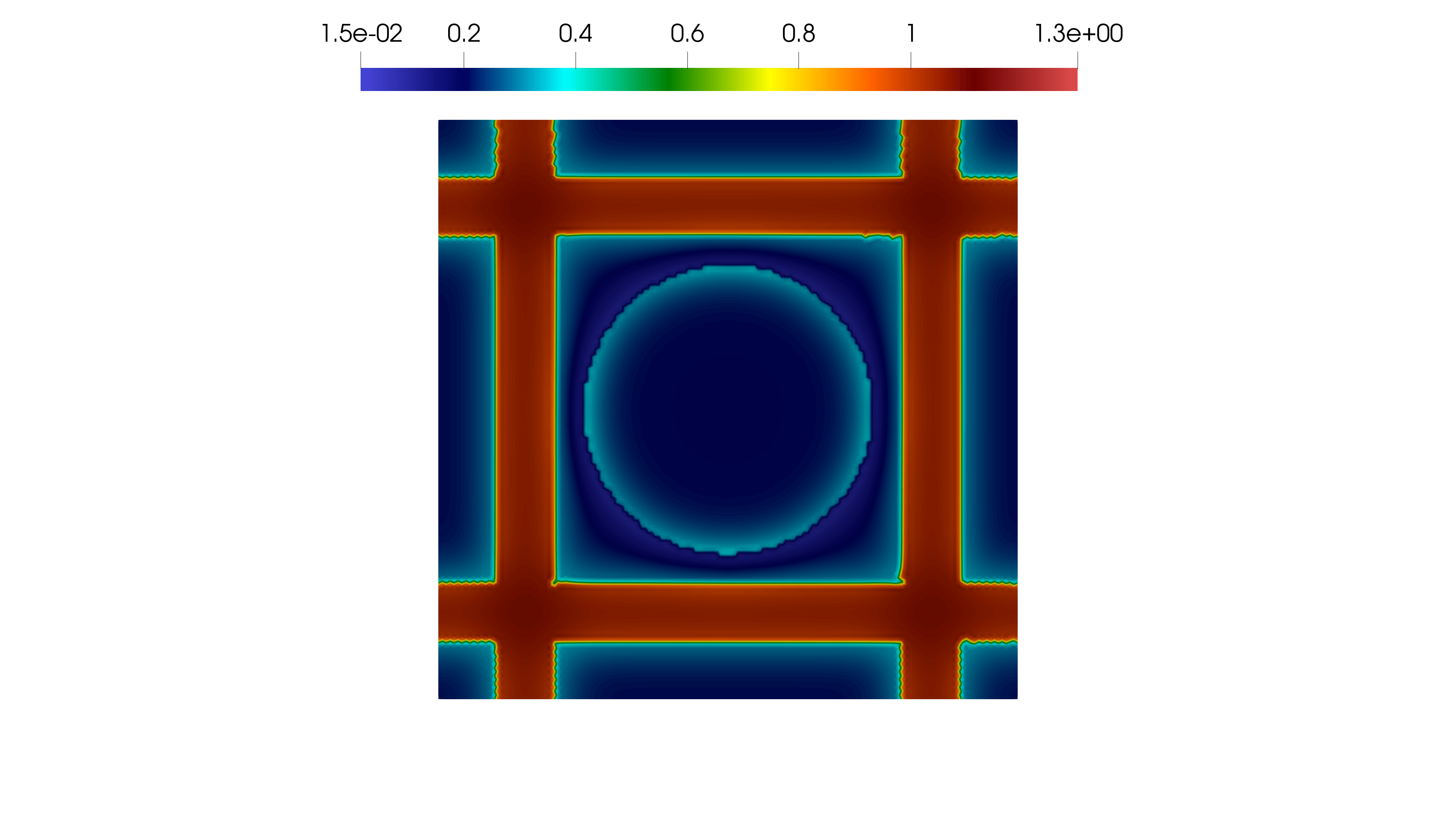}} \\[-0.25em]
	\includegraphics[trim={39cm 9.4cm 39.cm 9.5cm},clip,scale=0.042]{ rho_1.0010.png}
	&
	\includegraphics[trim={39cm 9.4cm 39.cm 9.5cm},clip,scale=0.042]{ 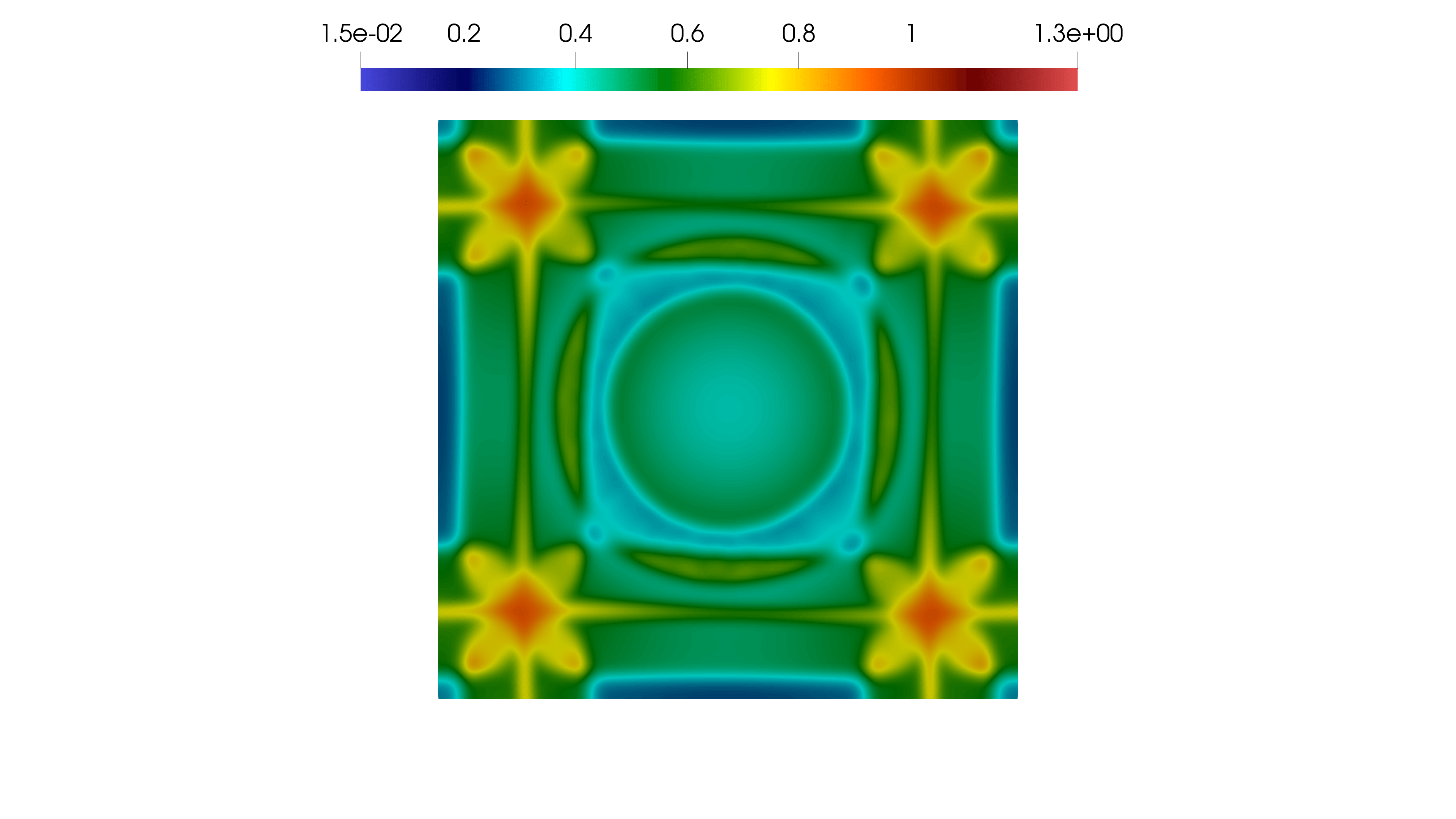} 
	&
	\includegraphics[trim={39cm 9.4cm 39.cm 9.5cm},clip,scale=0.042]{ 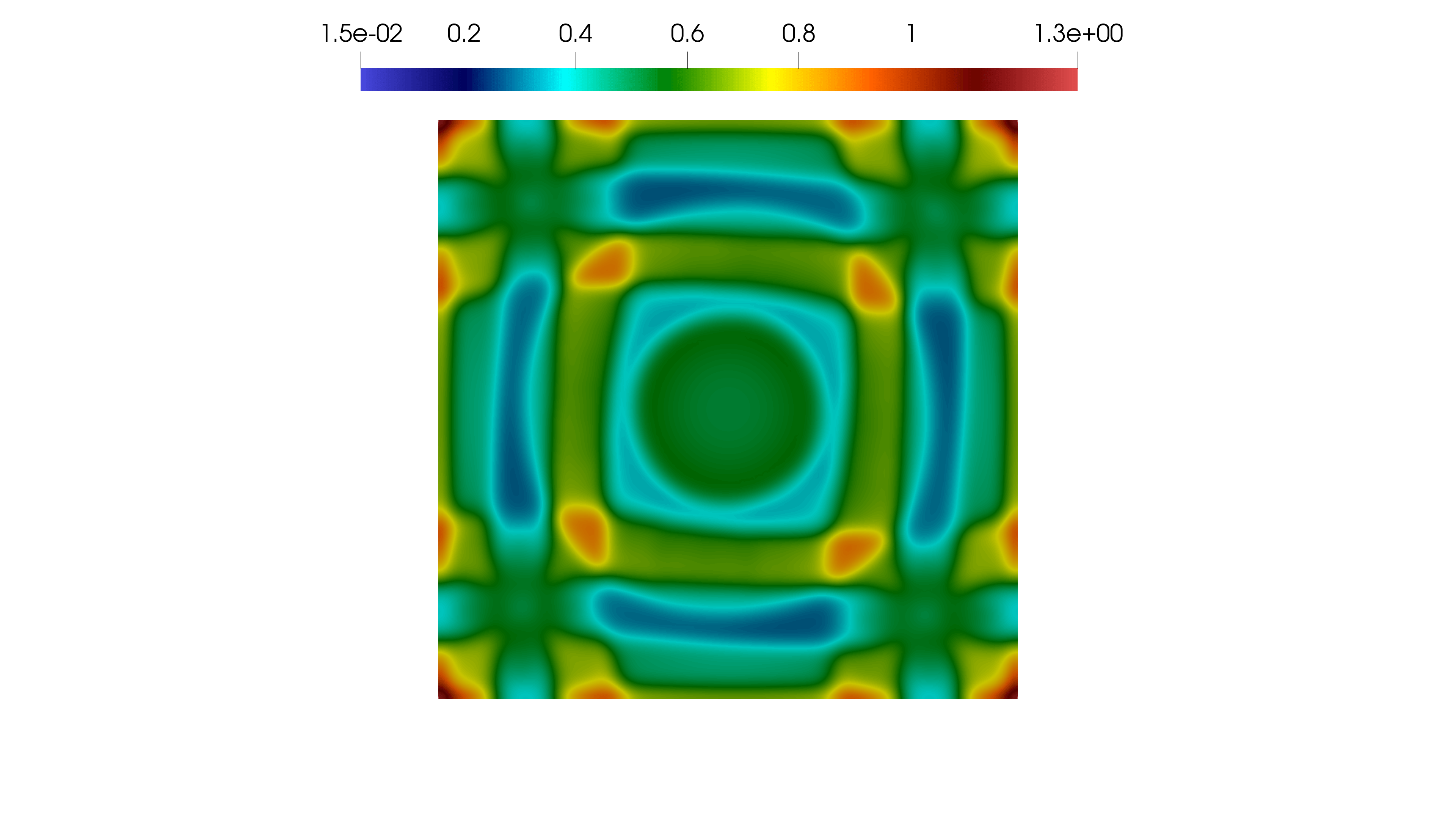}
	& 
	\includegraphics[trim={39cm 9.4cm 39.cm 9.5cm},clip,scale=0.042]{ 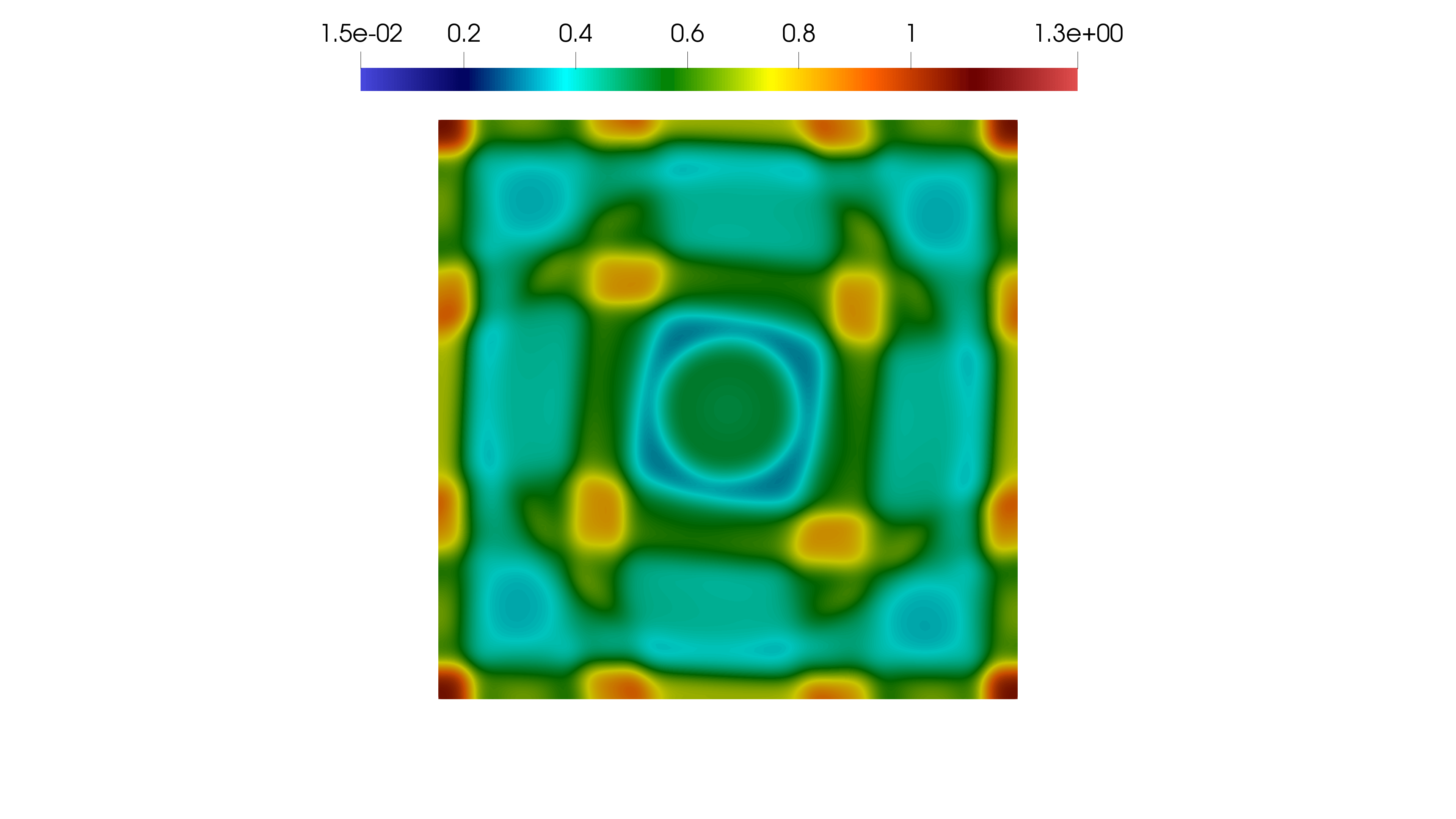} 
	&
	\includegraphics[trim={39cm 9.4cm 39.cm 9.5cm},clip,scale=0.042]{ 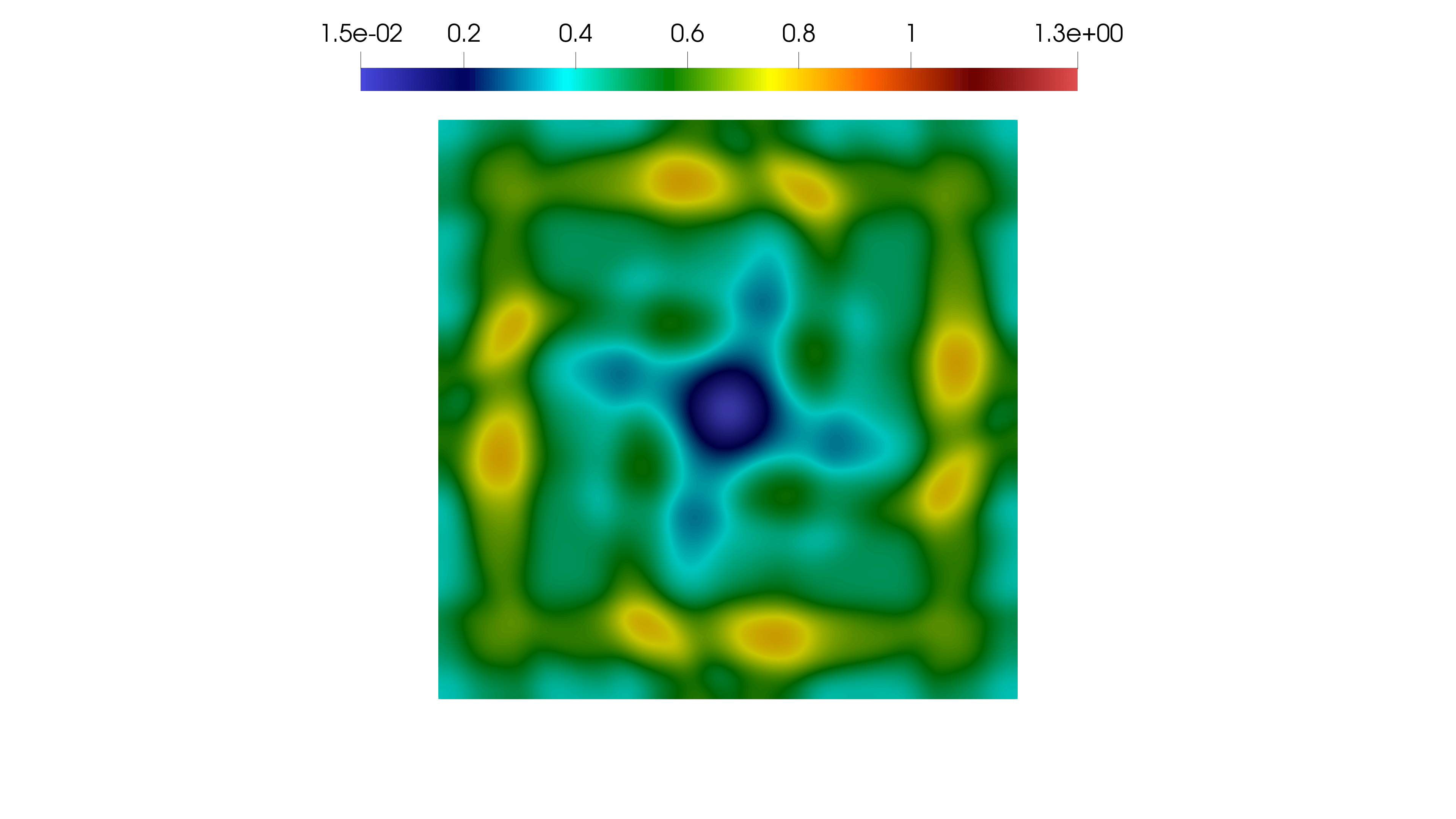}  
	&
	\includegraphics[trim={39cm 9.4cm 39.cm 9.5cm},clip,scale=0.042]{ 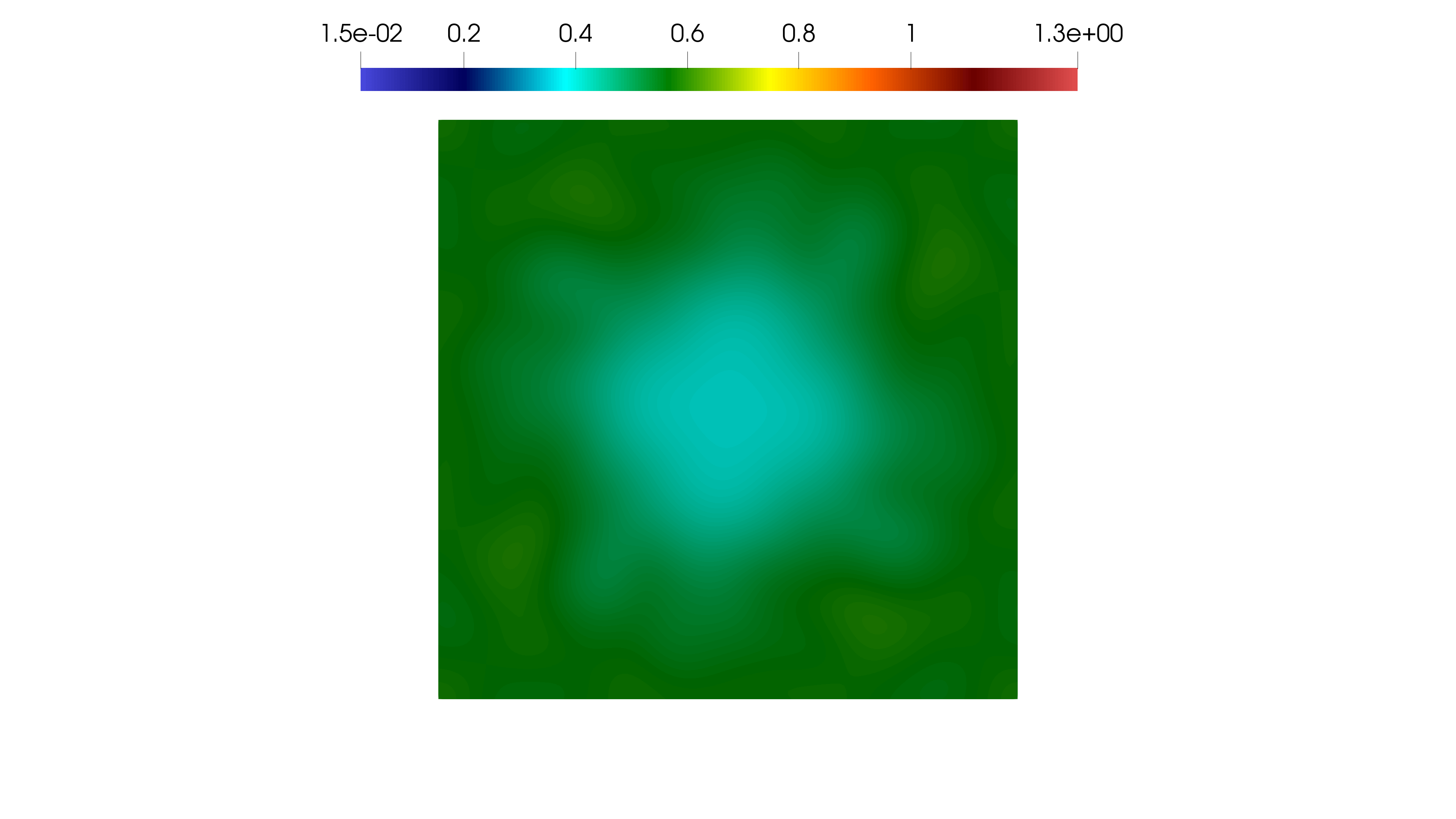} \\
   	\multicolumn{6}{c}{\includegraphics[trim={24.0cm 67.5cm 28.0cm 0.2cm},clip,scale=0.14]{ 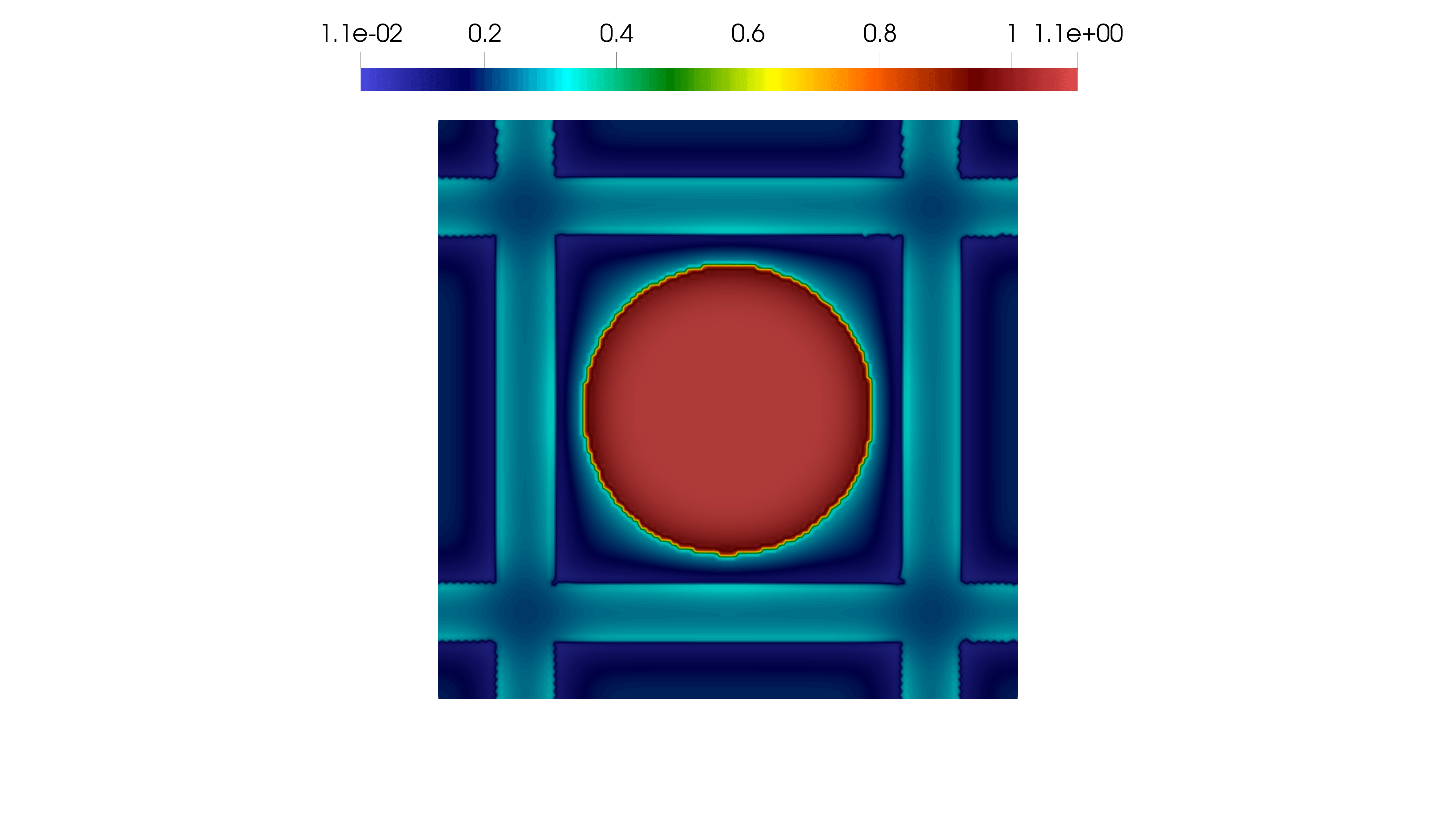}} \\[-0.5em]		\includegraphics[trim={39cm 9.4cm 39.cm 9.5cm},clip,scale=0.042]{ rho_2.0010.png} 
	&
	\includegraphics[trim={39cm 9.4cm 39.cm 9.5cm},clip,scale=0.042]{ 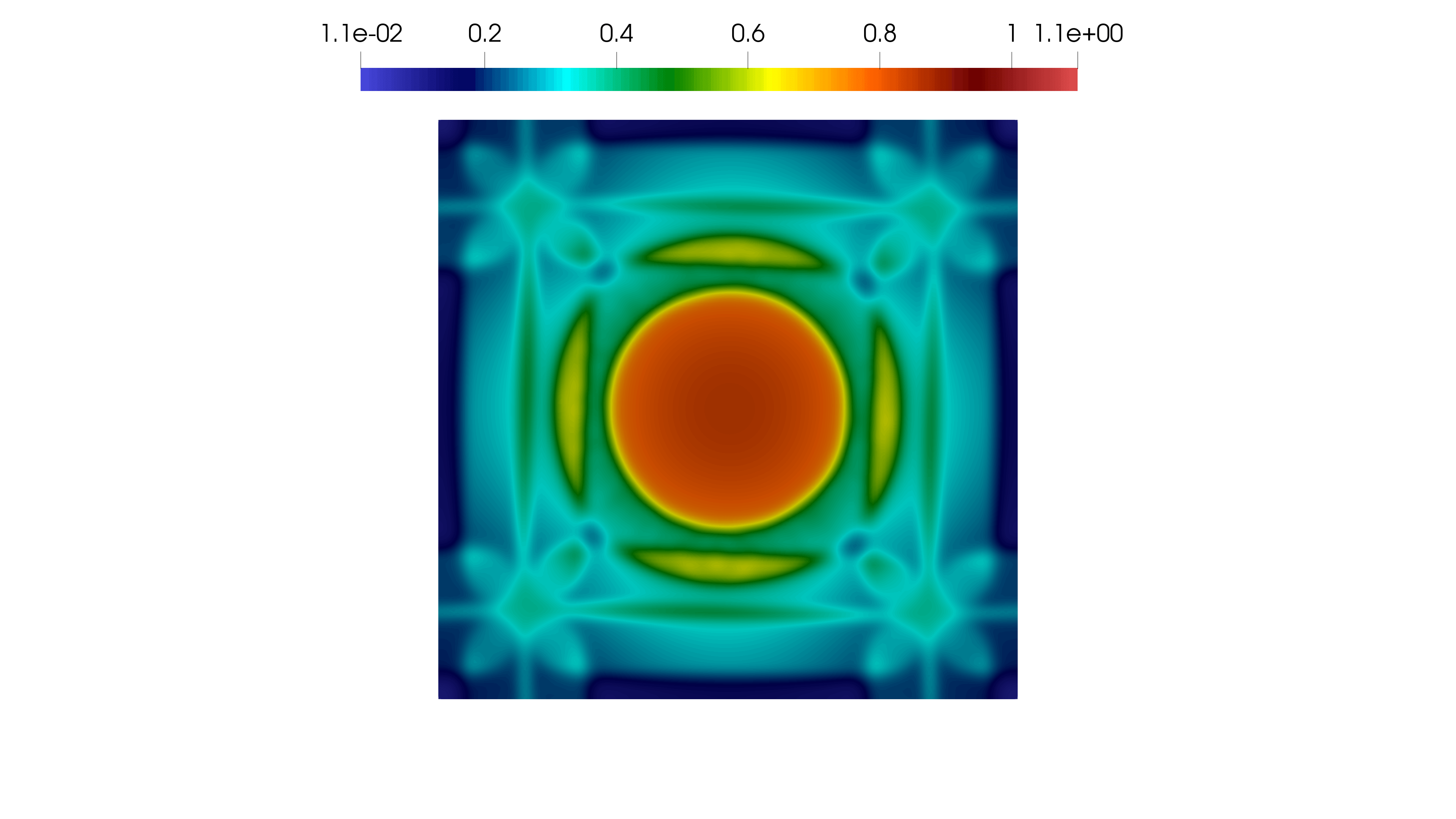} 
	&
	\includegraphics[trim={39cm 9.4cm 39.cm 9.5cm},clip,scale=0.042]{ 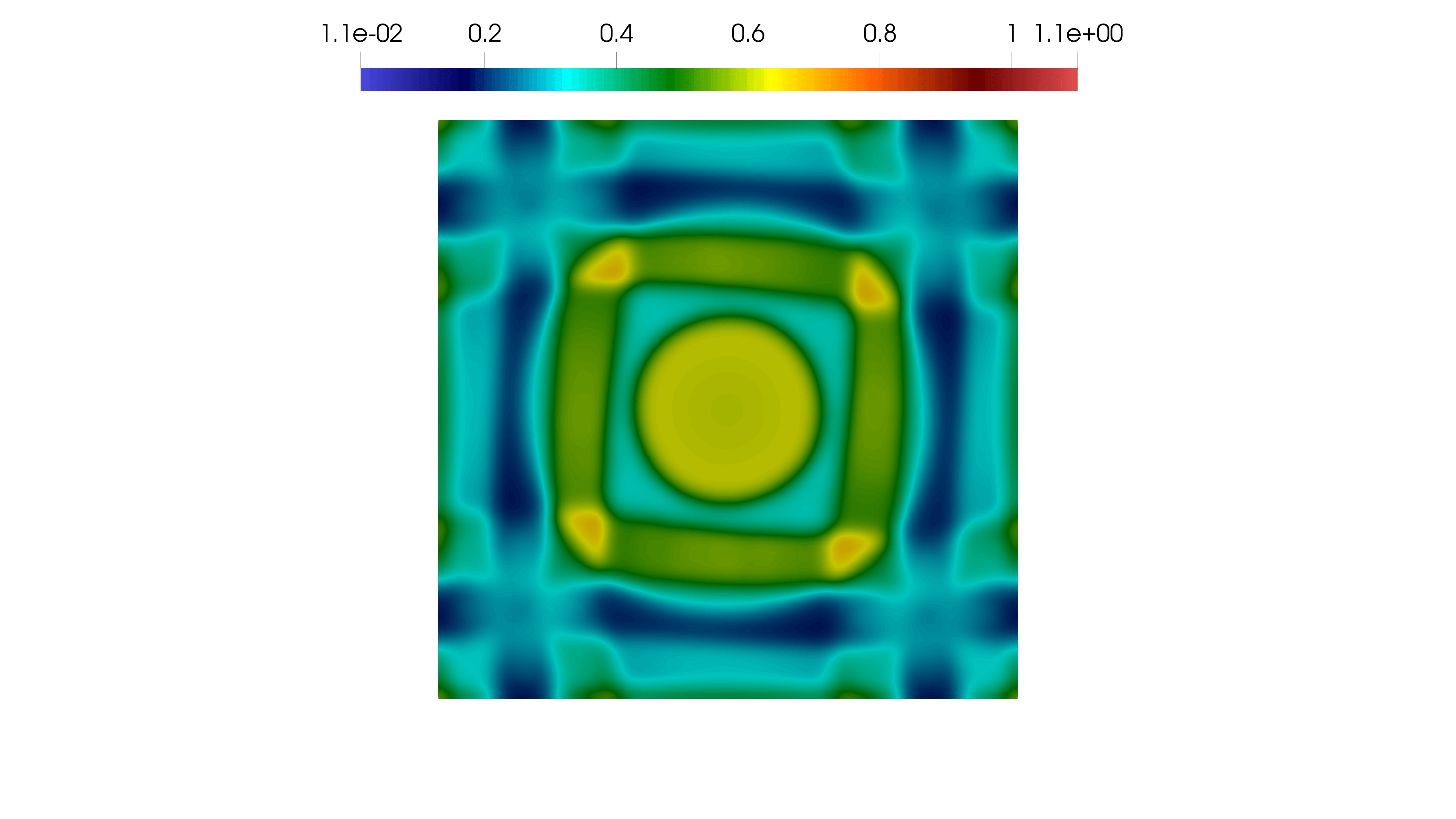}
	& 
	\includegraphics[trim={39cm 9.4cm 39.cm 9.5cm},clip,scale=0.042]{ 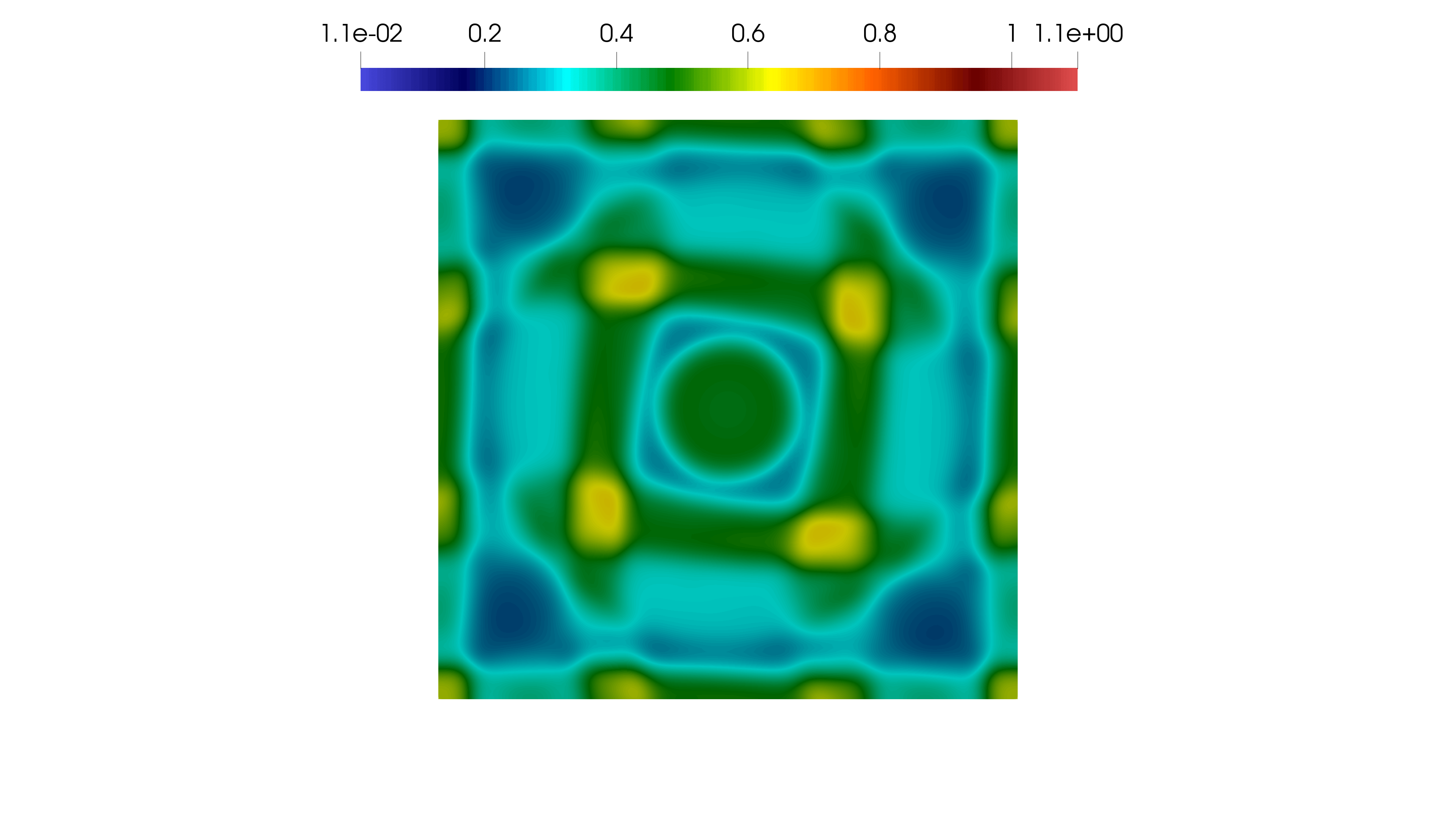} 
	&
	\includegraphics[trim={39cm 9.4cm 39.cm 9.5cm},clip,scale=0.042]{ 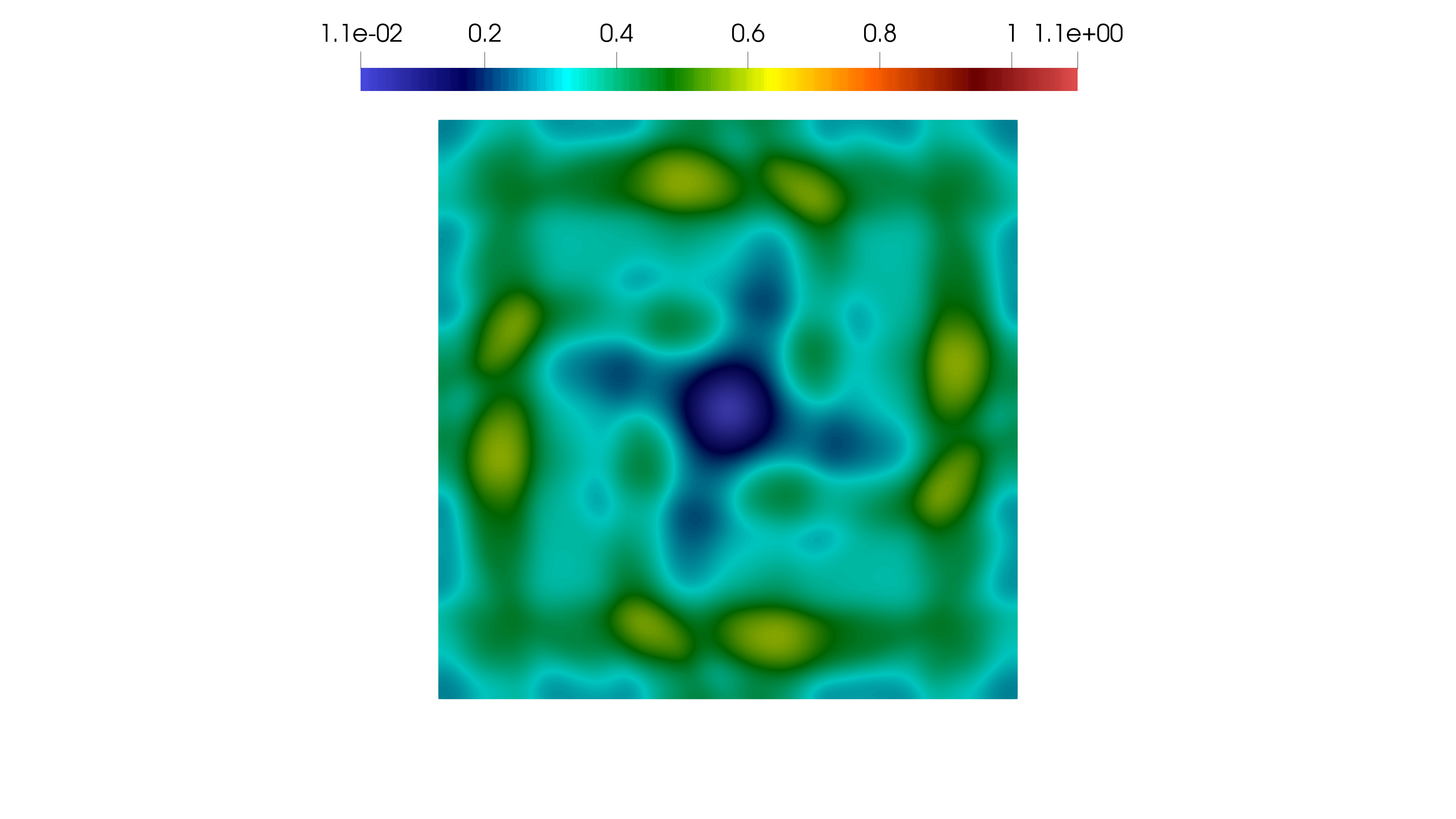}  
	&
	\includegraphics[trim={39cm 9.4cm 39.cm 9.5cm},clip,scale=0.042]{ 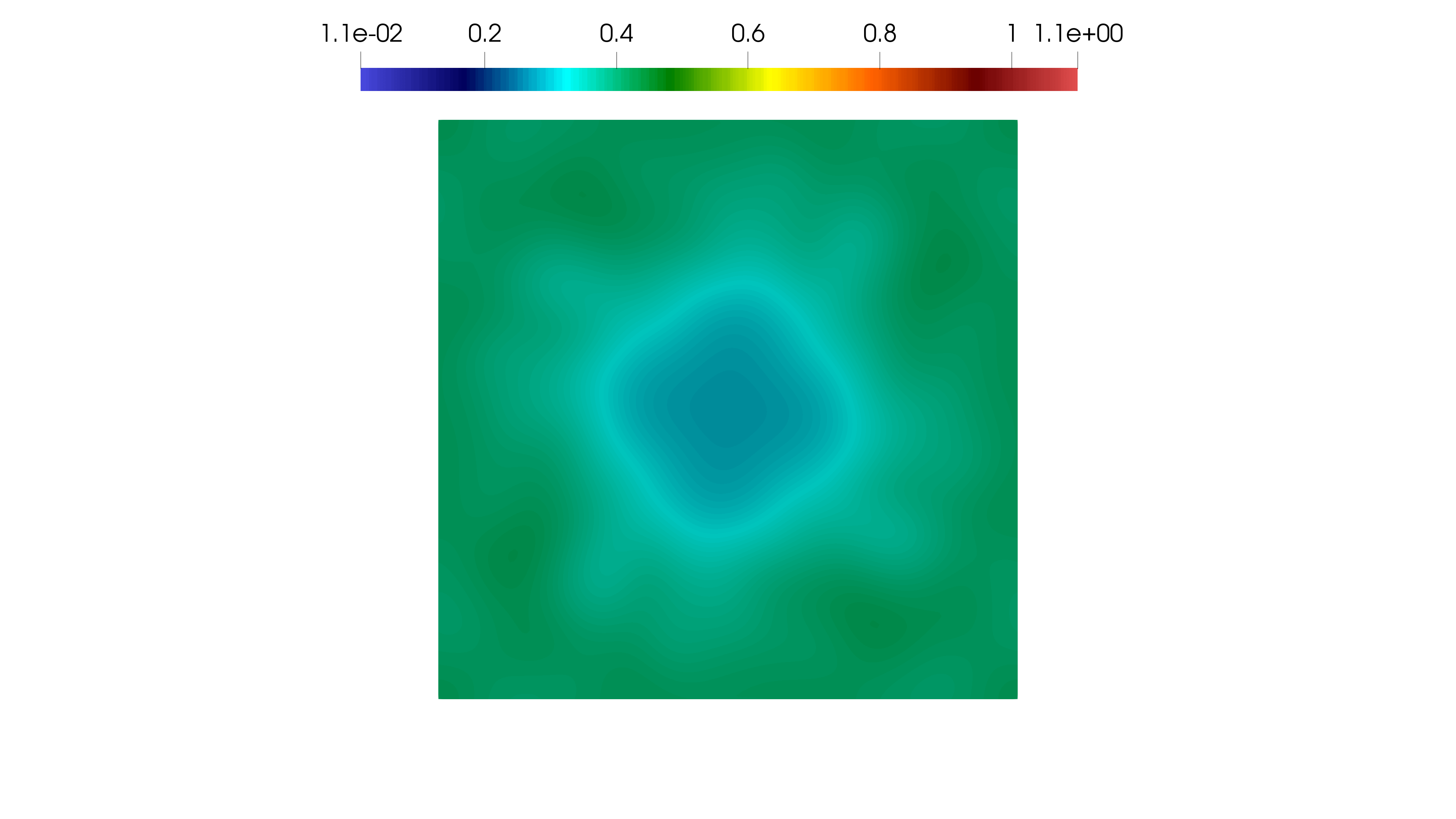}  \\
	\multicolumn{6}{c}{\includegraphics[trim={24.0cm 67.5cm 28.0cm 0.2cm},clip,scale=0.14]{ 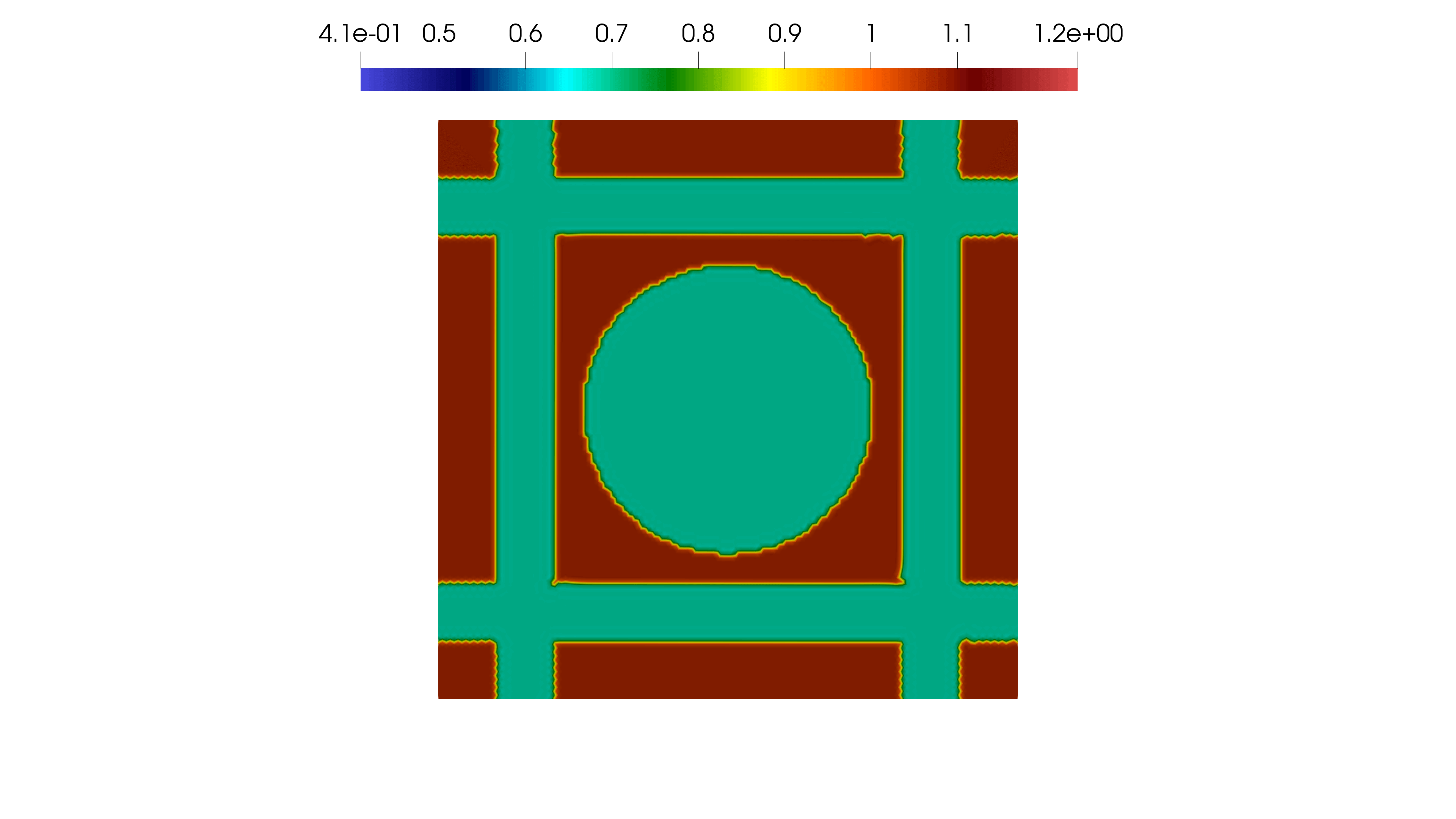}} \\[-0.5em]
	\includegraphics[trim={39cm 9.4cm 39.cm 9.5cm},clip,scale=0.042]{ rho_3.0010.png} 
	&
	\includegraphics[trim={39cm 9.4cm 39.cm 9.5cm},clip,scale=0.042]{ 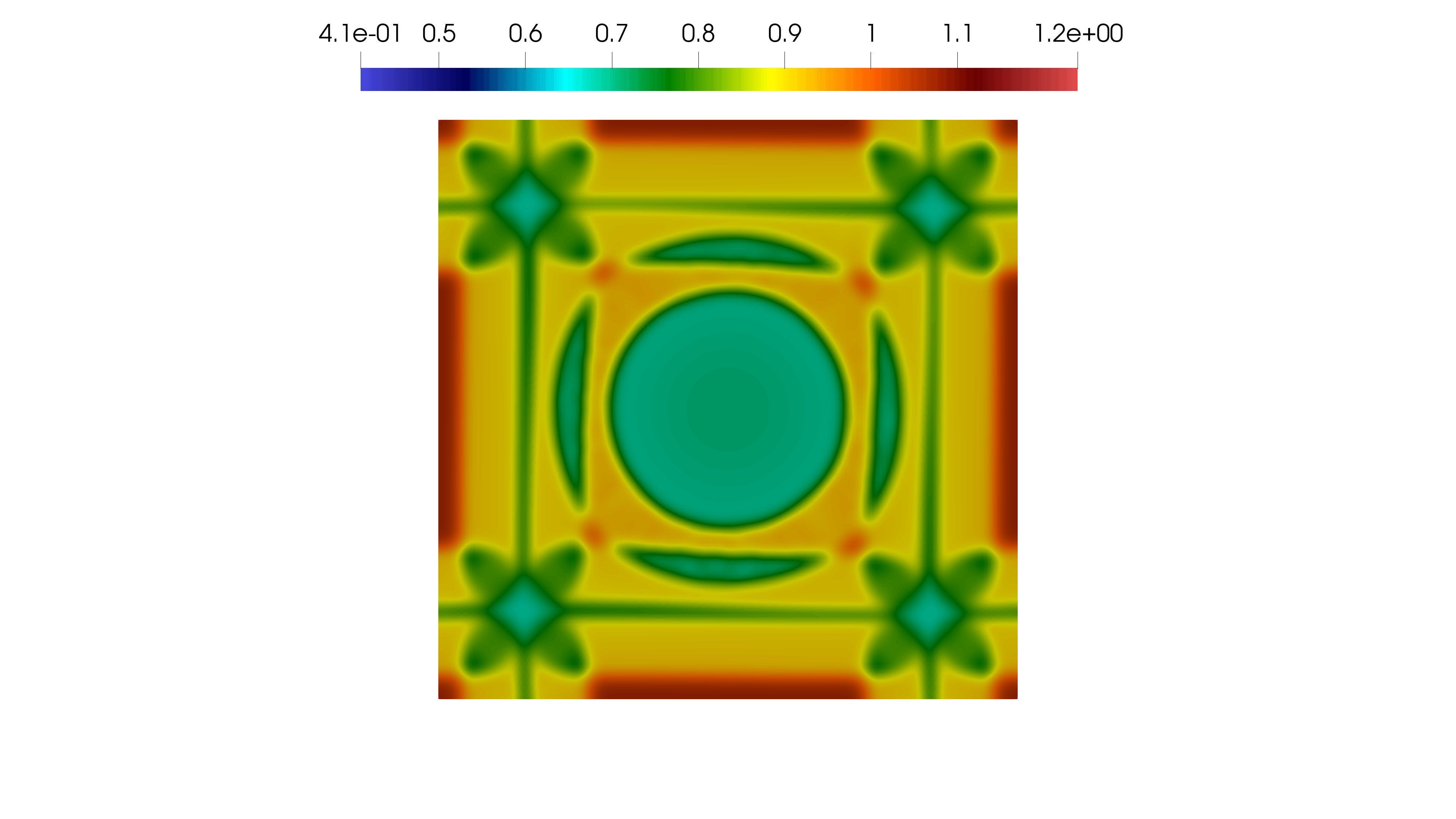} 
	&
	\includegraphics[trim={39cm 9.4cm 39.cm 9.5cm},clip,scale=0.042]{ 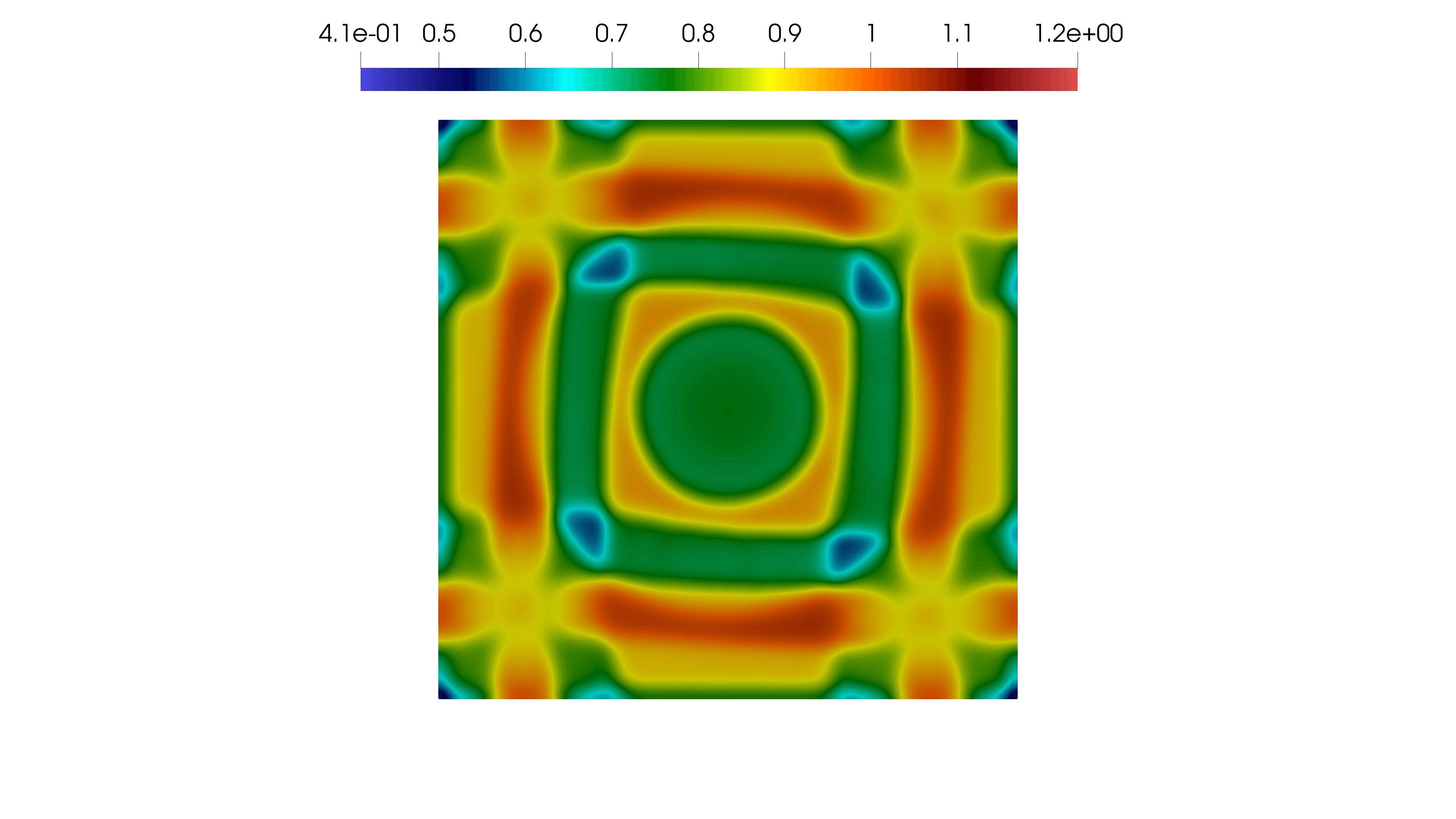}
	& 
	\includegraphics[trim={39cm 9.4cm 39.cm 9.5cm},clip,scale=0.042]{ 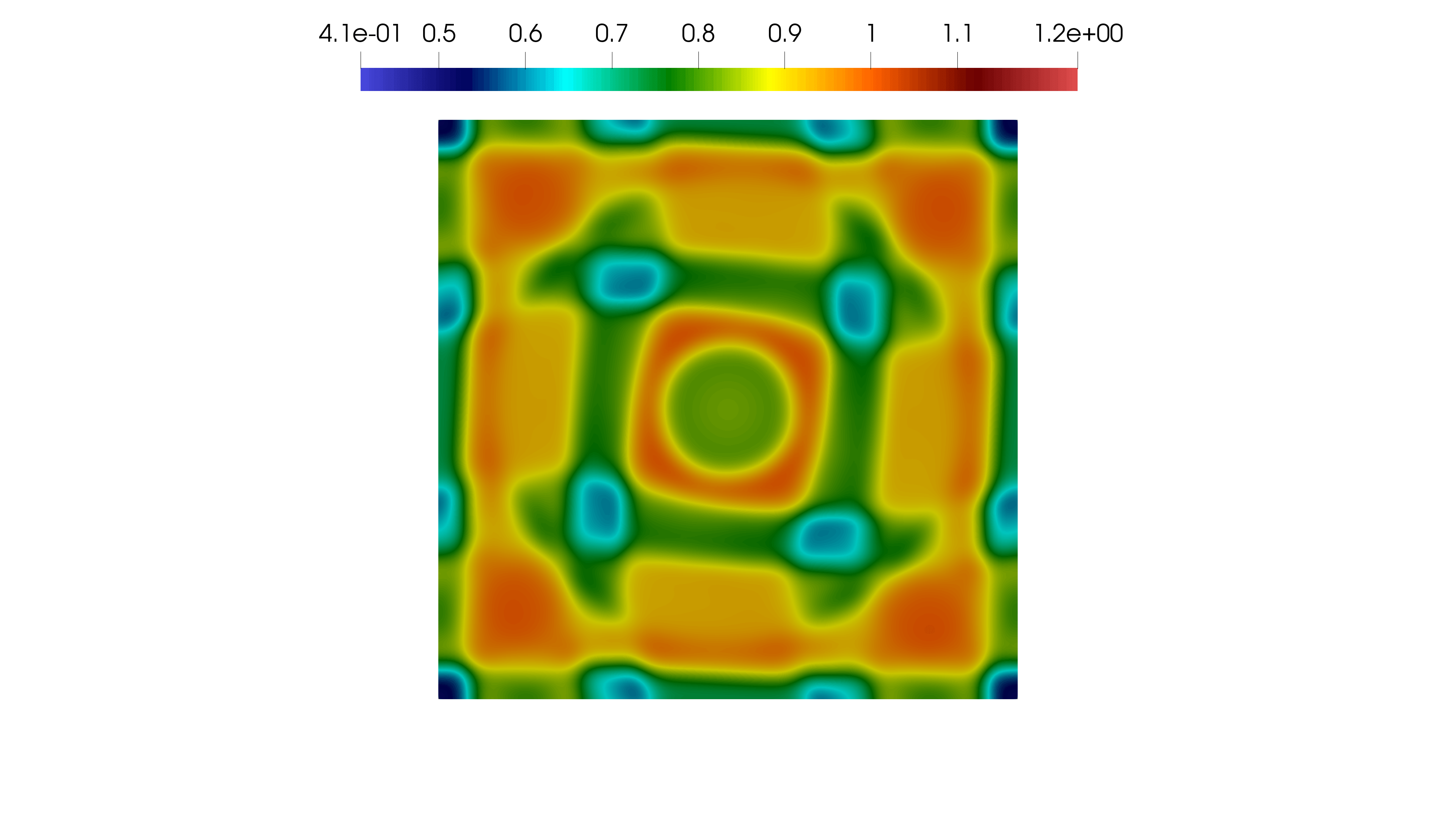} 
	&
	\includegraphics[trim={39cm 9.4cm 39.cm 9.5cm},clip,scale=0.042]{ 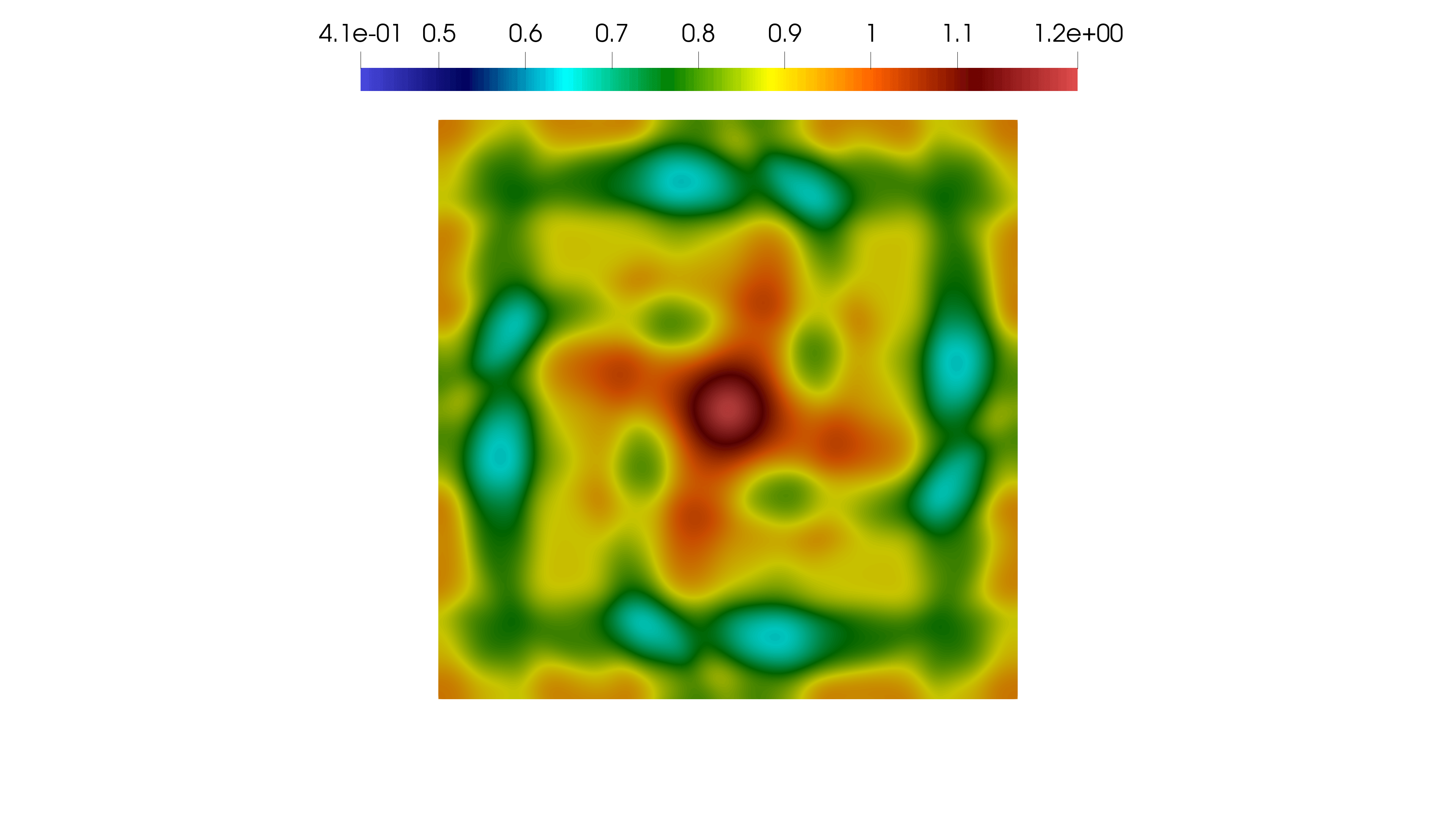}  
	&
	\includegraphics[trim={39cm 9.4cm 39.cm 9.5cm},clip,scale=0.042]{ 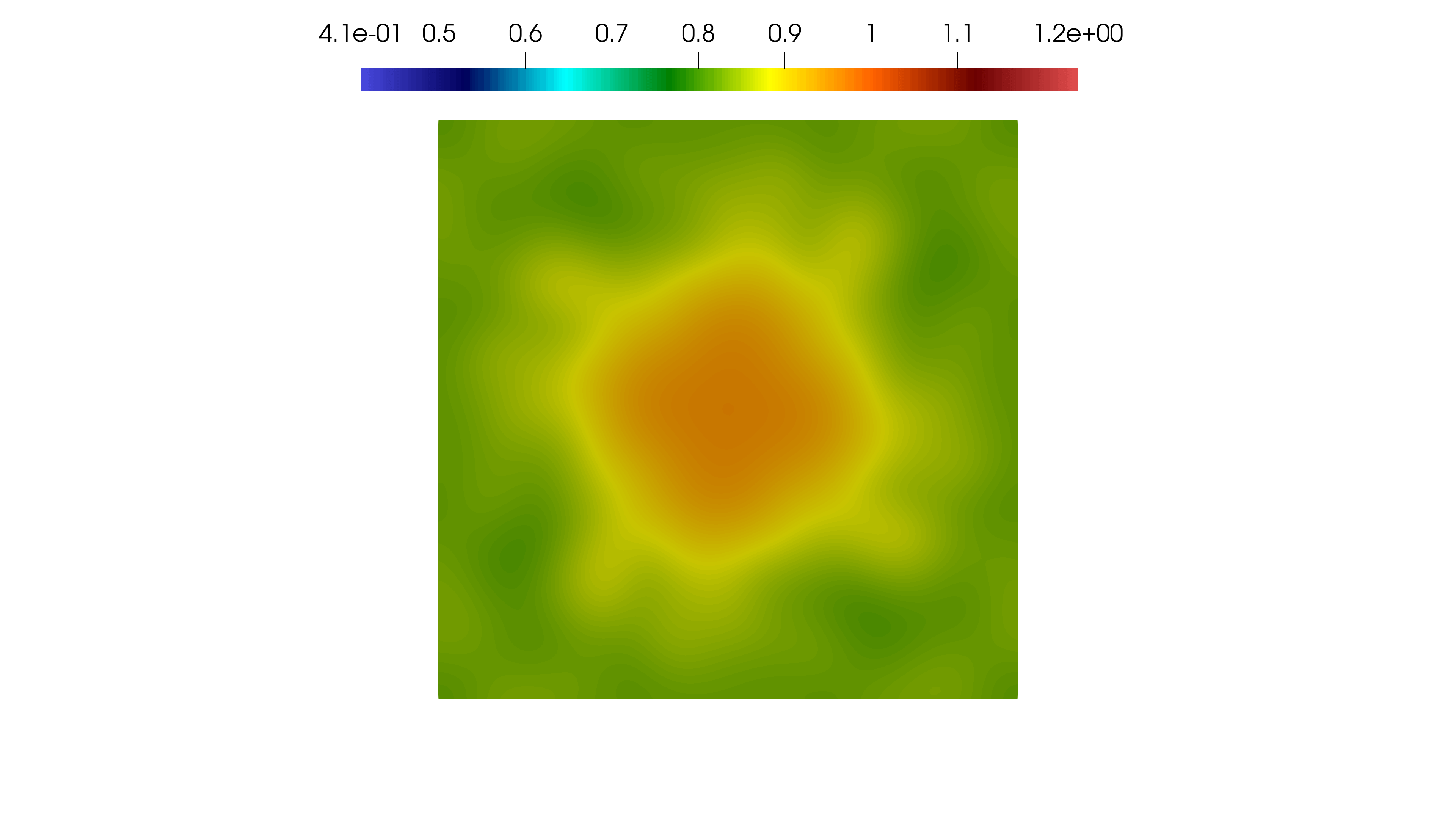}
\end{tabular}
\caption{Snapshots of the partial mass densities $\rho_1$ (upper row), $\rho_2$ (middle row), and $\rho_3$ (lower row) at times $t=0.001,0.02,0.04,0.06, 0.1,0.5$ for the experiment in Section \ref{subsec:experiment}.}
\label{fig.rho}
\end{figure}

\begin{figure}[ht]
\centering
\begin{tabular}{c@{}c@{}c@{}c@{}c@{}c}
	\multicolumn{6}{c}{\includegraphics[trim={24.0cm 67.5cm 24.0cm 0.2cm},clip,scale=0.14]{ 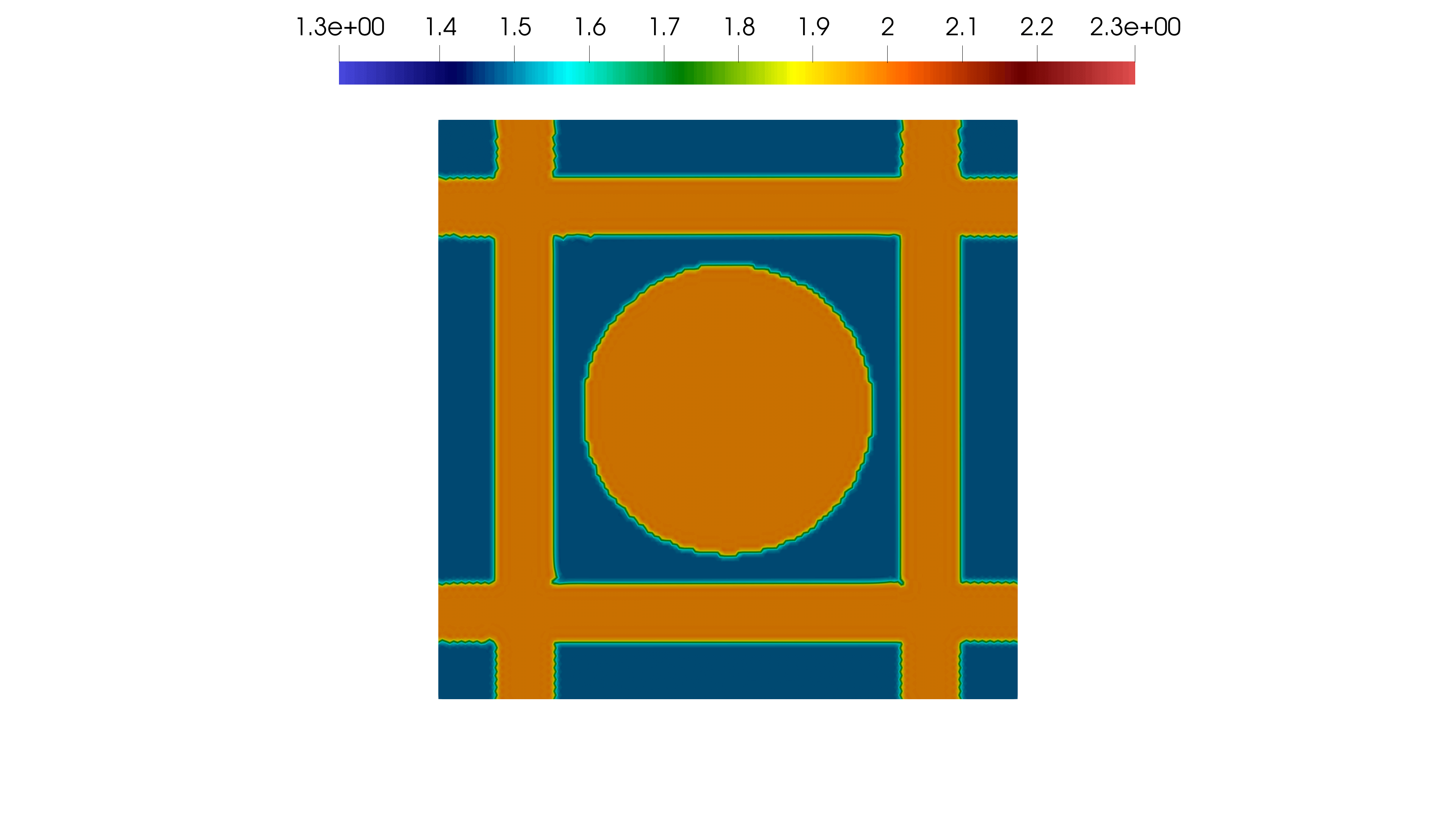}} \\[-0.5em]
	\includegraphics[trim={41cm 9.4cm 39.cm 9.5cm},clip,scale=0.042]{ rho_full.0010.png} 
	&
	\includegraphics[trim={39cm 9.4cm 39.cm 9.5cm},clip,scale=0.042]{ 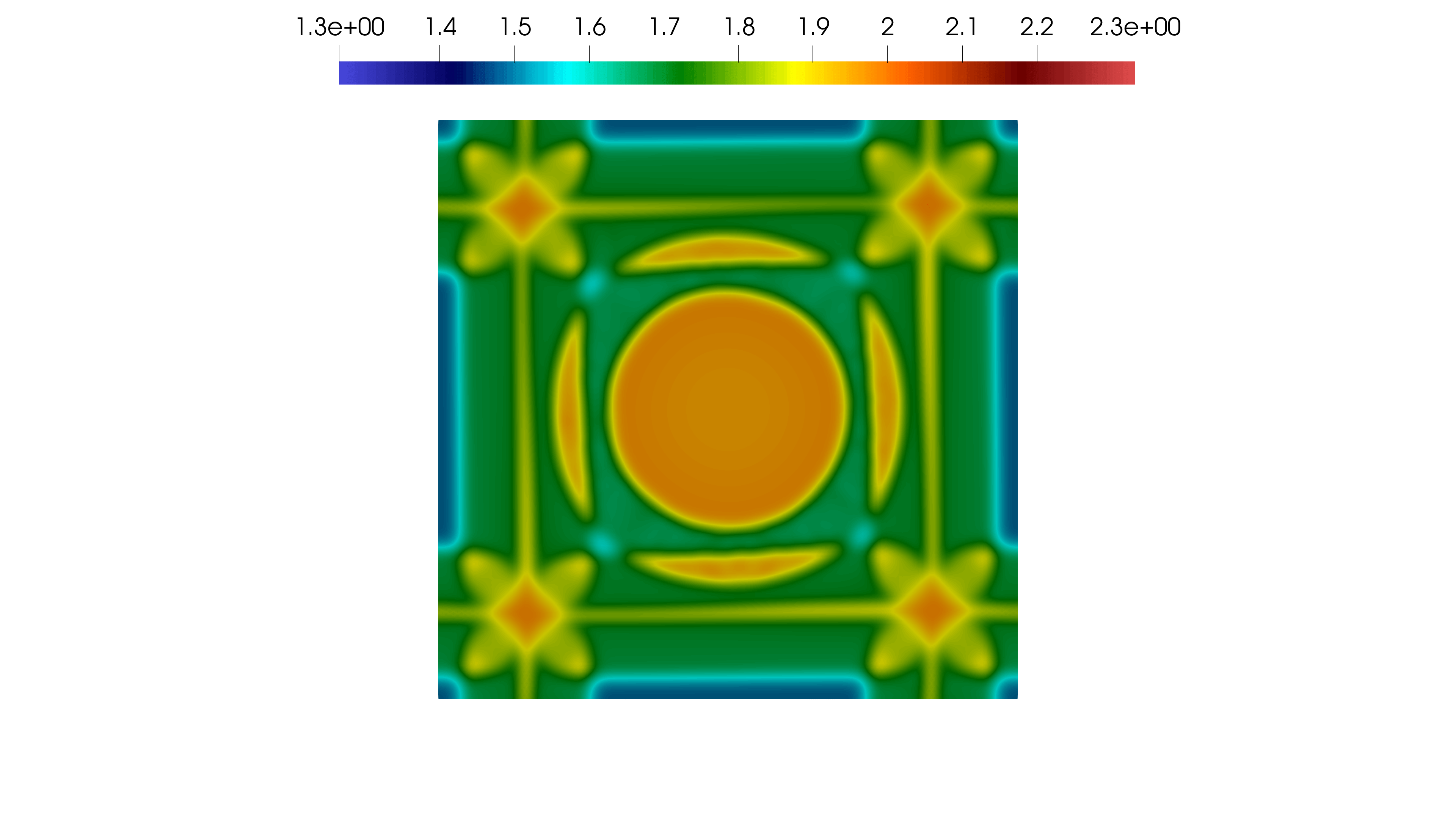} 
	&
	\includegraphics[trim={39cm 9.4cm 39.cm 9.5cm},clip,scale=0.042]{ 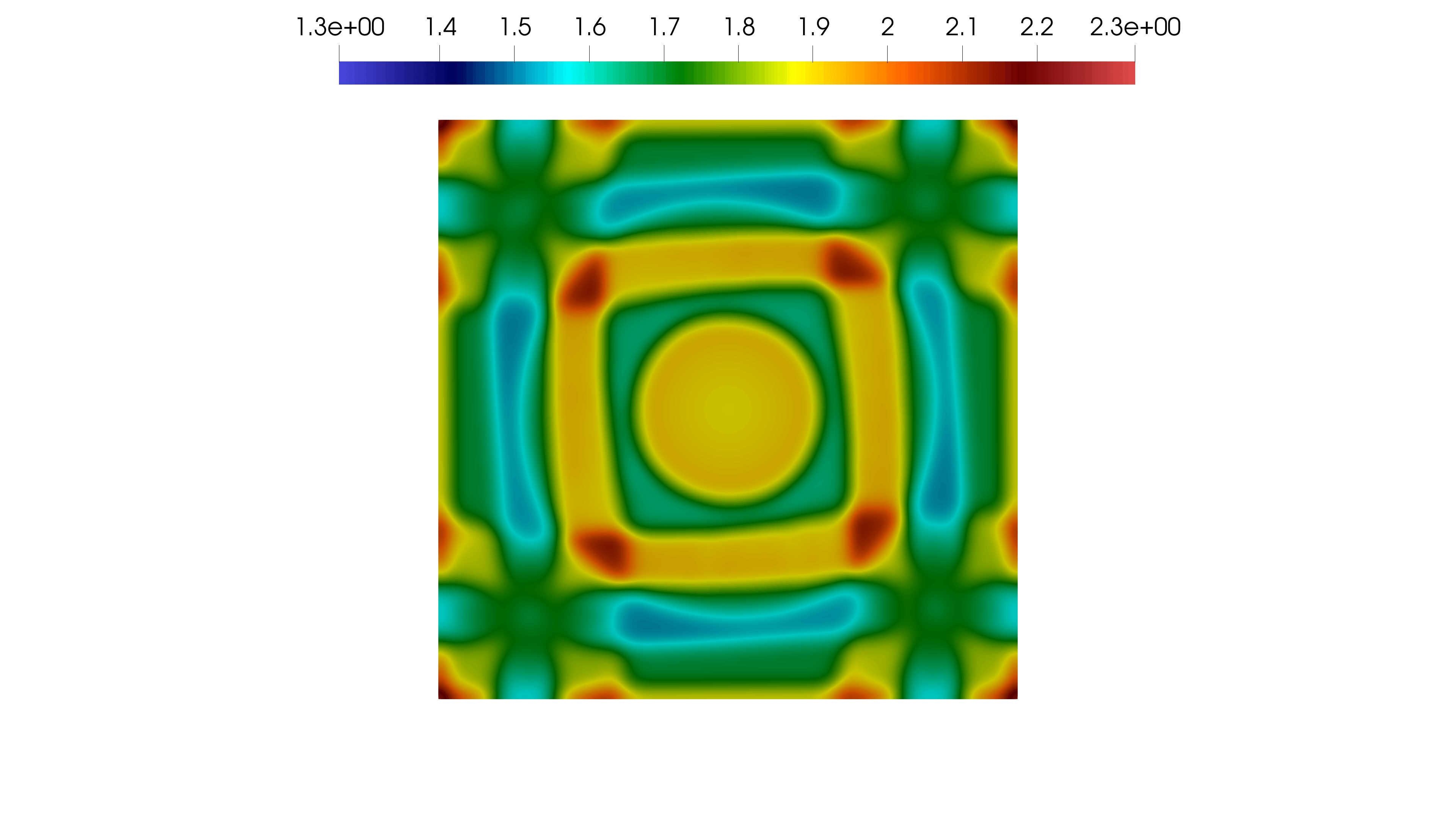}
	& 
	\includegraphics[trim={39cm 9.4cm 39.cm 9.5cm},clip,scale=0.042]{ 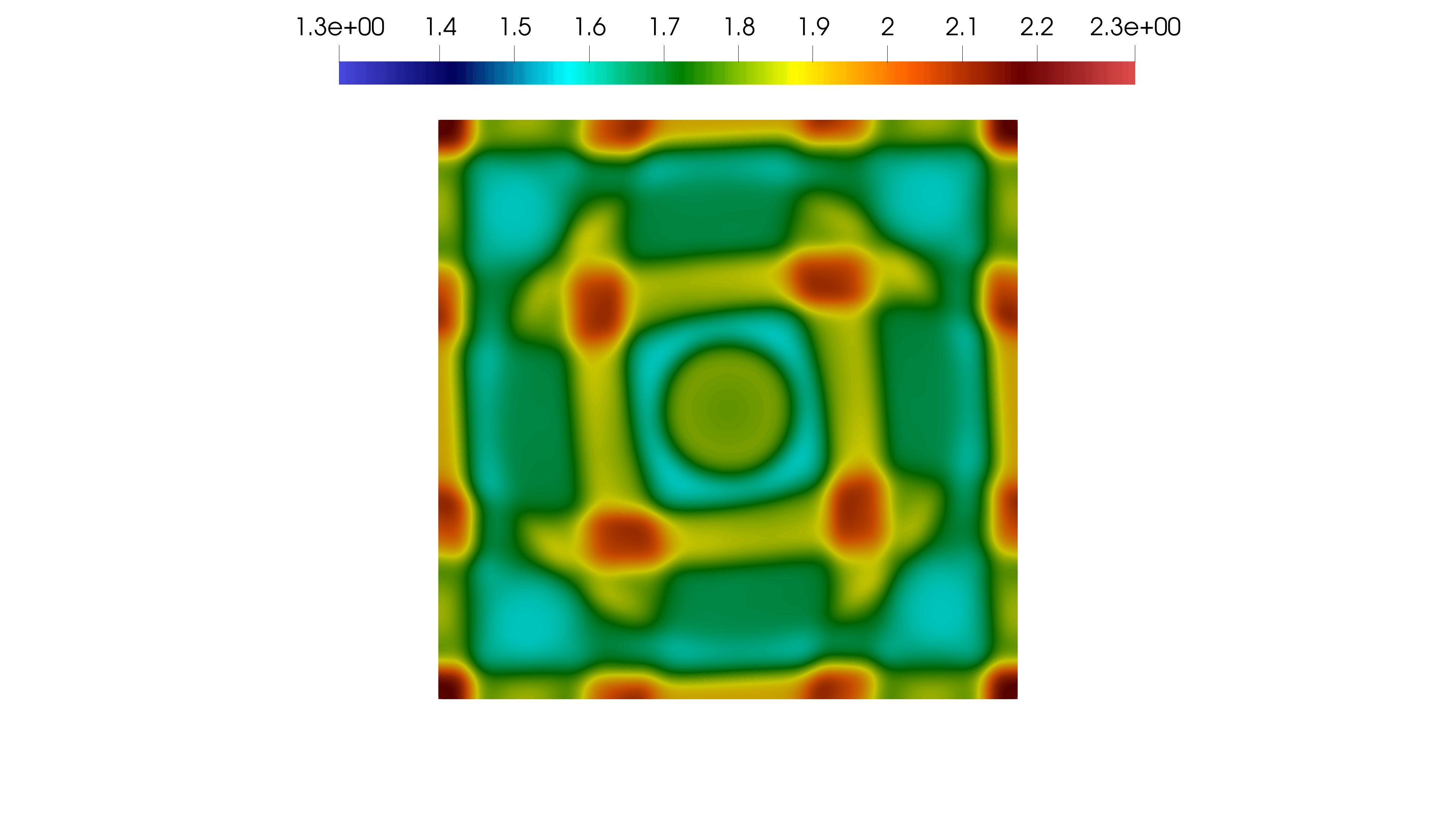} 
	&
	\includegraphics[trim={39cm 9.4cm 39.cm 9.5cm},clip,scale=0.042]{ 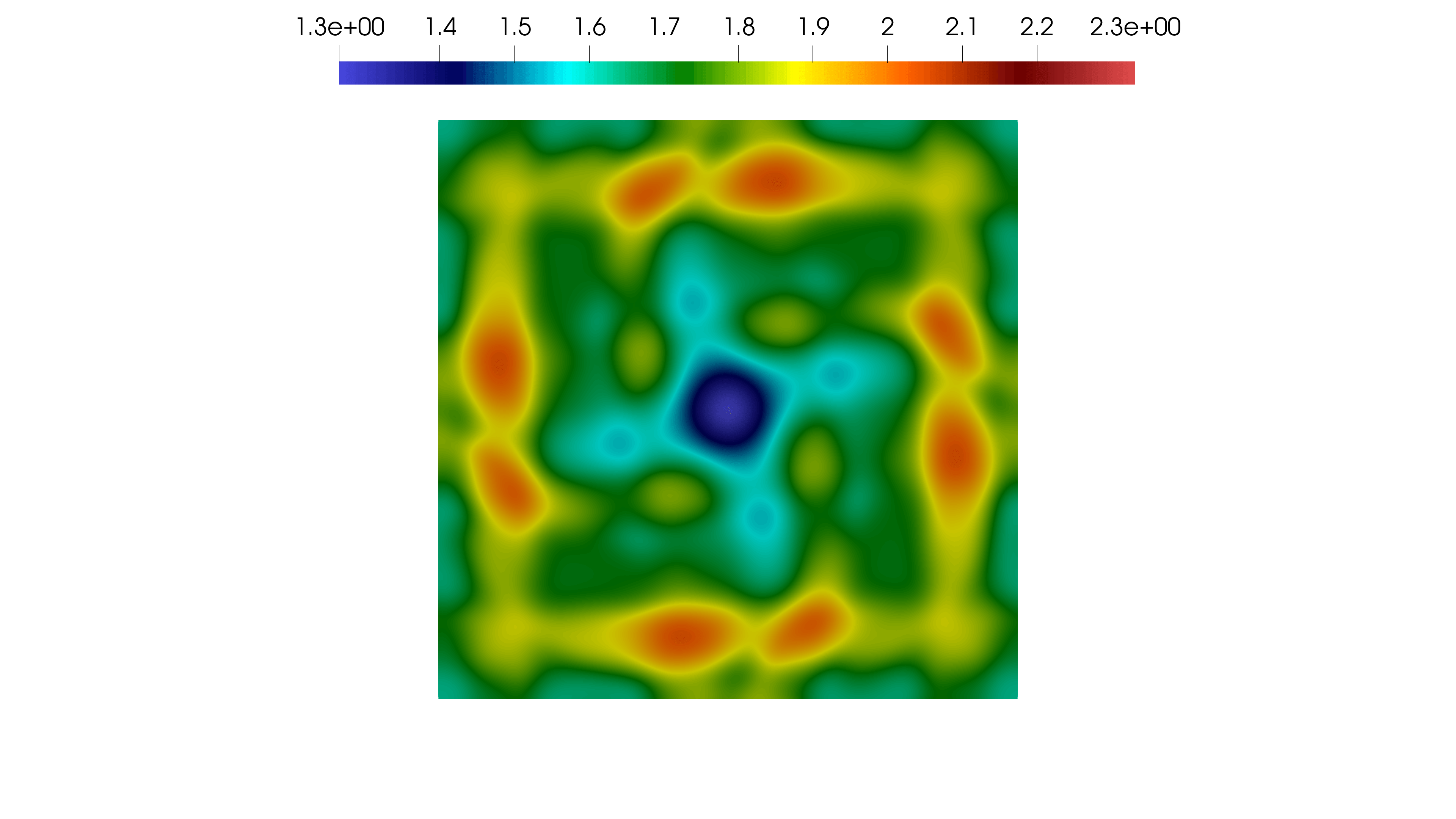}  
	&
	\includegraphics[trim={39cm 9.4cm 39.cm 9.5cm},clip,scale=0.042]{ 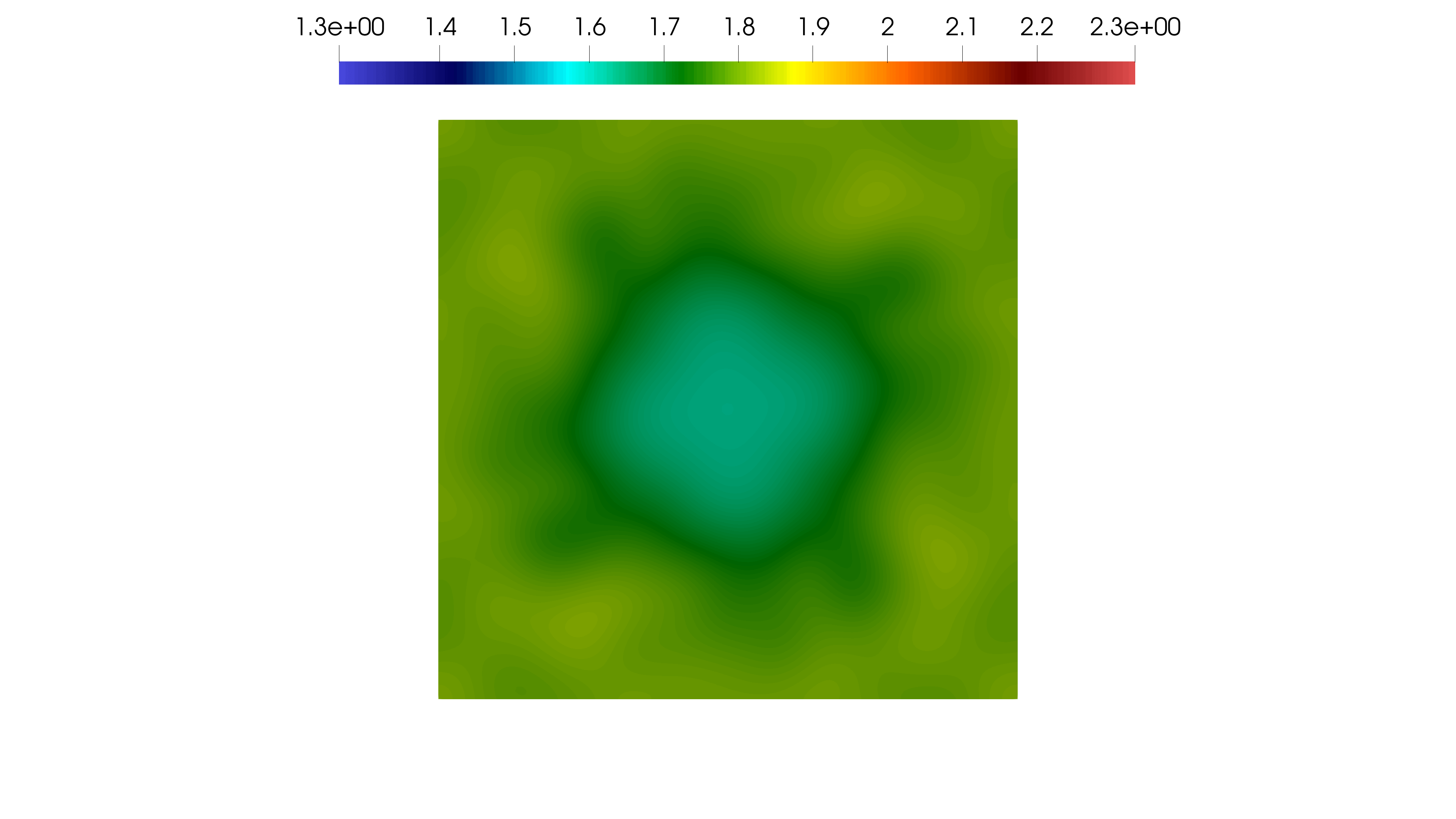}  \\
	\multicolumn{6}{c}{\includegraphics[trim={24.0cm 67.5cm 24.0cm 0.2cm},clip,scale=0.14]{ 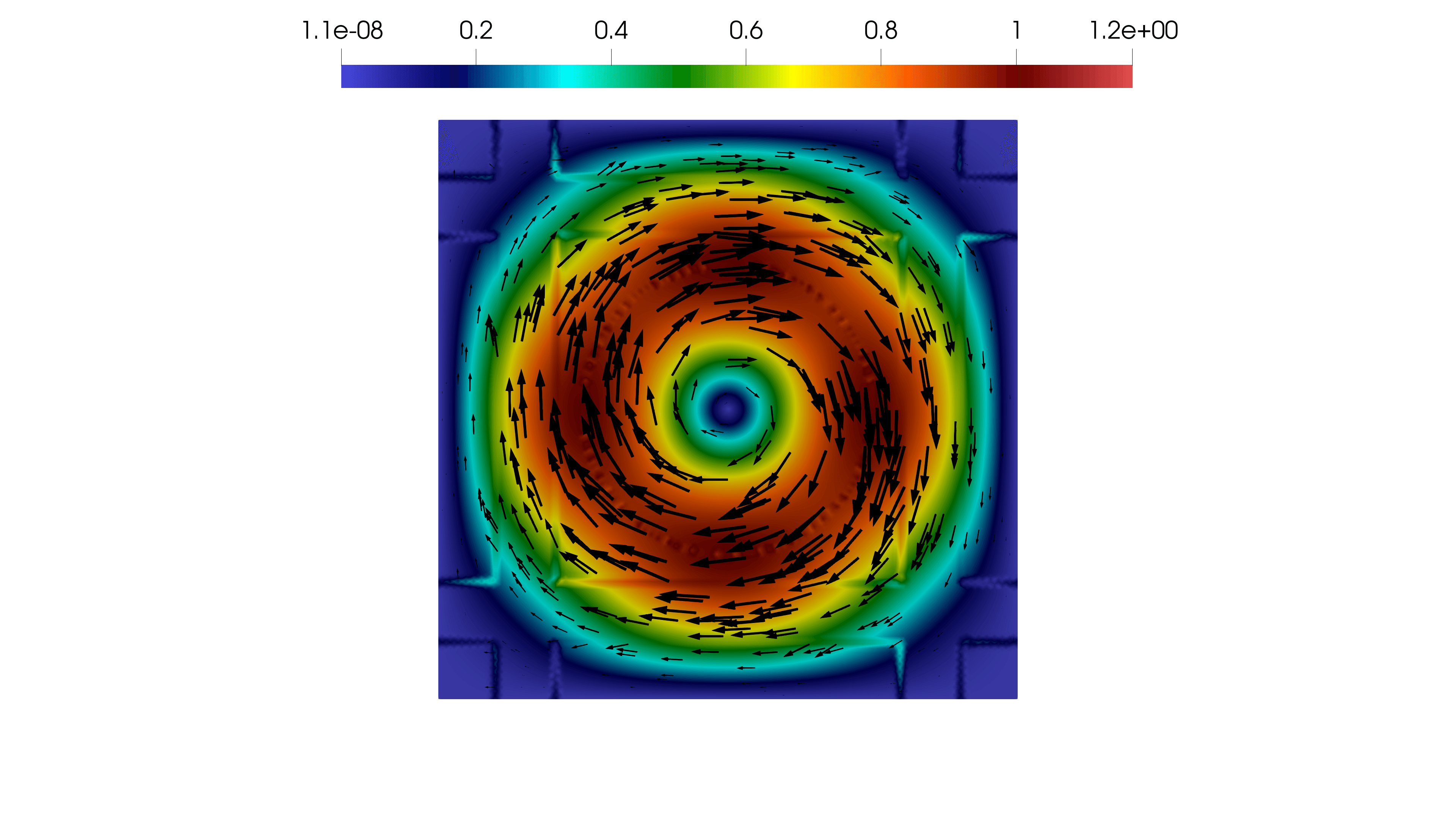}} \\[-0.5em]
	\includegraphics[trim={39cm 9.4cm 39.cm 9.5cm},clip,scale=0.042]{ Velofield.0010.png}
	&
	\includegraphics[trim={39cm 9.4cm 39.cm 9.5cm},clip,scale=0.042]{ 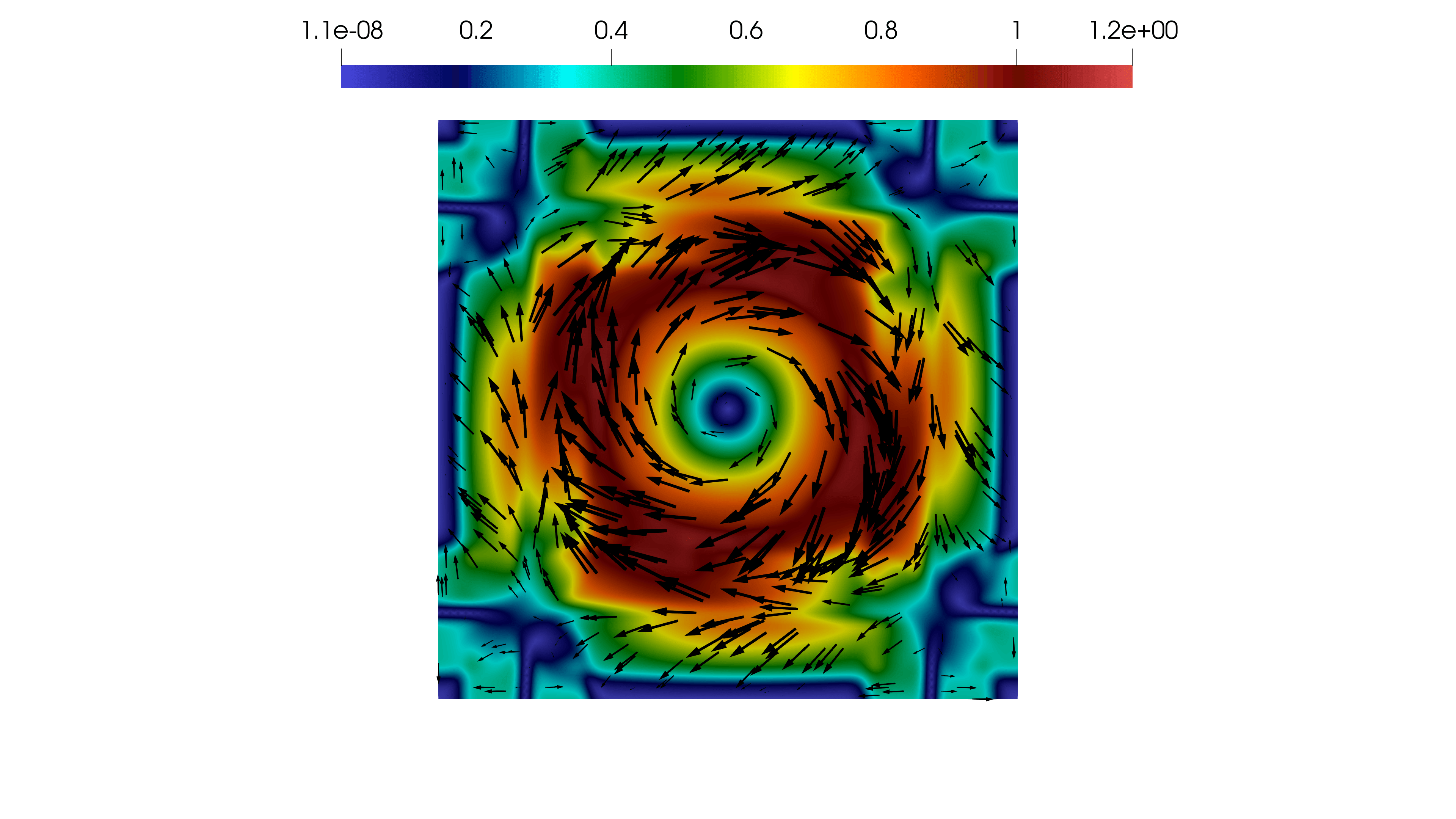} 
	&
	\includegraphics[trim={39cm 9.4cm 39.cm 9.5cm},clip,scale=0.042]{ 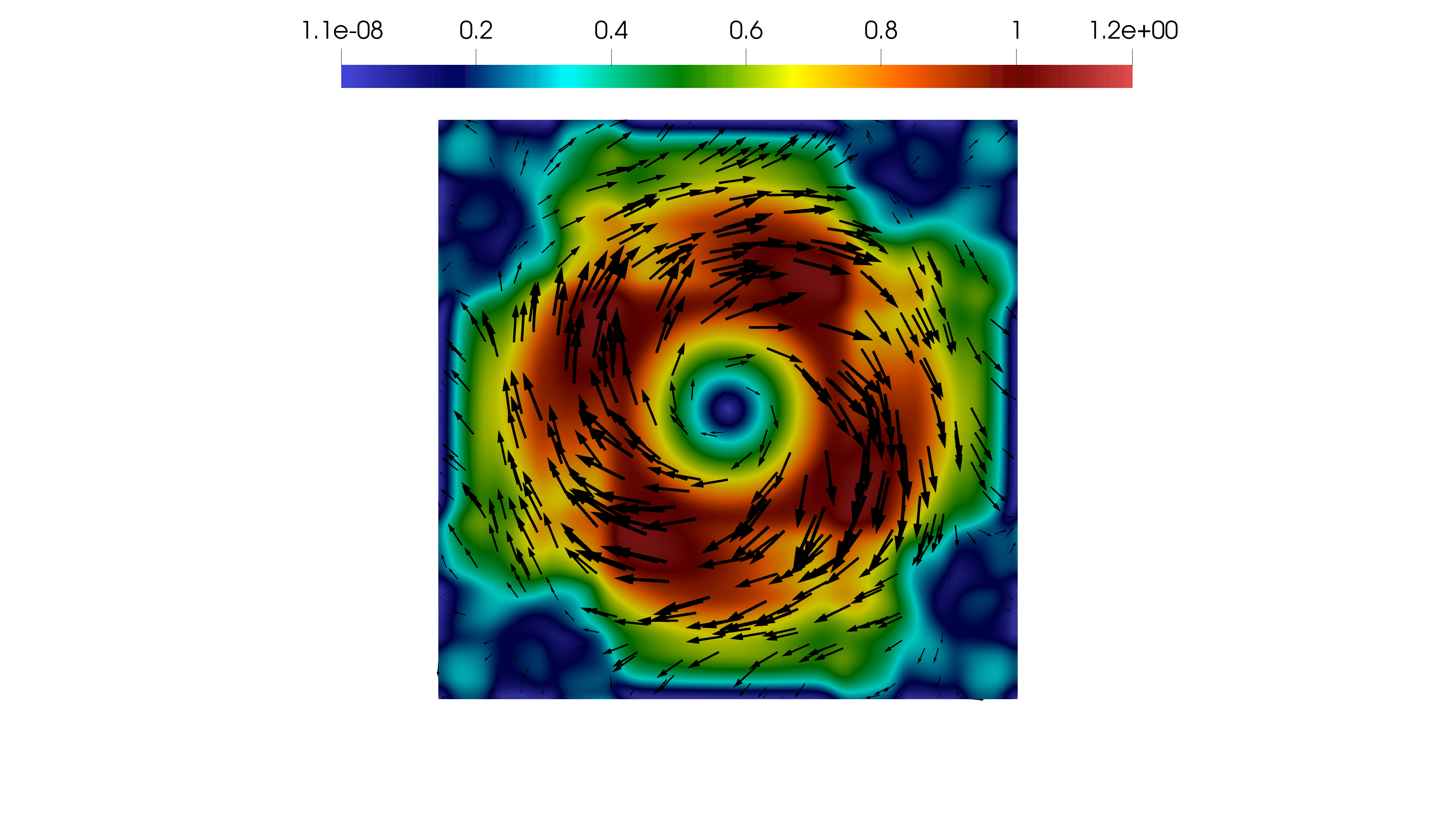}
	&
	\includegraphics[trim={39cm 9.4cm 39.cm 9.5cm},clip,scale=0.042]{ 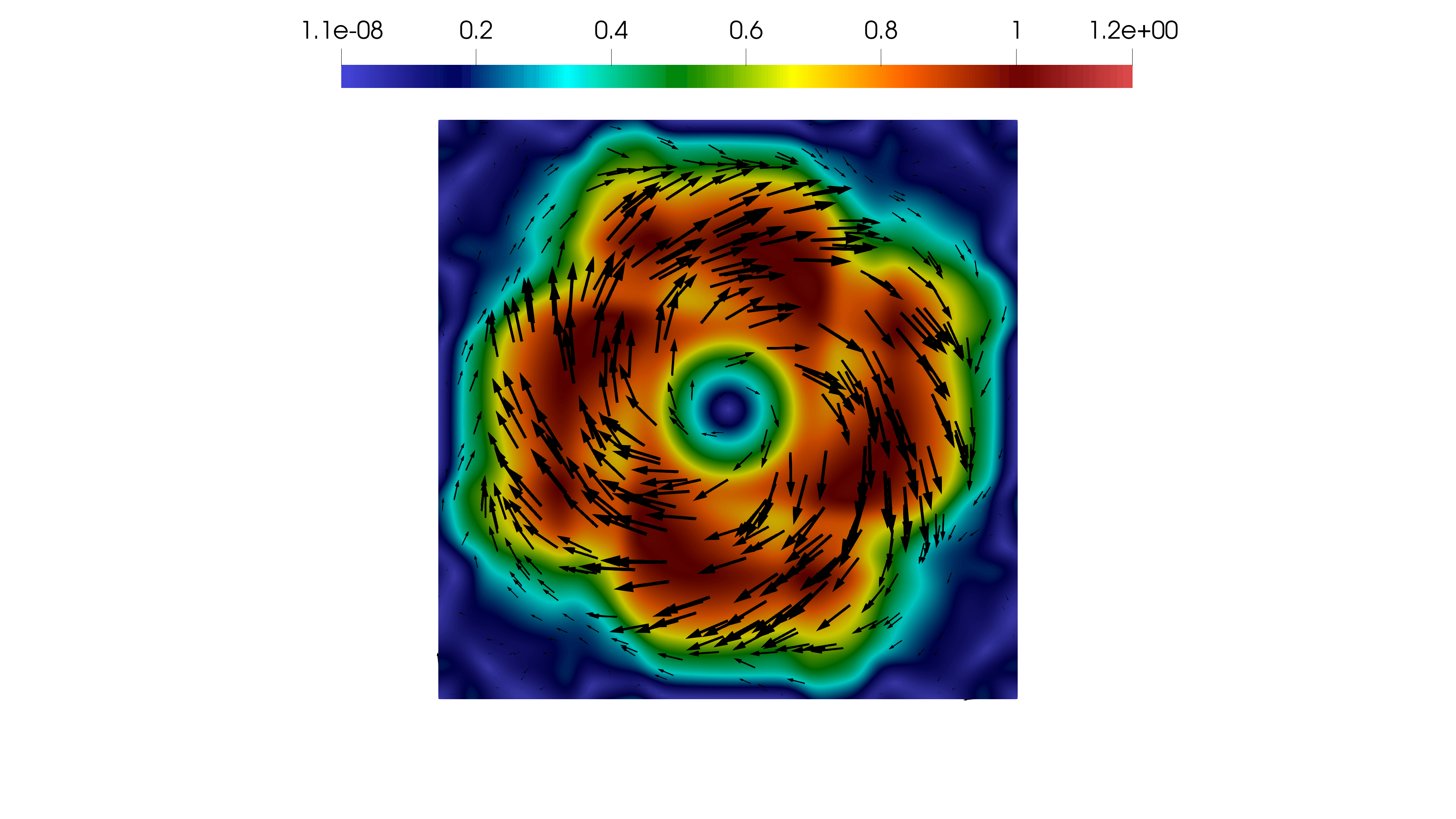}
	&
	\includegraphics[trim={39cm 9.4cm 39.cm 9.5cm},clip,scale=0.042]{ 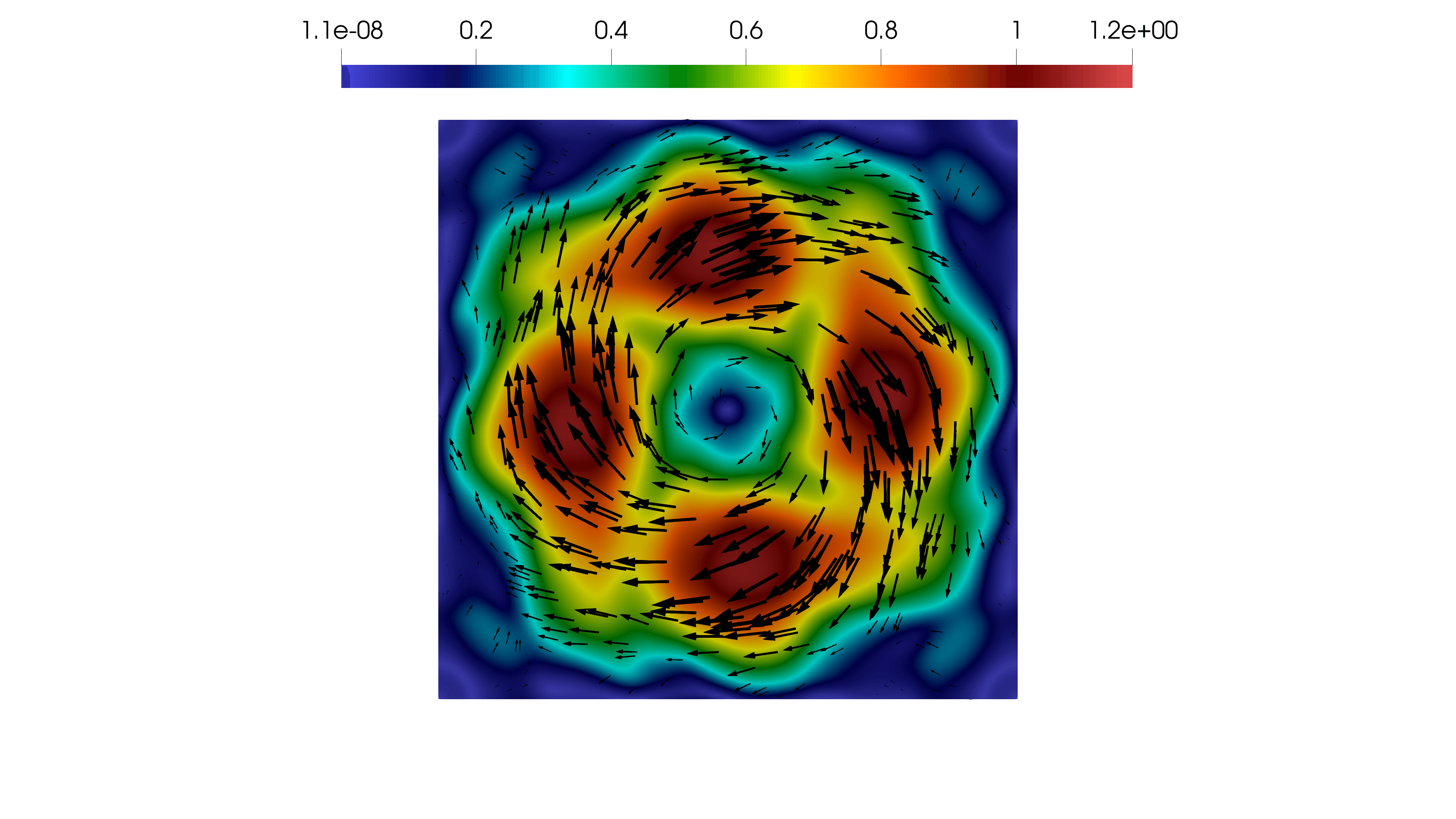} 
	&
	\includegraphics[trim={39cm 9.4cm 39.cm 9.5cm},clip,scale=0.042]{ 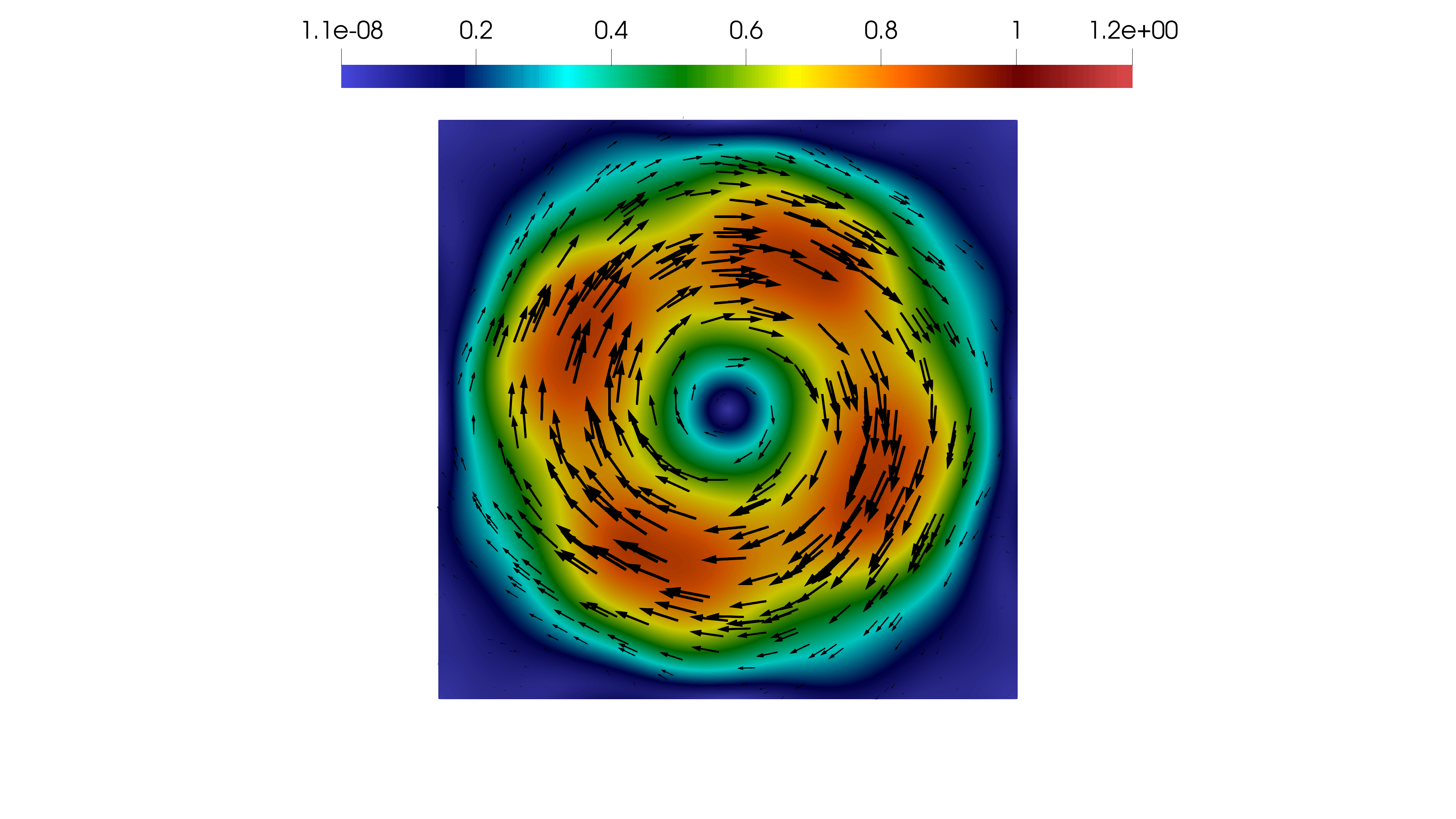}\\[-0.5em]
	\multicolumn{6}{c}{\includegraphics[trim={24.0cm 67.5cm 24.0cm 0.2cm},clip,scale=0.14]{ 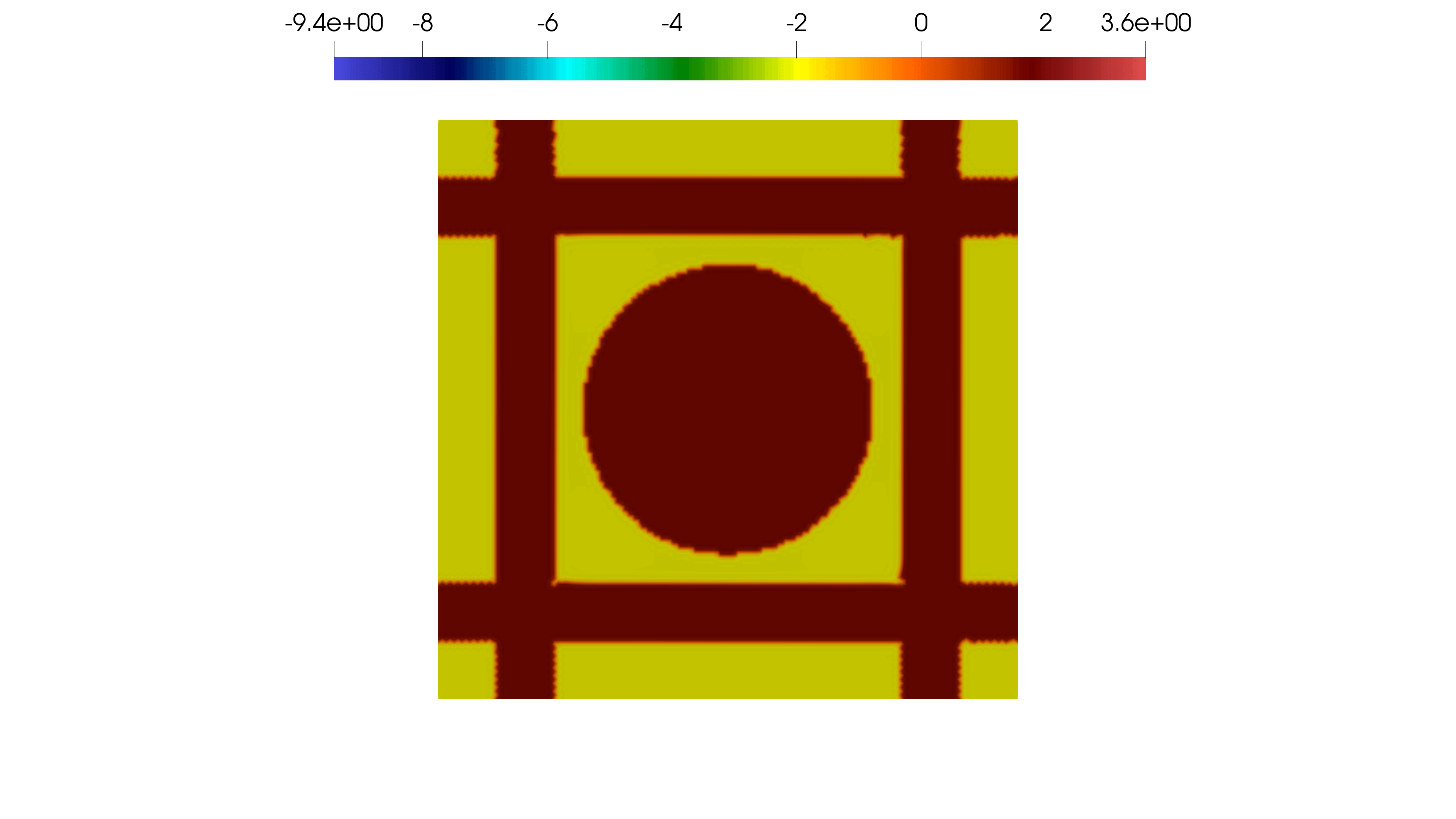}} \\[-0.5em]
	\includegraphics[trim={39cm 9.4cm 39.cm 9.5cm},clip,scale=0.042]{ pressure.0010.png} 
	& 
	\includegraphics[trim={39cm 9.4cm 39.cm 9.5cm},clip,scale=0.042]{ 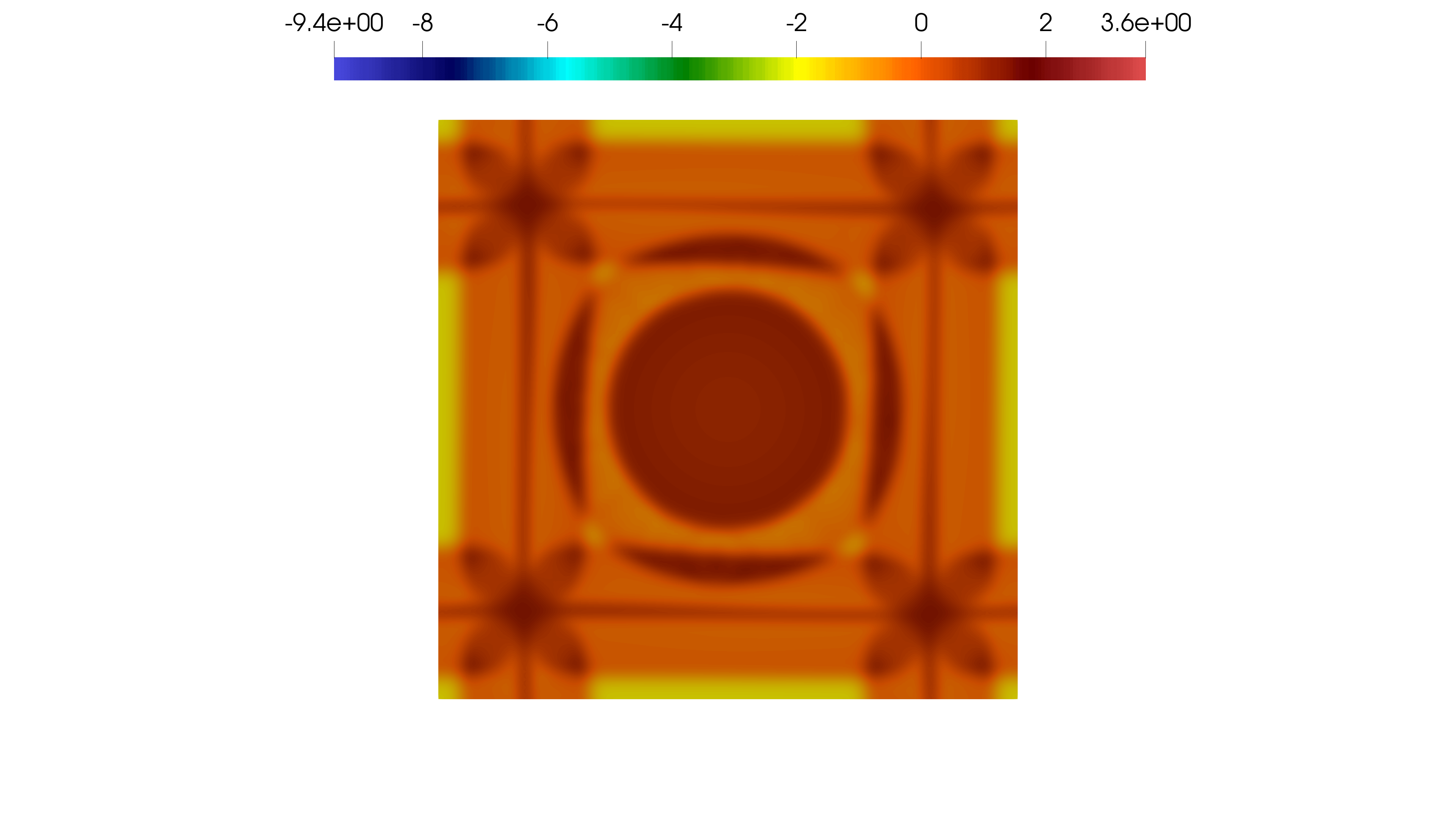} 
	&
	\includegraphics[trim={39cm 9.4cm 39.cm 9.5cm},clip,scale=0.042]{ 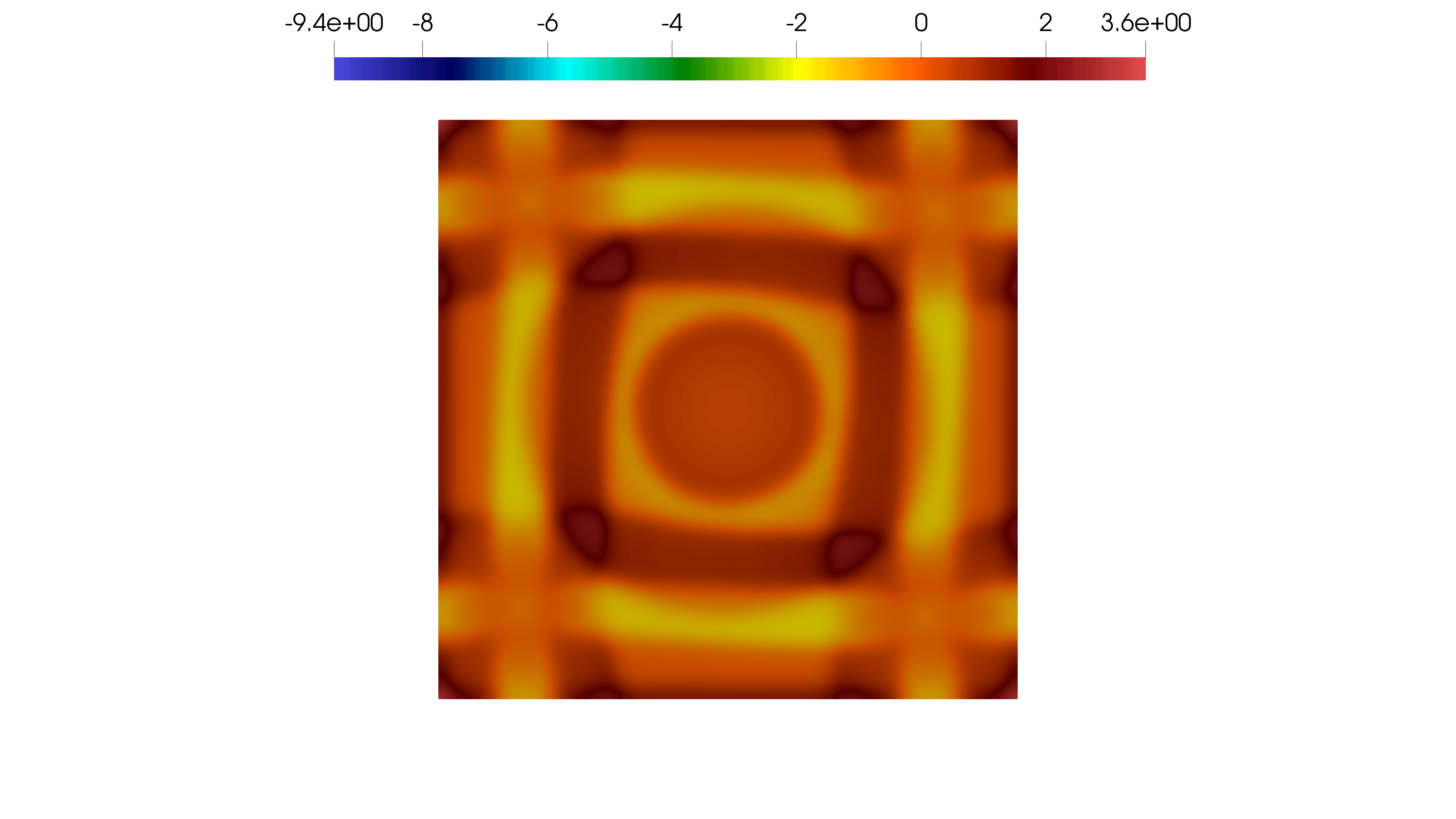} 
	&
	\includegraphics[trim={39cm 9.4cm 39.cm 9.5cm},clip,scale=0.042]{ 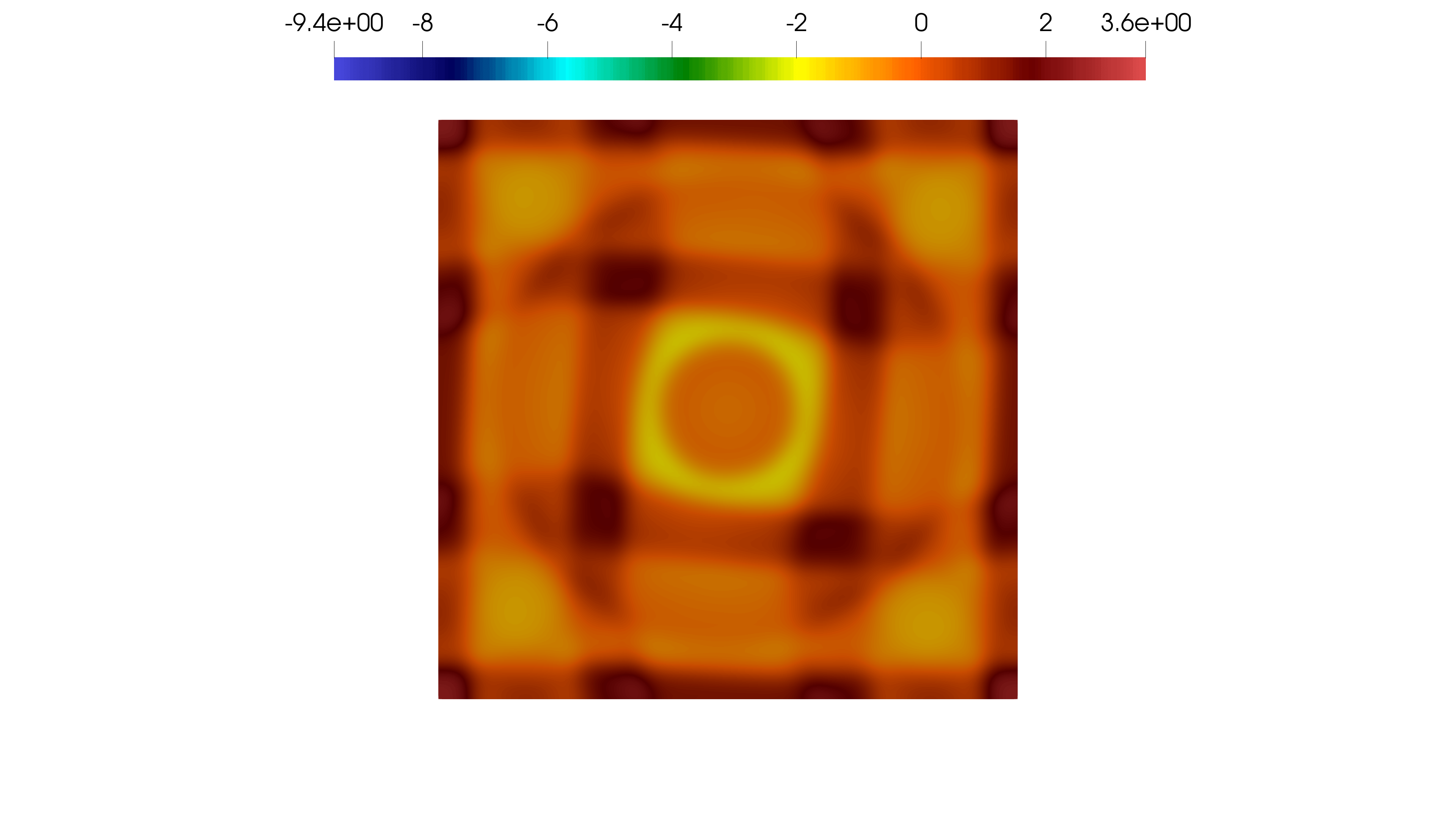} 
	& 
	\includegraphics[trim={39cm 9.4cm 39.cm 9.5cm},clip,scale=0.042]{ 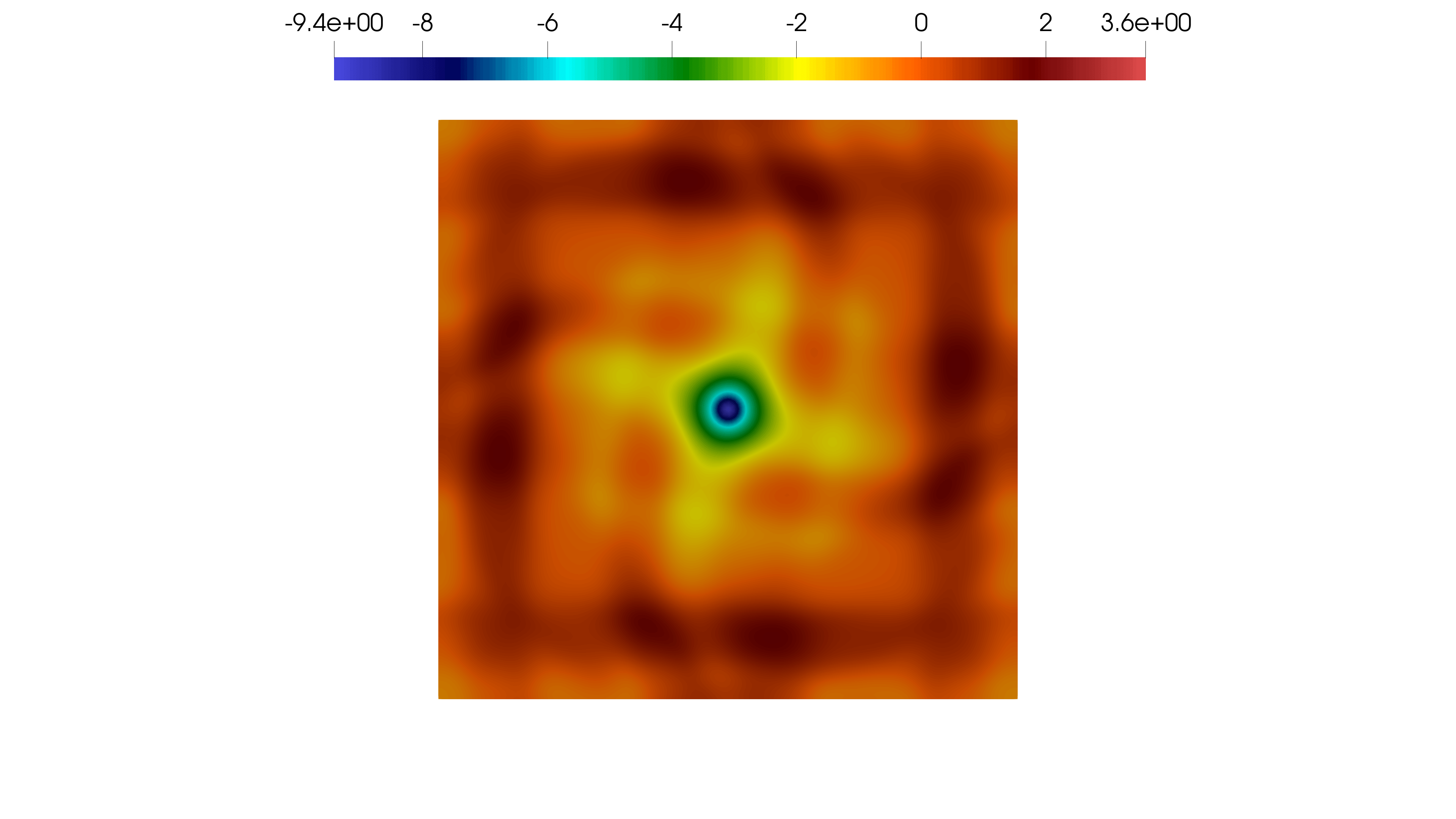} 
	&
	\includegraphics[trim={39cm 9.4cm 39.cm 9.5cm},clip,scale=0.042]{ 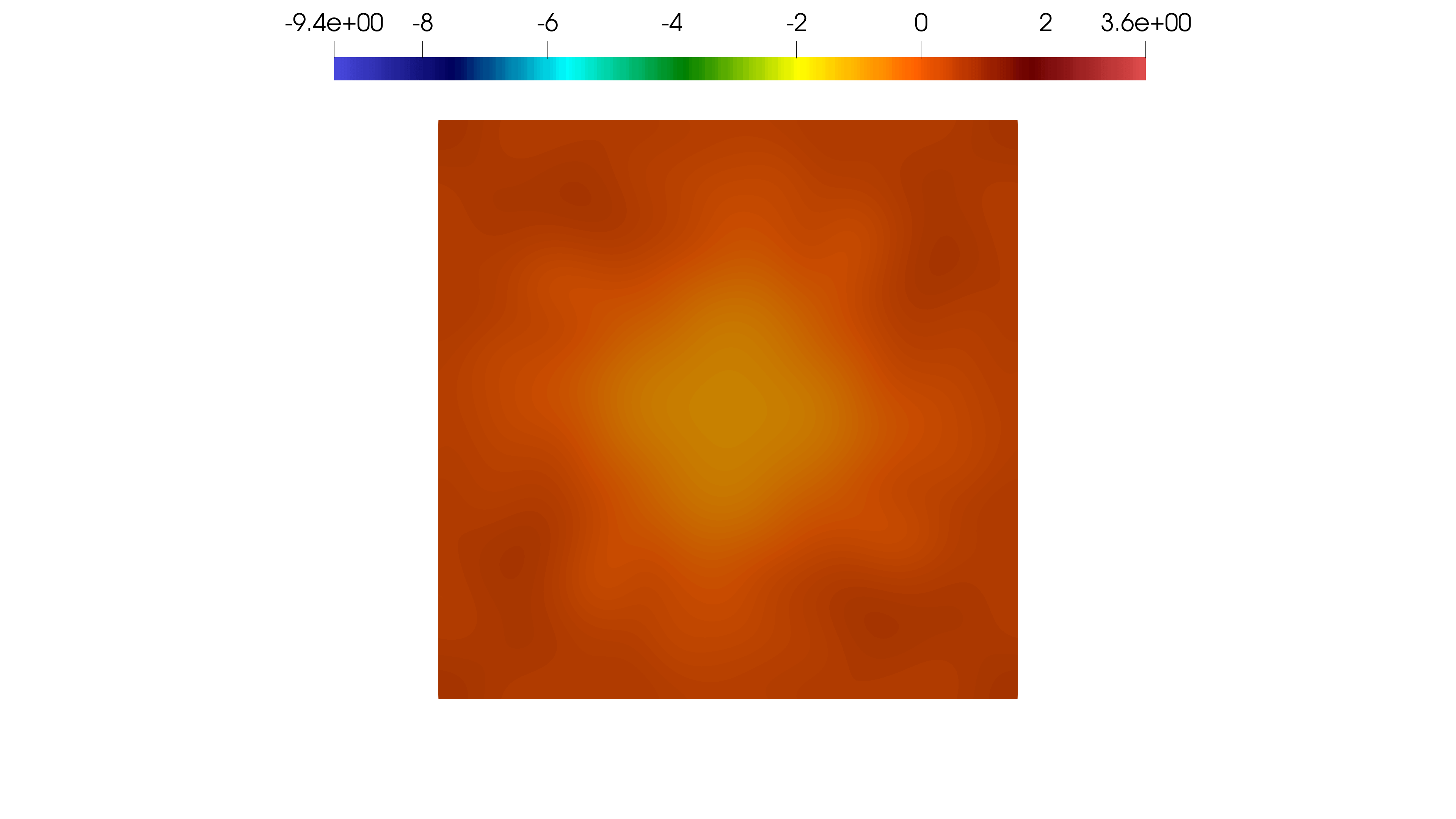} 
\end{tabular}
\caption{Snapshots of the total mass density $\rho$ (upper row), velocity magnitude $|\bm{u}|^2$, and pressure $p$ at times $t=0.001,0.02,0.04,0.06,0.1,0.5$ for the experiment in Section \ref{subsec:experiment}.}
\label{fig.up}
\end{figure}

The relative energy $E_{\rm rel}$ is a very useful tool to measure the distance between two solutions of cross-diffusion systems and in particular of the Maxwell--Stefan equations. In contrast to the usual $L^2$ distance, the relative energy takes care of the highly nonlinear structure of cross-diffusion equations. Here, we compare the time-dependent solution and the steady state:
\begin{align*}
  E_{\rm rel} = \int_\Omega\bigg(\sum_{i=1}^n\rho_i\log\frac{\rho_i}{\rho_i^\infty}
  - \rho\log\frac{\rho}{\rho^\infty} 
  + \frac{\rho}{2}|\bm{u}-\bm{u}^\infty|^2\bigg)\dd x,
\end{align*}
where the stationary variables are defined by
\begin{align*}
  \rho_i^\infty = \langle\rho_i^0,\mathrm{1}\rangle, \quad
  \bm{u}^\infty = \frac{\langle\rho^0\bm{u}^0,1\rangle}{\langle\rho^0,1\rangle}.
\end{align*}
Moreover, $\rho^\infty:=\sum_{i=1}^N\rho_i^\infty$ is the stationary total mass, and the stationary pressure vanishes, $p^\infty=0$. We note that the relative energy can also be used to derive error estimates, cf. \cite{Brunk25,BrunkM3AS,MariaBook}.

Figure \ref{fig.ener} shows the time decay of the discrete variant of the relative energy. It is computed by NGSolve by using a fifth-order quadrature rule. We observe that the relative energy decays with time (nearly exponentially for large times). Interestingly, the components of the relative energy (the relative internal energy and the relative kinetic energy) do not decrease in time, but only its sum. This is in contrast to the incompressible Navier--Stokes--Maxwell--Stefan model, where the kinetic and internal energies are both decaying in time. This is not surprising, since in that model, the Maxwell--Stefan diffusion and Navier--Stokes flow are basically decoupled, while the coupling is much stronger in our model. Again the partial and total masses and the pointwise quasi-incompressibility constraint are conserved up to errors of order $10^{-14}$.

\begin{figure}[ht]
\centering
\includegraphics[width=100mm]{ 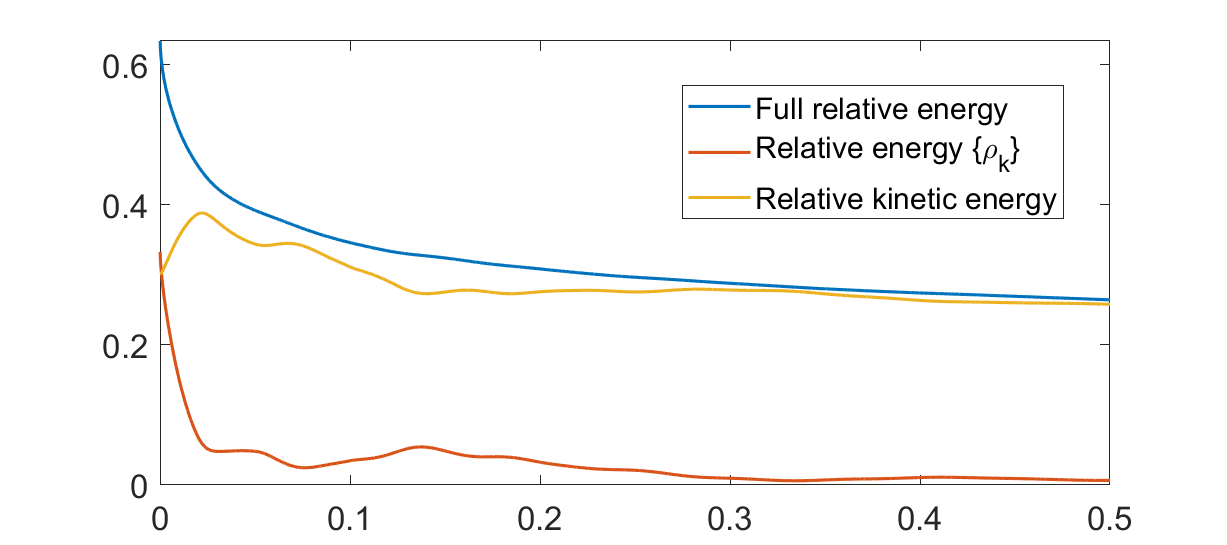}
\caption{Relative energy versus time for the experiment in Section \ref{subsec:experiment}.}
\label{fig.ener}
\end{figure}

\subsection{Three-component experiment II}\label{subsec:experiment2}

We present a second three-component experiment with the same data as in the previous subsection, except $\eta=10^{-1}$, $M_{ij}(\pvec{\rho}) =10^{-1}(\rho_i\delta_{ij} - \tfrac{\rho_i\rho_j}{\rho}),$ $V_1=0.35$, $V_2=0.65$, $V_3=0.5$ and 
\begin{align*}
  \rho_1^0(x,y) &= 0.5 + 0.49999\tanh\bigg(
  \frac{\sqrt{(x-0.4)^2+(y-0.4)^2}-0.25}{10^{-2}\sqrt{2}}\bigg), \\
  \rho_2^0(x,y) &= 0.5 + 0.49999\tanh\bigg(
  \frac{\sqrt{(x-0.6)^2+(y-0.6)^2}-0.25}{10^{-2}\sqrt{2}}\bigg), \\
  \rho_3^0(x,y) &= V_3^{-1}\big(1-V_1\rho_1(x,y) - V_2\rho_3(x,y)\big), \\
  \bm{u}^0(x,y) &= \big(-\sin(\pi x)^2\sin(2\pi y),
  \sin(\pi y)^2\sin(2\pi x))^T.
\end{align*}  
Compared to the previous test, the effect of the initial velocity can be observed much better. Figure \ref{fig.test2} illustrates the dynamics of the partial and total mass densities at various times. Again, the mass densities diffuse over time and show some rotational effect, which is well seen for the component $\rho_3$. Figure \ref{fig.ener2} shows the decay in the relative energy. In this case, all relative energies separately decay towards zero.

\begin{figure}[ht]
\centering
\begin{tabular}{cccc}
     \multicolumn{4}{c}{\includegraphics[trim={28.0cm 67.5cm 28.0cm 0.2cm},clip,scale=0.15]{ 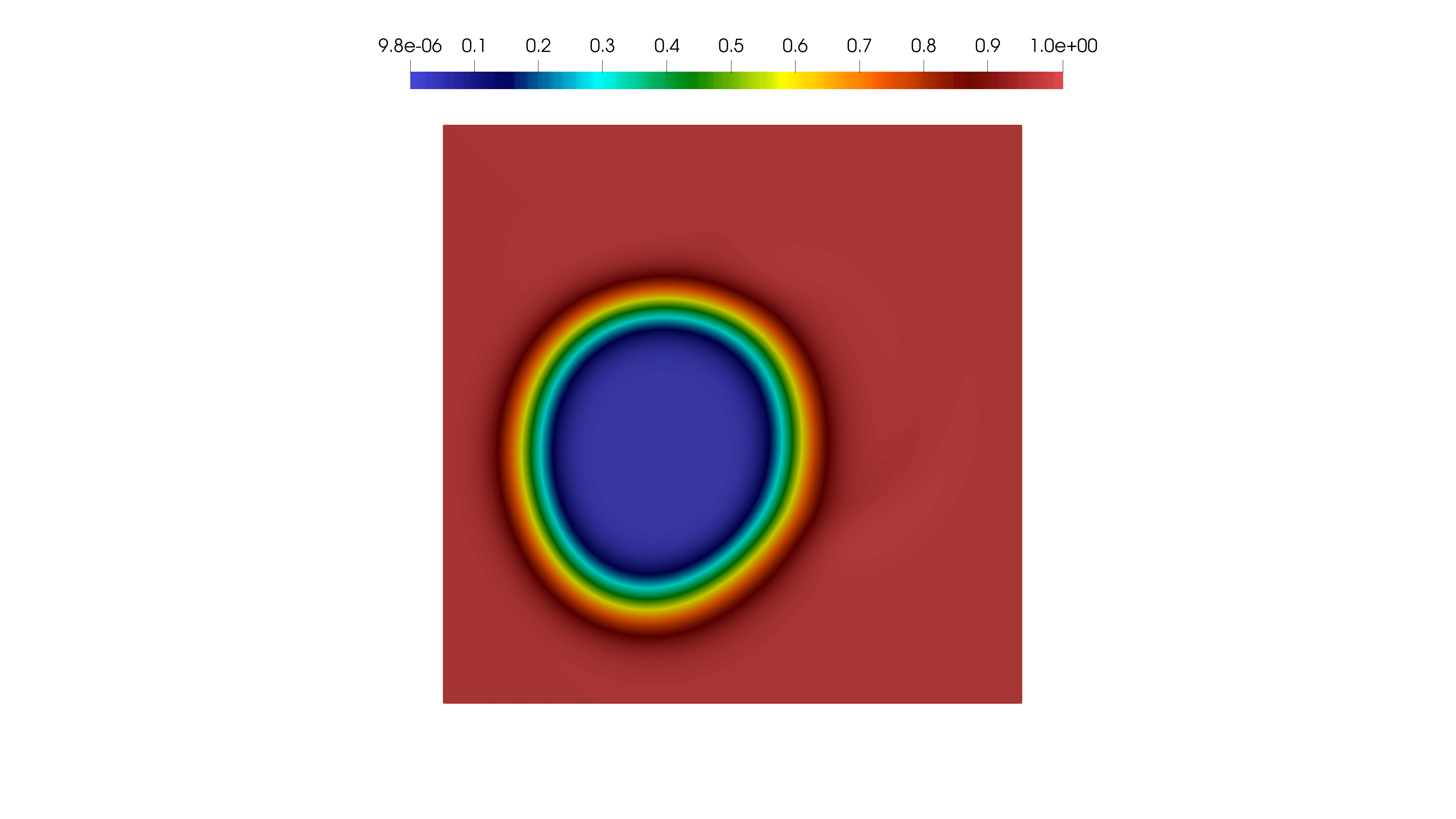}} \\[-0.5em]
     \includegraphics[trim={41cm 1.4cm 41.cm 8.5cm},clip,scale=0.052]{ test2rho_1.0020.png} 
    &
    \includegraphics[trim={41cm 1.4cm 41.cm 8.5cm},clip,scale=0.052]{ 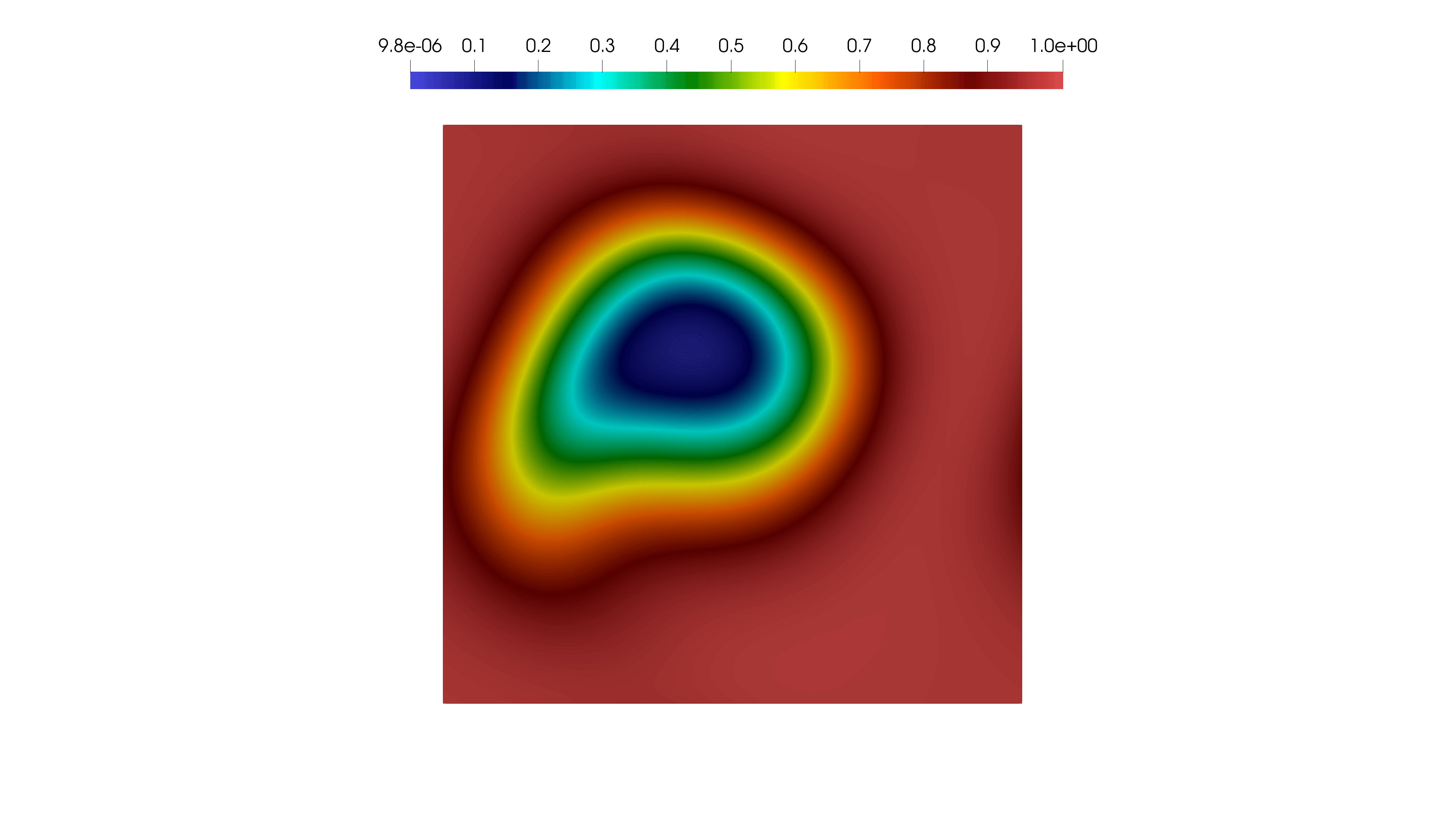}  
    &
    \includegraphics[trim={41cm 1.4cm 41.cm 8.5cm},clip,scale=0.052]{ 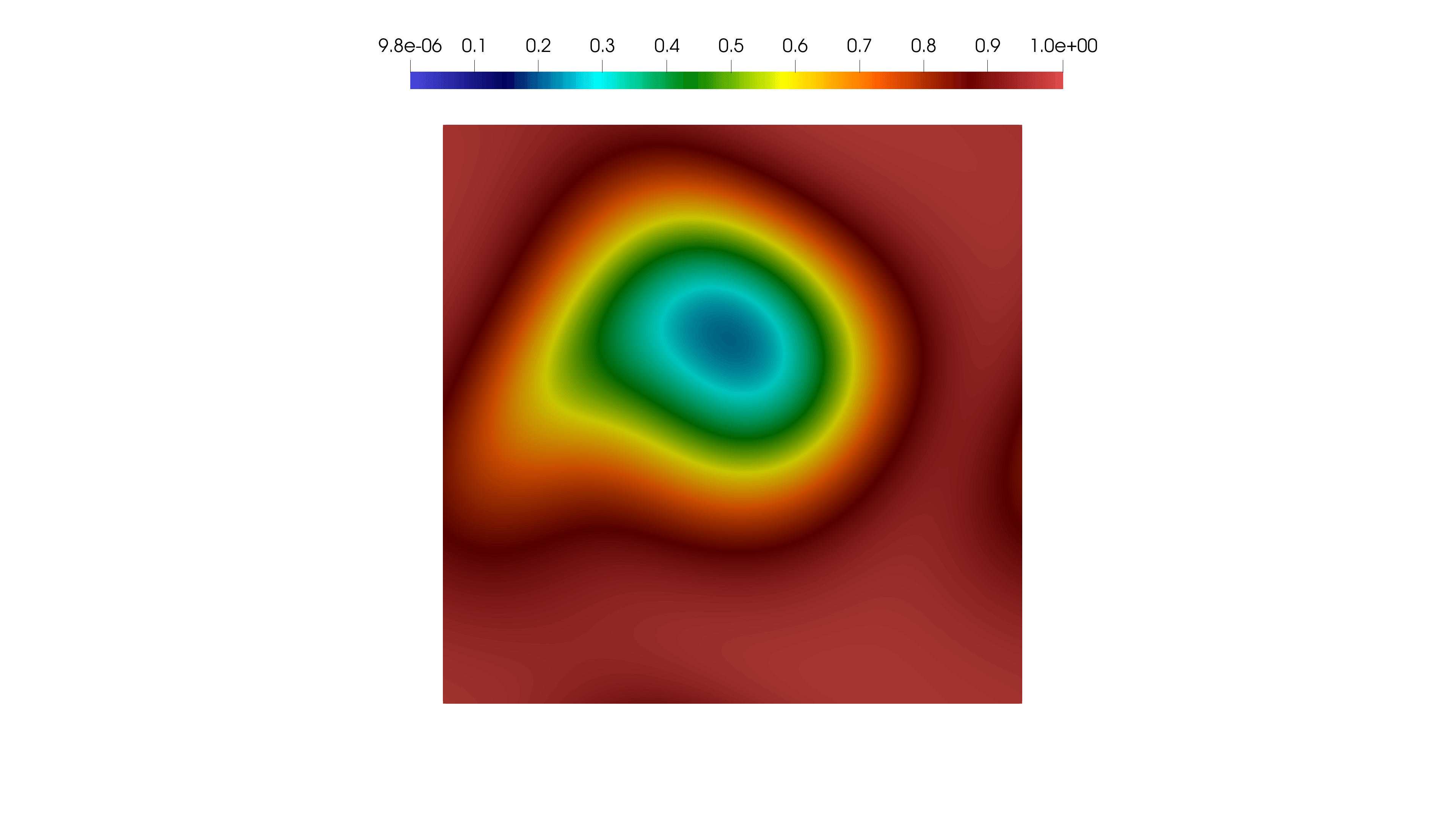} 
    &
    \includegraphics[trim={41cm 1.4cm 41.cm 8.5cm},clip,scale=0.052]{ 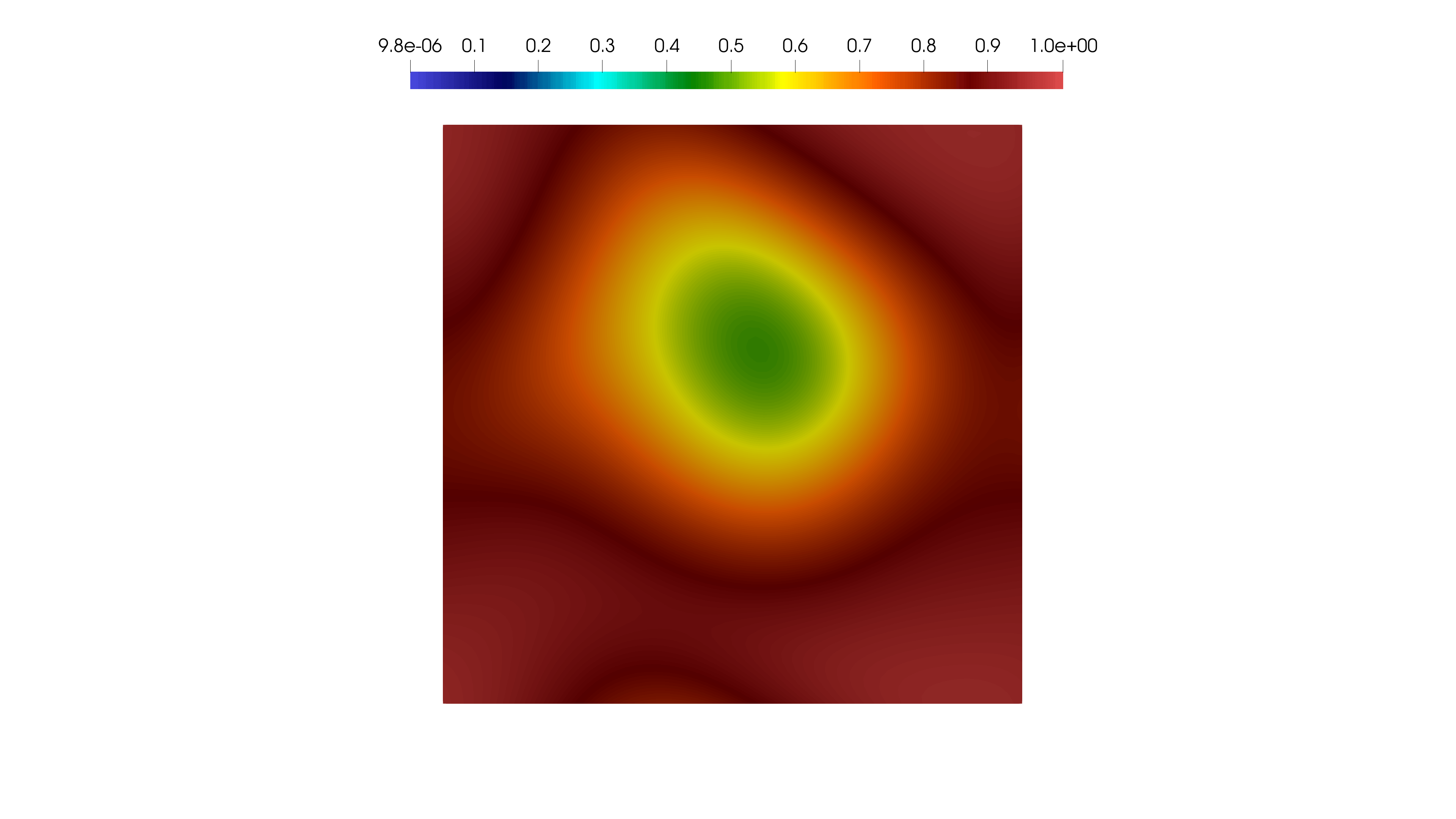}  \\
    \multicolumn{4}{c}{\includegraphics[trim={28.0cm 67.5cm 28.0cm 3.5cm},clip,scale=0.15]{ 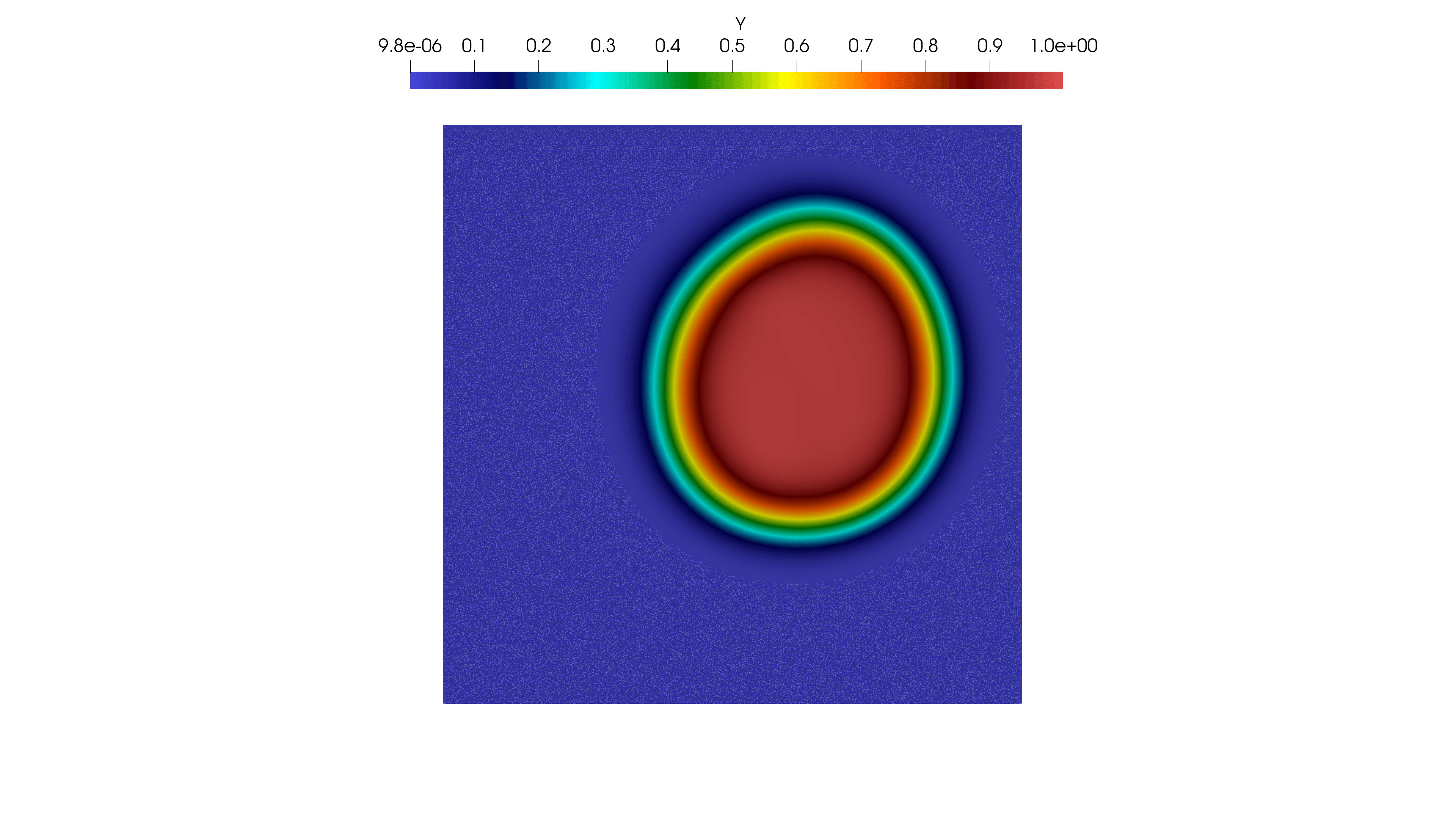}} \\[-0.5em]
     \includegraphics[trim={41cm 1.4cm 41.cm 8.5cm},clip,scale=0.052]{ test2rho_2.0020.png} 
    &
    \includegraphics[trim={41cm 1.4cm 41.cm 8.5cm},clip,scale=0.052]{ 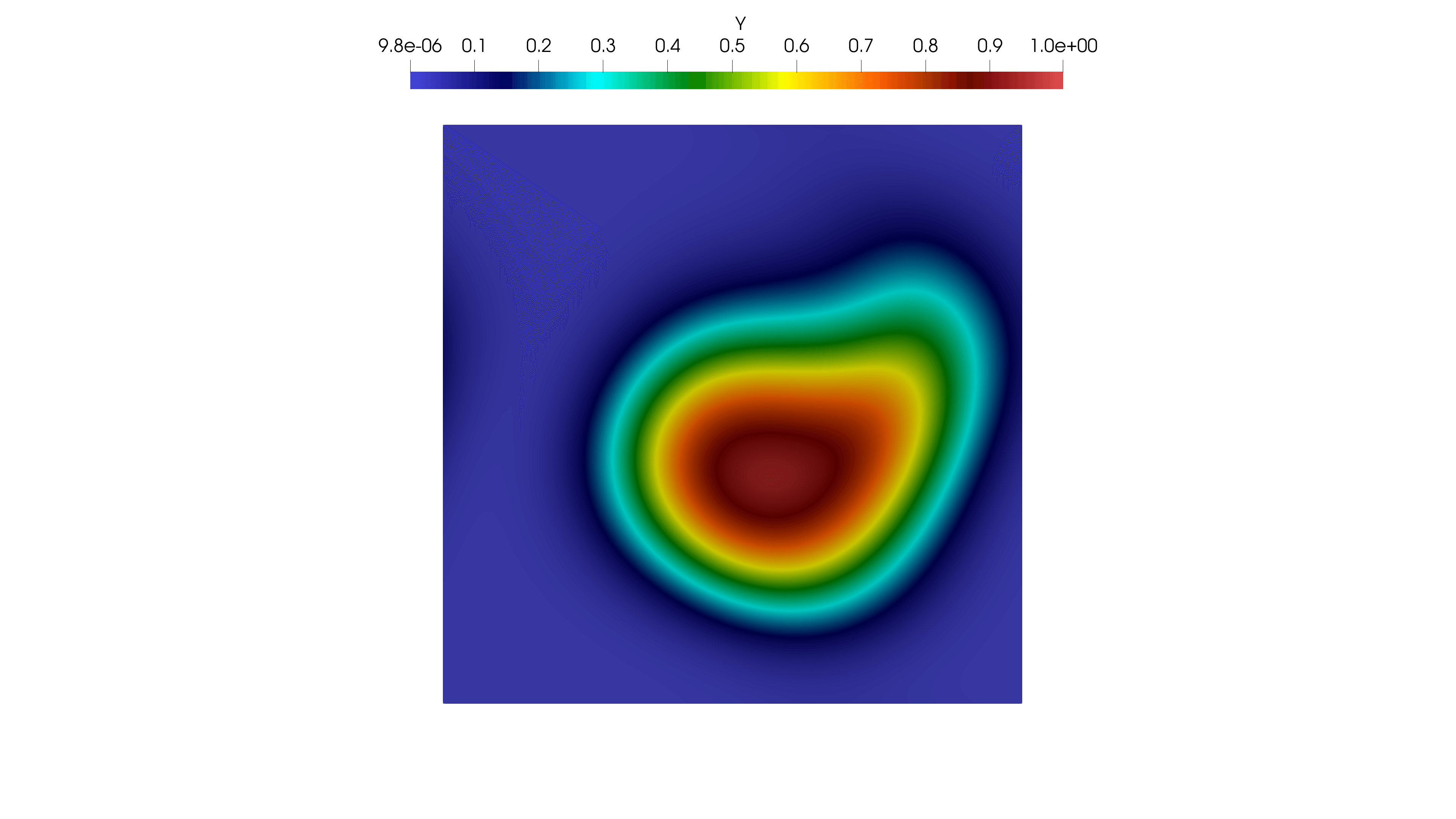}  
    &
    \includegraphics[trim={41cm 1.4cm 41.cm 8.5cm},clip,scale=0.052]{ 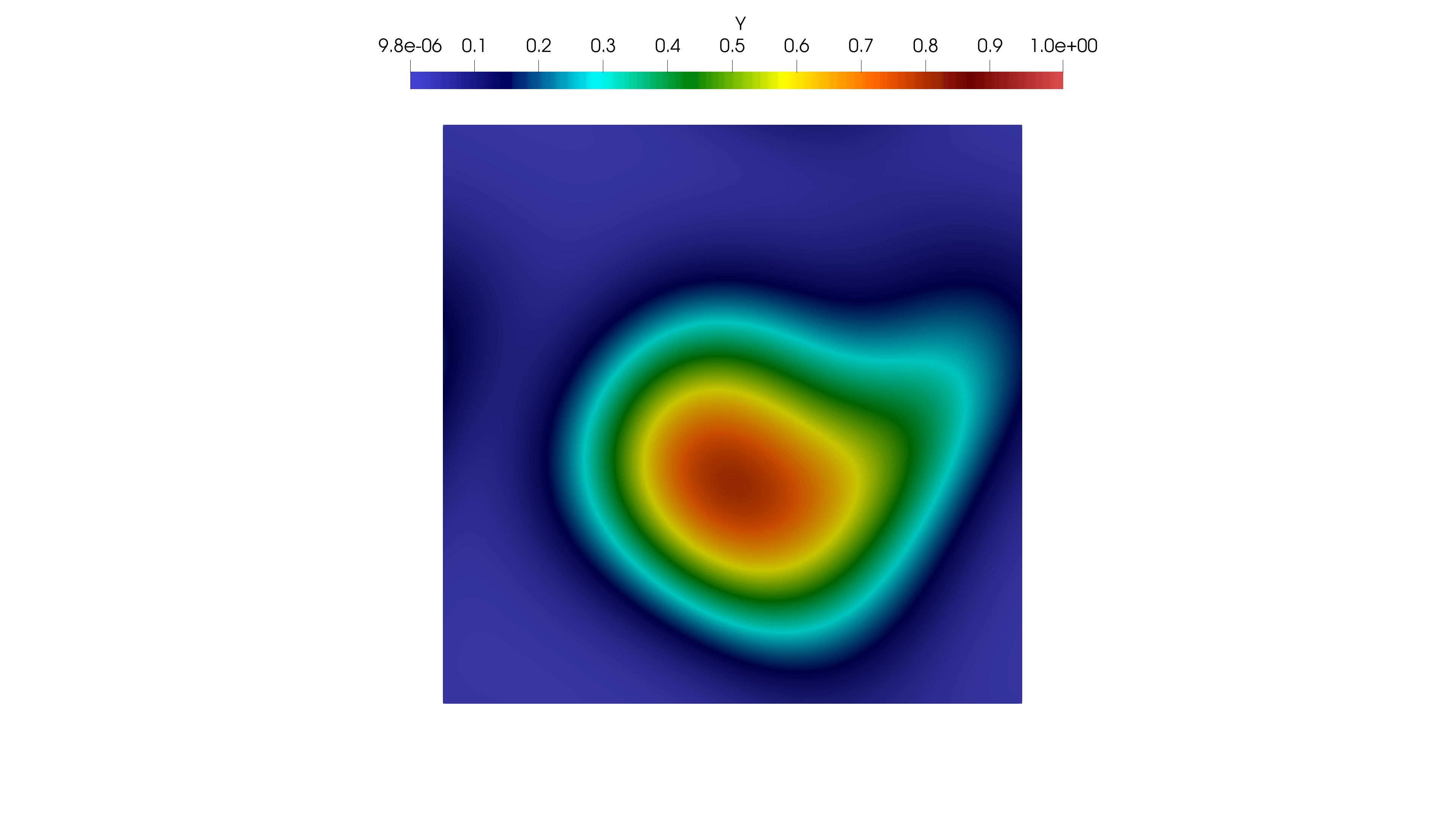} 
    &
    \includegraphics[trim={41cm 1.4cm 41.cm 8.5cm},clip,scale=0.052]{ 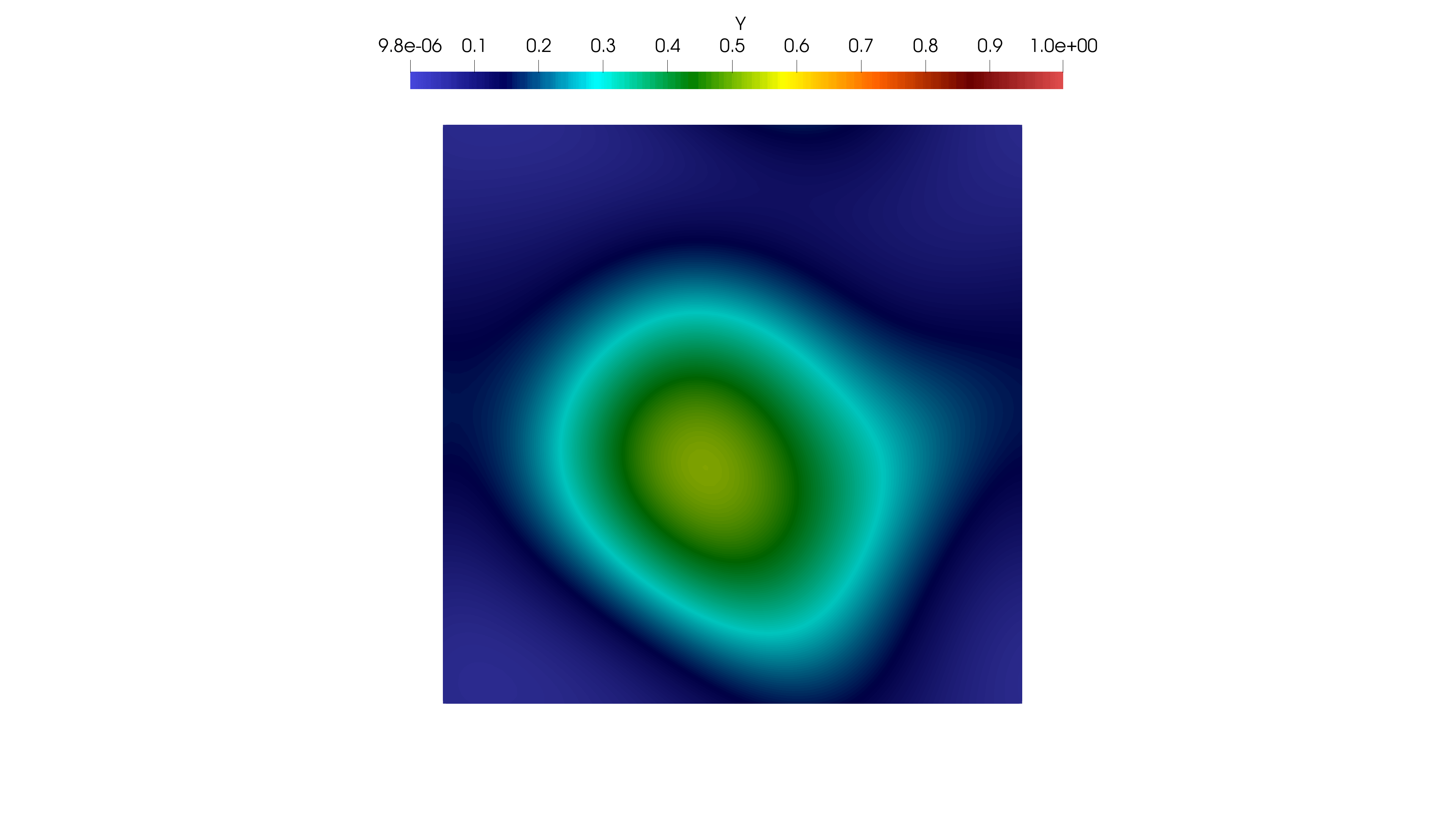}  \\
    \multicolumn{4}{c}{\includegraphics[trim={28.0cm 67.5cm 28.0cm 3.5cm},clip,scale=0.15]{ 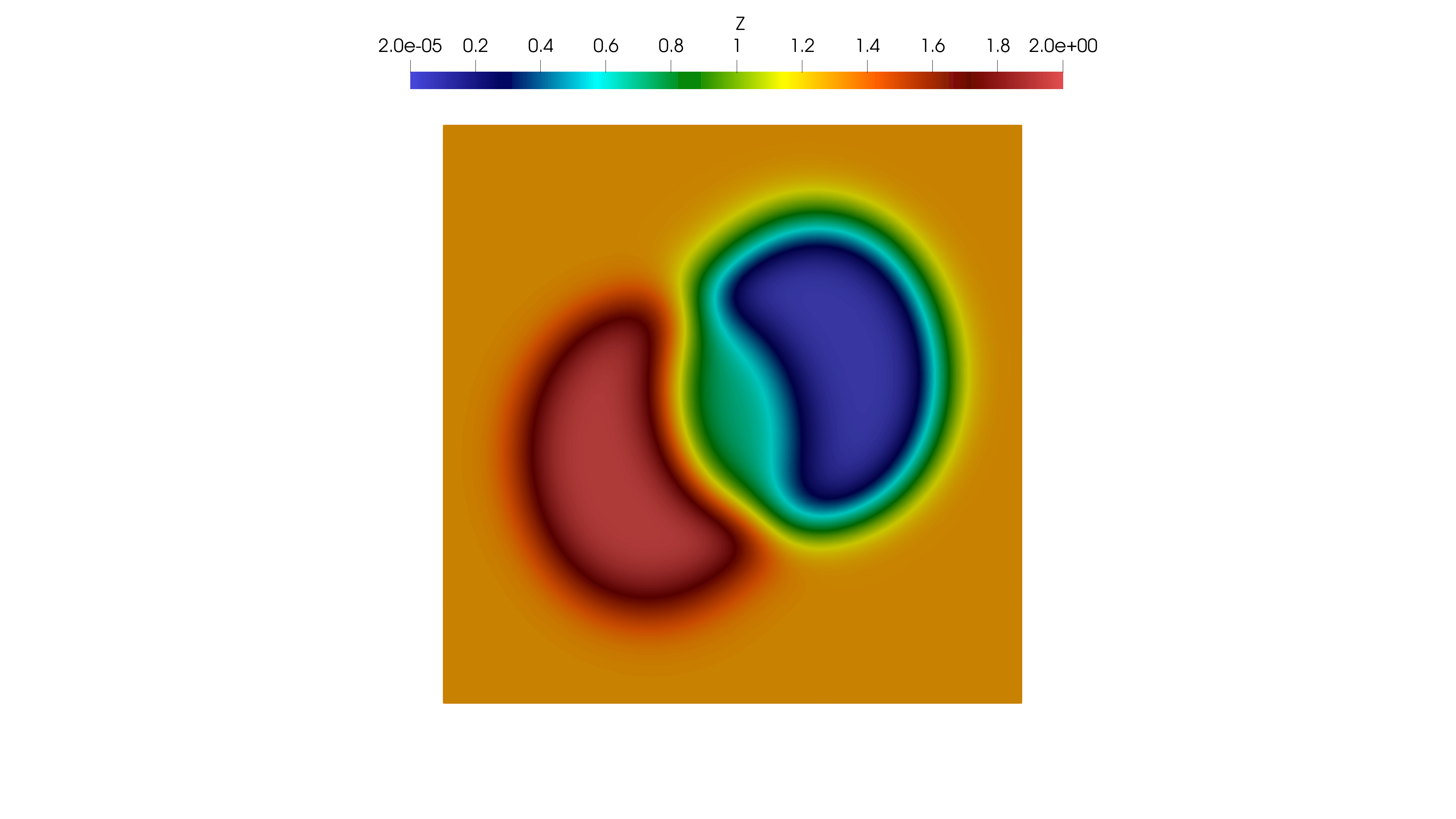}} \\[-0.5em]
     \includegraphics[trim={41cm 1.4cm 41.cm 8.5cm},clip,scale=0.052]{ test2rho_3.0020.png} 
    &
    \includegraphics[trim={41cm 1.4cm 41.cm 8.5cm},clip,scale=0.052]{ 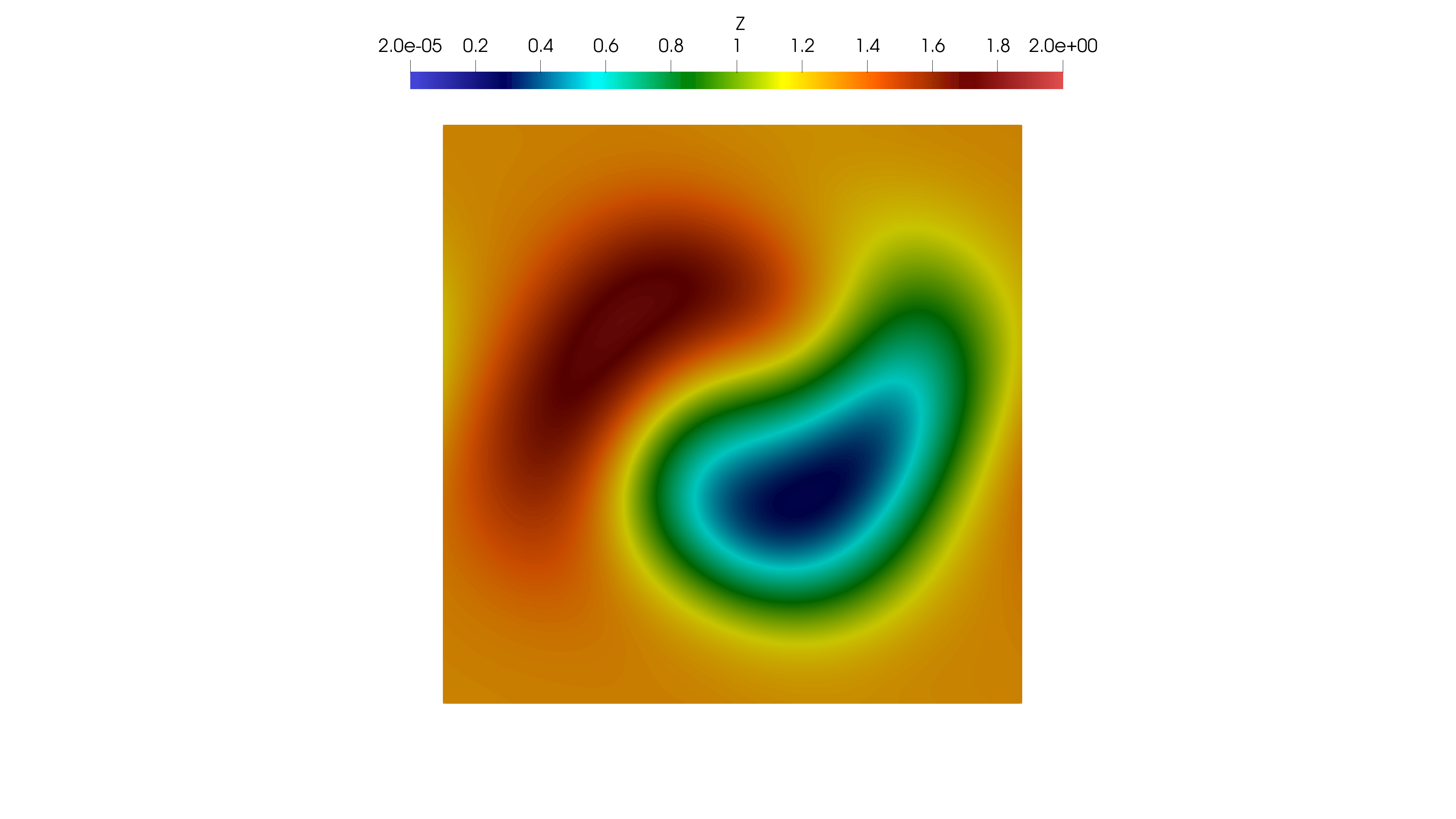}  
    &
    \includegraphics[trim={41cm 1.4cm 41.cm 8.5cm},clip,scale=0.052]{ 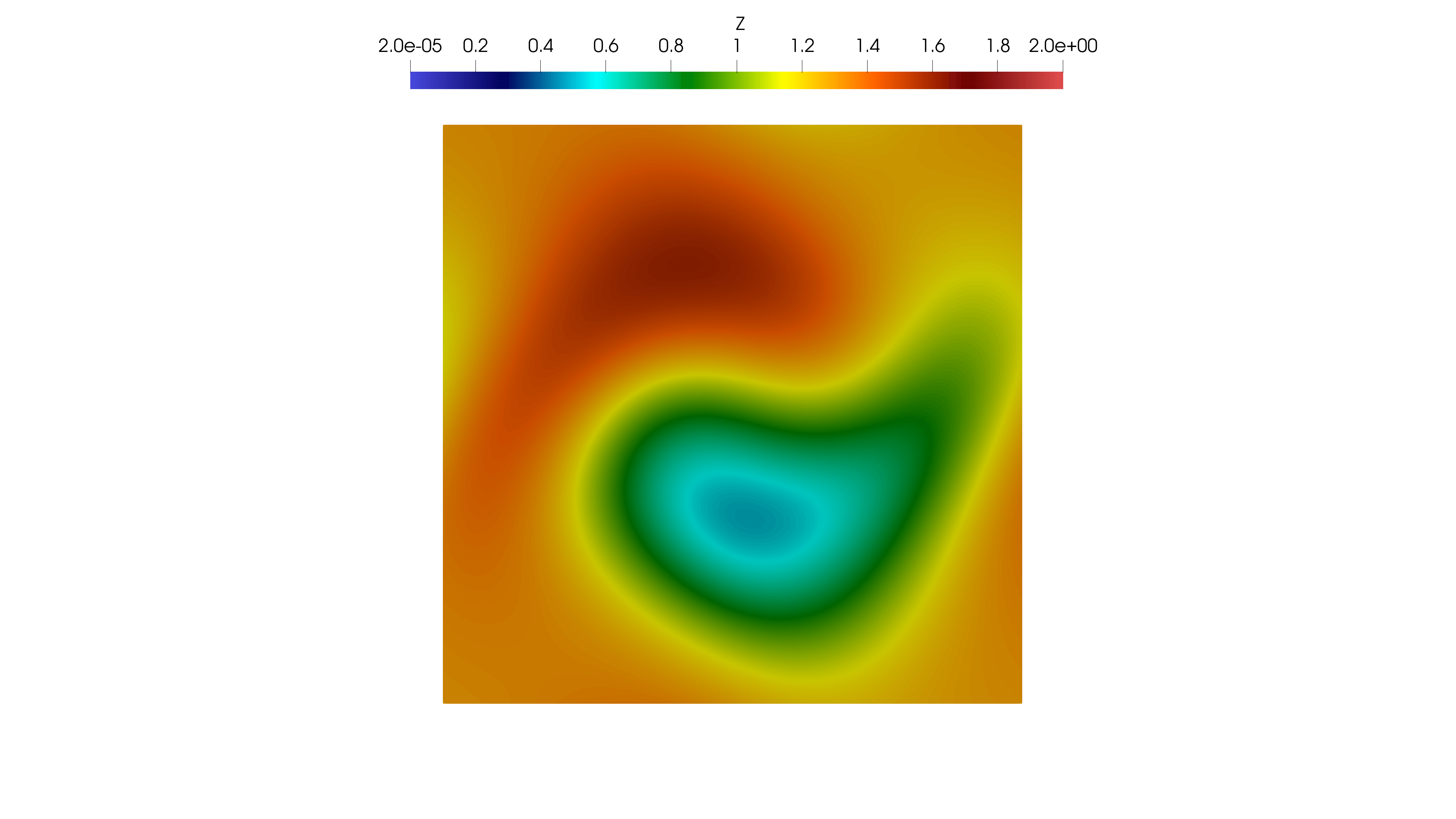} 
    &
    \includegraphics[trim={41cm 1.4cm 41.cm 8.5cm},clip,scale=0.052]{ 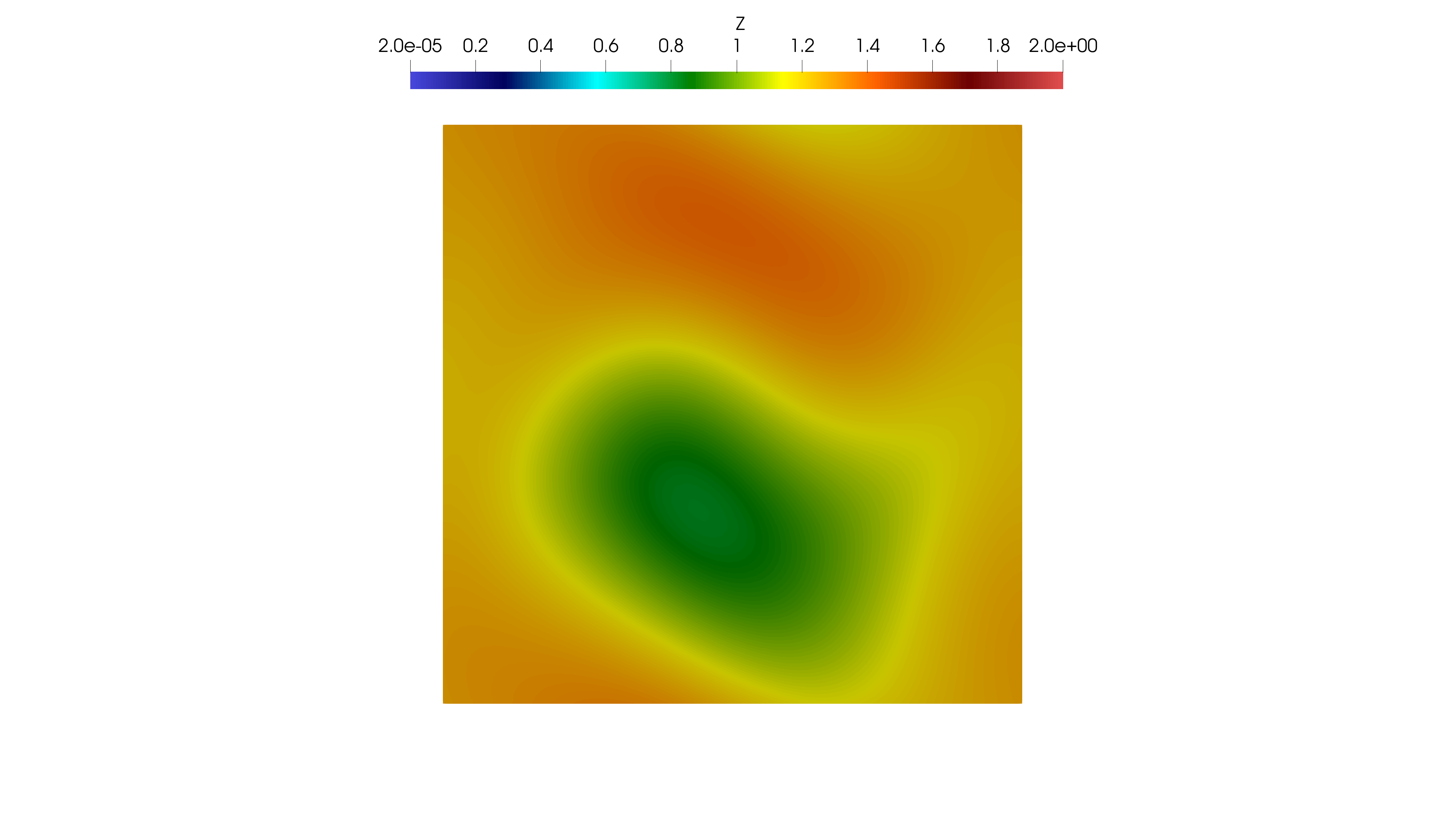} \\
    \multicolumn{4}{c}{\includegraphics[trim={28.0cm 67.5cm 28.0cm 2.5cm},clip,scale=0.15]{ 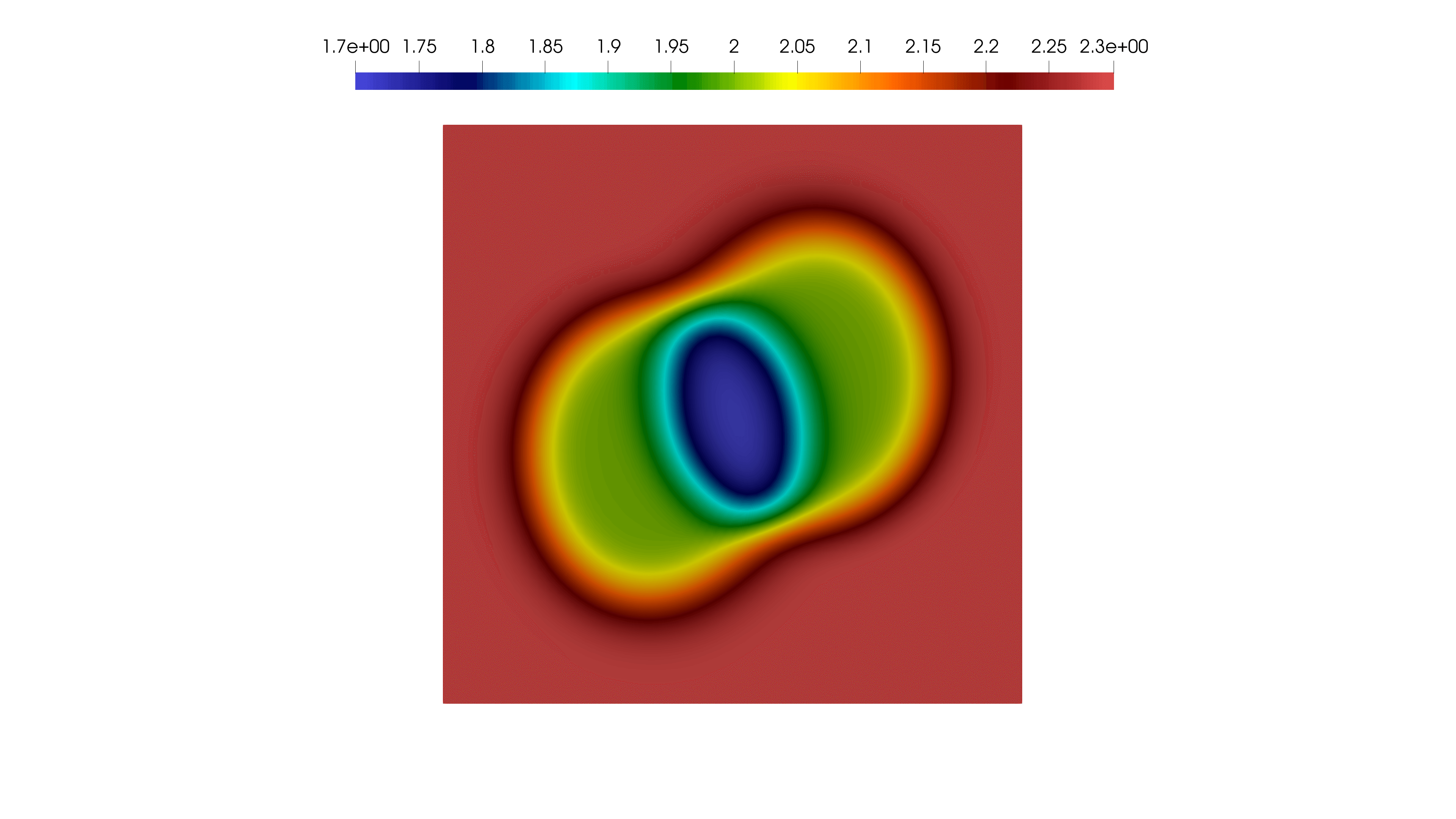}} \\[-0.5em]
     \includegraphics[trim={41cm 8.4cm 41.cm 8.5cm},clip,scale=0.052]{ test2rho_full.0020.png}
    &
    \includegraphics[trim={41cm 8.4cm 41cm 8.5cm},clip,scale=0.052]{ 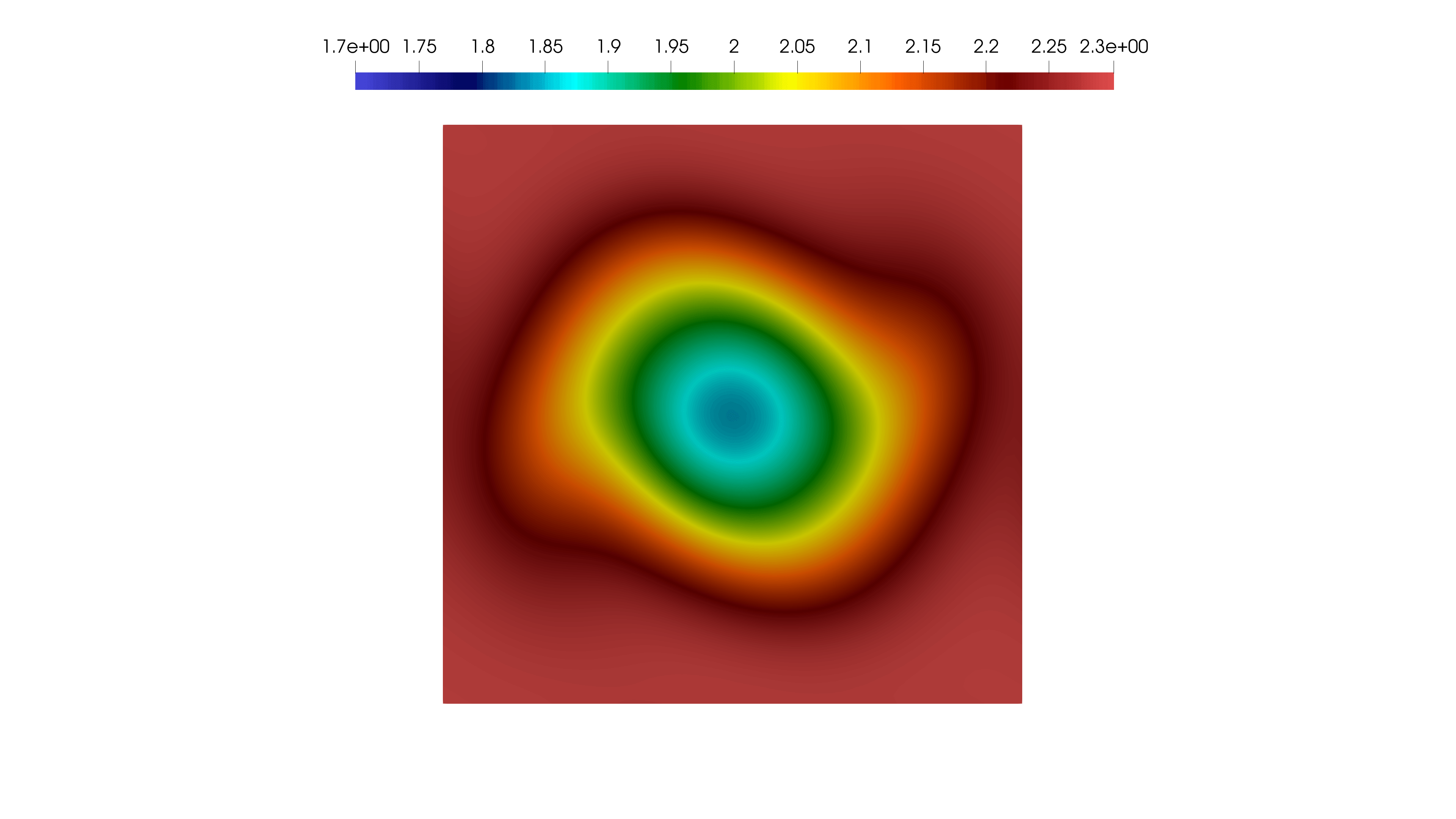}  
    &
    \includegraphics[trim={41cm 8.4cm 41cm 8.5cm},clip,scale=0.052]{ 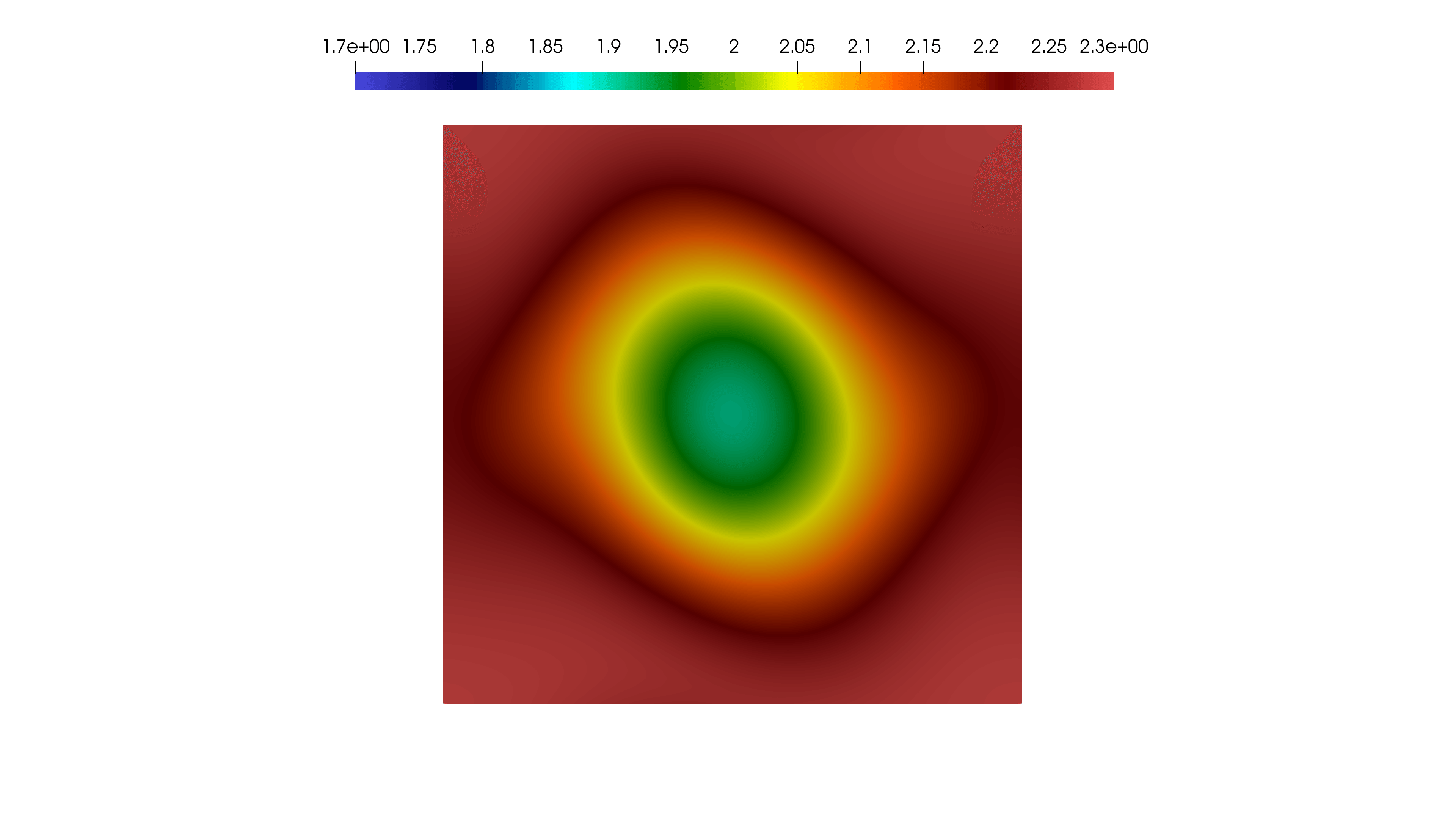} 
    &
    \includegraphics[trim={41cm 8.4cm 41cm 8.5cm},clip,scale=0.052]{ 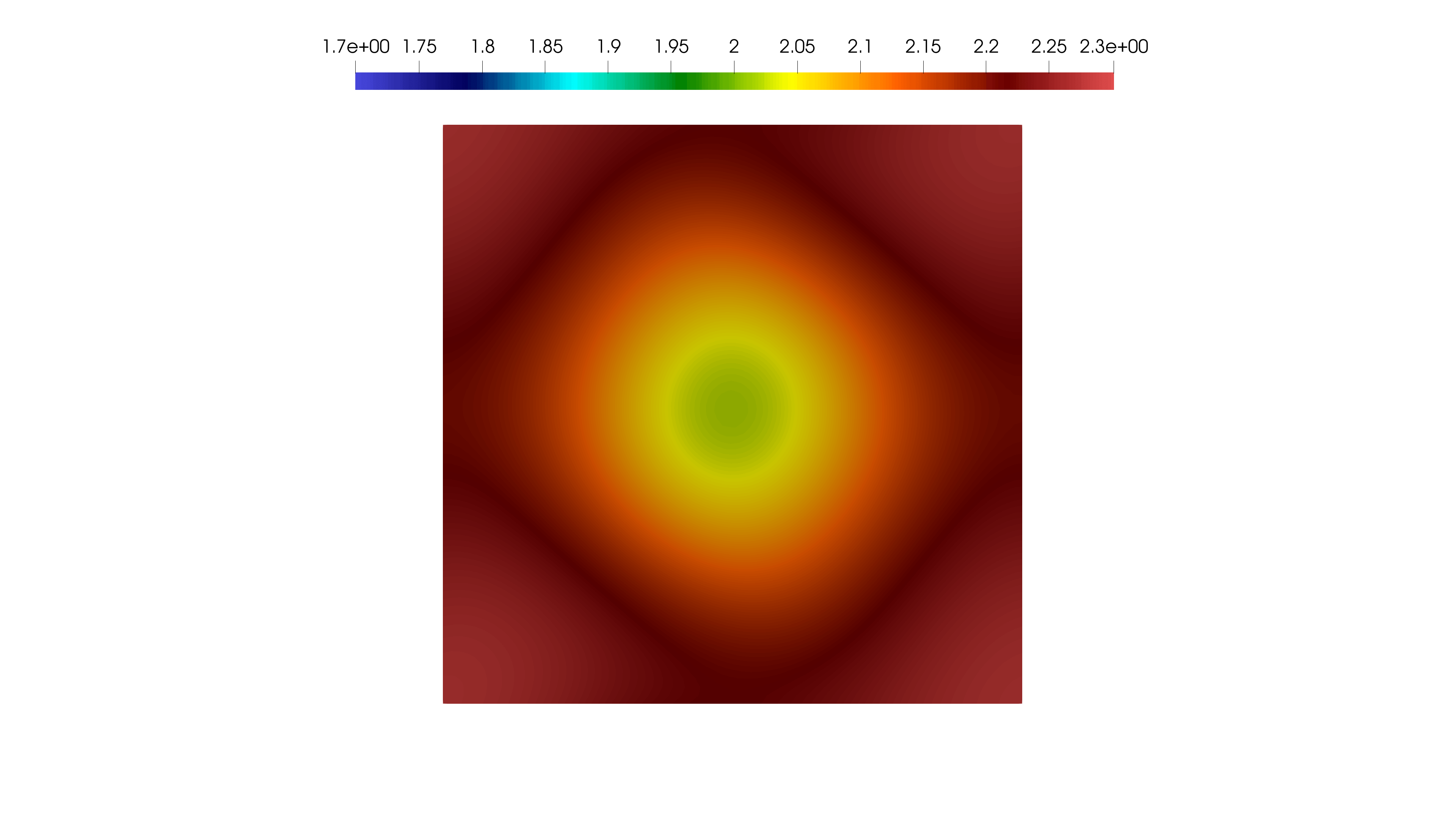}  \\
\end{tabular}
    \caption{Snapshots of the partial mass densities $\rho_{1}$ (first row), $\rho_2$ (second row), $\rho_3$ (third row) and the total mass density $\rho$ (last row) at times $t=0.1,0.6,1,2$. for the experiment in Section \ref{subsec:experiment2}.}\label{fig.test2}
\end{figure}

\begin{figure}[ht]
\centering
\includegraphics[width=100mm]{ 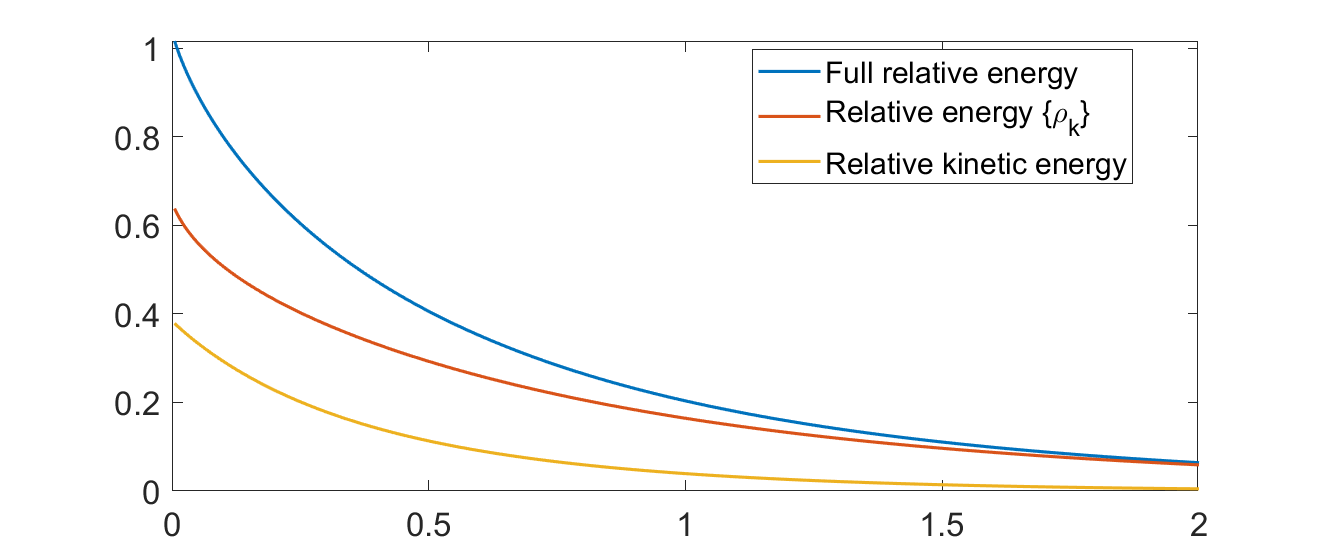}
\caption{Relative energy versus time for the experiment in Section \ref{subsec:experiment2}.}
\label{fig.ener2}
\end{figure}

\section{Conclusion and outlook}

We have derived and analyzed a new structure-preserving finite element scheme for the Navier--Stokes--Maxwell--Stefan equations with a generalized incompressibility constraint. After a suitable variational formulation of the continuous problem, the space discretization is realized by a conforming piecewise linear/quadratic finite element approximation. The time discretization is of mixed explicit--implicit type, allowing us to prove that the quasi-incompressibility constraint is satisfied pointwise in space. The resulting numerical scheme conserves the partial and total masses and satisfies a discrete energy inequality with a nonnegative numerical dissipation term. Convergence tests for a two-component mixture verify the expected convergence orders and structure-preserving properties. The numerical tests for a three-component mixture exhibit some rotational behavior of the mass densities, which is absent in the incompressible model. 

Future work will be concerned with the existence analysis of the numerical scheme and the proof of the positivity of the mass densities. One approach is to use the reformulation of the problem from \cite{Dru21} in terms of the so-called relative chemical potentials. A critical part of the proof is the $L^1(\Omega)$ bound for the pressure, since the proof in \cite[Prop.~7.4]{Dru21} relies on the (continuous) Bogovskiǐ operator. The hope is that the discrete counterpart of this operator can be used; see \cite{KPS18}. The (local-in-time) proof of \cite{BoDr21} does not use the Bogovskiǐ operator and may provide another approach to prove the existence of discrete solutions. The derivation of error estimates, using the relative energy approach, and the design of structure-preserving higher-order approximations similar to \cite{AF25} is future work.


\section*{Data availability statement}

This study did not generate or analyze any datasets. Therefore, data sharing is not applicable to this work.


\section*{Declarations}

The authors declare that they have no conflict of interest.


\end{document}